\documentclass[12pt, reqno]{amsart}
\usepackage{amsmath, amsthm, amscd, amsfonts, amssymb, graphicx, color}
\usepackage[bookmarksnumbered, colorlinks, plainpages]{hyperref}
\hypersetup{colorlinks=true,linkcolor=red, anchorcolor=green,
 citecolor=cyan, urlcolor=red, filecolor=magenta, pdftoolbar=true}

\allowdisplaybreaks

\def\rr{{\mathbb R}}
\def\rn{{\mathbb{R}^n}}

\def\nrr{{\mathbb R}^{n-1}}
\def\zz{{\mathbb Z}}
\def\cc{{\mathbb C}}

\def\nn{{\mathbb N}}

\def\cl{{\mathcal L}}
\def\cm{{\mathcal M}}

\def\cp{{\mathcal P}}
\def\cq{{\mathcal Q}}

\def\cs{{\mathcal S}}

\def\fz{\infty }

\def\lz{\lambda}

\def\vz{\varphi}
\def\vez{\varepsilon}

\def\lf{\left}
\def\r{\right}

\def\la{\langle}
\def\ra{\rangle}
\def\hs{\hspace{0.25cm}}
\def\ls{\lesssim}
\def\gs{\gtrsim}

\def\ov{\overline}
\def\noz{\nonumber}
\def\wz{\widetilde}
\def\wh{\widehat}

\def\gfz{\genfrac{}{}{0pt}{}}

\def\loc{{\mathop\mathrm{\,loc\,}}}
\def\supp{\mathop\mathrm{\,supp\,}}

\def\vlp{{L^{p(\cdot)}(\rn)}}

\def\tlve{f_{p(\cdot),q(\cdot)}^{s(\cdot),\phi}(\rn)}
\def\btlve{F_{p(\cdot),q(\cdot)}^{s(\cdot),\phi}(\rn)}

\newtheorem{theorem}{Theorem}[section]
\newtheorem{lemma}[theorem]{Lemma}
\newtheorem{proposition}[theorem]{Proposition}
\newtheorem{corollary}[theorem]{Corollary}
\theoremstyle{definition}
\newtheorem{definition}[theorem]{Definition}

\theoremstyle{remark}
\newtheorem{remark}[theorem]{Remark}
\numberwithin{equation}{section}

\begin{document}

\arraycolsep=1pt

\setcounter{page}{1}

\title[Triebel-Lizorkin-Type Spaces]{{\vspace{-3cm}\small\hfill\bf Banach J. Math. Anal. (to appear)}\\
\vspace{2cm}Triebel-Lizorkin-Type
Spaces with Variable Exponents}

\author[D. Yang, C. Zhuo, W. Yuan]{Dachun Yang$^1$,
Ciqiang Zhuo$^1$ and Wen Yuan$^1$$^{*}$}

\address{$^{1}$ School of Mathematical Sciences, Beijing Normal University,
Laboratory of Mathematics and Complex Systems, Ministry of
Education, Beijing 100875, People's Republic of China.}
\email{\textcolor[rgb]{0.00,0.00,0.84}{dcyang@bnu.edu.cn}}
\email{\textcolor[rgb]{0.00,0.00,0.84}{cqzhuo@mail.bnu.edu.cn}}
\email{\textcolor[rgb]{0.00,0.00,0.84}{wenyuan@bnu.edu.cn}}




\date{Received: Sep. 21, 2014; Accepted: Jan. 4, 2015.
\newline \indent $^{*}$ Corresponding author}

\begin{abstract}
In this article, the authors first introduce the Triebel-Lizorkin-type
space $F_{p(\cdot),q(\cdot)}^{s(\cdot),\phi}(\mathbb R^n)$
with variable exponents, and establish its $\varphi$-transform characterization
in the sense of Frazier and Jawerth, which further
implies that this new scale of function spaces is well defined.
The smooth molecular and the smooth atomic characterizations of
$F_{p(\cdot),q(\cdot)}^{s(\cdot),\phi}(\mathbb R^n)$ are also obtained,
which are used to prove a trace theorem of
$F_{p(\cdot),q(\cdot)}^{s(\cdot),\phi}(\mathbb R^n)$.
The authors also characterize the space
$F_{p(\cdot),q(\cdot)}^{s(\cdot),\phi}(\mathbb R^n)$
via Peetre maximal functions.
\end{abstract}

\maketitle

\section{Introduction\label{s1}}

Between 1960's and 1970's,
the Besov space $B_{p,q}^s(\rn)$ and the Triebel-Lizorkin
space $F_{p,q}^s(\rn)$
were introduced and investigated accompanying with the development of
the theory of function spaces (see, for example, \cite{t83}).
These spaces form a very general unifying scale of many
well-known classical concrete function spaces such as
Lebesgue spaces, H\"older-Zygmund spaces, Sobolev spaces,
Bessel-potential spaces, Hardy spaces
and BMO, which have their own history.
A comprehensive treatment of these
function spaces and their history can be founded in
Triebel's monographes \cite{t83,t92,t95,t06}.
Recently, to clarify the relations among Besov spaces,
Triebel-Lizorkin spaces and $Q$ spaces (see \cite{dx04,ejpx}),
Besov-type spaces $B_{p,q}^{s,\tau}(\rn)$ and Triebel-Lizorkin-type spaces
$F_{p,q}^{s,\tau}(\rn)$ and their homogeneous counterparts
for all admissible parameters were introduced and studied in
\cite{yyjfa,yymz10,ysiy}. Moreover, the Besov-type and the Triebel-Lizorkin-type
spaces, including some of their special cases related to $Q$ spaces,
have been used to study the existence and the regularity of
solutions of some partial differential equations such as (fractional)
Navier-Stokes equations; see, for example,
\cite{lxy,lzjmaa,lzjfa,t13,ysy13}.
For more properties of these spaces, we refer the reader to
\cite{sickel1,sickel2,yyaa13,yyz13}.

On the other hand, in recent years, there has been a growing interesting
in generalizing classical spaces such as Lebesgue and Sobolev spaces to
cases with either variable integrability or variable smoothness
(see \cite{cfbook,dhr11}),
which are obviously not covered by any function space with invariable exponents.
Spaces of variable integrability can be
traced back to Birnbaum-Orlicz \cite{bo31}, Orlicz \cite{ol32}
and Nakano \cite{nak50,nak51}. In particular, the definition of so-called
Musielak-Orlicz spaces was clearly written by Nakano in \cite[Section 89]{nak50},
while it seems that Orlicz was mainly interested in the completeness of function
spaces. But the modern development was started with
the article \cite{kr91} of Kov\'a\v{c}ik and R\'akosn\'{\i}k in 1991
and widely used in the study of harmonic analysis as well as partial
differential equations; see, for example,
\cite{ai00,cdh11,cfbook,cfmp06,din04,dhhm09,dhr11,drm03,io02,mw08}.
The motivation to study such function spaces also
comes from applications to fluid dynamic, image processing
and the calculus of variation; see, for example,
\cite{am02,aemg02,am05,clr06,drm03,fan07,rm00,rm04}.

To complete the theory of the variable exponent Lebesgue and Sobolev spaces,
Almeida and Samko \cite{as06} and Gurka et al. \cite{ghn07}
introduced and investigated variable exponent Bessel potential
spaces $\cl^{\alpha,p(\cdot)}$ with
variable integrability index $p(\cdot)$.
 Later, Xu \cite{Xu081,Xu082,Xu09}
studied Besov spaces $B_{p(\cdot),q}^{s}(\rn)$ and Triebel-Lizorkin spaces
$F_{p(\cdot),q}^{s}(\rn)$ with the variable exponent $p(\cdot)$
but invariable exponents $q$ and $s$. Along a different line of study, when Leopold
\cite{leop891,leop892,leop91,leop99} and Leopold and Schrohe \cite{leops96}
studied pseudo-differential operators with symbols of the type
$$(1+|\xi|^2)^{s(x)/2},$$
they defined and investigated related Besov spaces
with variable smoothness, $B_{p,p}^{s(\cdot)}(\rn)$. Function spaces
of variable smoothness including Besov space $B_{p,q}^{s(\cdot)}(\rn)$
and Triebel-Lizorkin space $F_{p,q}^{s(\cdot)}(\rn)$ have been studied by Besov
\cite{besov99,besov03,besov05}, which was a generalization of Leopold's work.
Another interesting research direction of function spaces
with variable integrability is the theory of Hardy spaces $H^{p(\cdot)}(\rn)$
with variable exponents as well as local Hardy spaces $h^{p(\cdot)}(\rn)$,
which was introduced and investigated by Nakai and Sawano \cite{ns12} and
they proved that
$$h^{p(\cdot)}(\rn)=F_{p(\cdot),2}^{0}(\rn).$$
Independently, Cruz-Uribe and
Wang in \cite{cw13} also investigated the variable exponent Hardy space
with some weaker conditions than those used in \cite{ns12}.

As we can see from the trace and the embedding theorems
of the classical function spaces,
the smoothness and the integrability often interact each other; see, for example,
\cite[Theorem 6.8 and Corollary 2.2]{ysiy}.
As was pointed out in \cite[p.\,1629]{ah10} and \cite[p.\,1733]{dhr09},
the unifications of the trace and the Sobolev embedding do not occur on function
spaces with only one index variable. For example, the trace space of
the variable exponent Sobolev space $W^{k,p(\cdot)}$ is no longer a space of
the same type (see \cite{dhr11}), since they involve an interaction
between integrability
and smoothness. As one of motivations, to tackle this problem,
Alexandre and H\"ast\"o \cite{dhr09}
introduced and investigated
Triebel-Lizorkin spaces with variable smoothness and integrability
$F_{p(\cdot),q(\cdot)}^{s(\cdot)}(\rn)$ with $s(\cdot)\ge0$,
and showed that these spaces behaved nicely with
respect to the trace operator. Subsequently, Vyb\'iral \cite{vj09} established
Sobolev and Jawerth embeddings of these spaces and,
moreover, Kempka \cite{kempka09} characterized
$F_{p(\cdot),q(\cdot)}^{s(\cdot)}(\rn)$ by local means, and Kempka and
Vyb\'iral \cite{kv12} obtained the equivalent characterization via
ball means of differences.
The main difficulty of studying $F_{p(\cdot),q(\cdot)}^{s(\cdot)}(\rn)$
is the absence of the vector-valued inequality
for the boundedness on
$L^{p(\cdot)}(\ell^{q(\cdot)}(\rn))$ of the Hardy-Littlewood maximal function,
which, in the classical case with $p,\ q, s$ being constant exponents, is a very
important tool in studying the space $F_{p,q}^{s}(\rn)$.
The vector-valued convolution inequality developed in \cite[Theorem 3.2]{dhr09}
(see also Lemma \ref{l-estimate2} below)
supplies a well remedy for this absence.

Vyb\'iral \cite{vj09} and Kempka \cite{kempka09} also studied
the Besov space $B_{p(\cdot),q}^{s(\cdot)}(\rn)$ with the only index $q$ being
a constant, which is a quite natural case, since the norm in the Besov space
 is usually defined via the iterated space
$\ell^q(L^p(\rn))$. Furthermore, Almeida and H\"ast\"o \cite{ah10} introduced
and investigated the Besov space $B_{p(\cdot),q(\cdot)}^{s(\cdot)}(\rn)$
with all three variable exponents,
which makes a further step in completing the unification process of function
spaces with variable smoothness and integrability.
The atomic characterization of $B_{p(\cdot),q(\cdot)}^{s(\cdot)}(\rn)$
was established by Drihem \cite{dd12} and some
equivalent characterizations via local means and ball means of differences
were also obtained by Kempka and Vyb\'ral \cite{kv12}.
Moreover, Noi and Sawano \cite{nois12} investigated
the complex interpolation of Besov spaces and Triebel-Lizorkin
spaces with variable exponents (see also \cite{yyz13} for
the complex interpolation of Besov-type spaces and Triebel-Lizorkin-type
spaces but with invariable exponents) and, in \cite{noi14}, Noi studied the trace and the extension operators
for Beosv spaces and Triebel-Lizorkin spaces with variable exponents.
Very recently, Izuki et al. \cite{ins14} gave out an elementary
introduction to function spaces with variable exponents and
a survey of related function spaces.

More generally, Kempka \cite{kempka10} introduced and
studied 2-microlocal Besov and Triebel-Lizorkin
spaces with variable integrability and gave out characterizations
by decompositions in
atoms, molecules and wavelets, which cover the usual Besov and
Triebel-Lizorkin spaces
as well as spaces of variable smoothness and integrability and also include
the space $F_{p(\cdot),q(\cdot)}^{s(\cdot)}(\rn)$ without the restriction
$s(\cdot)\ge0$. The trace of 2-microlocal Besov and Triebel-Lizorkin
spaces with variable exponents was studied by Moura et al. \cite{mns13},
as well as Gon\c{c}alves et al. \cite{gmn14}.
Moreover, Ho \cite{ho12}
investigated the variable Triebel-Lizorkin-Morrey space, which is
an extension of Triebel-Lizorkin-Morrey spaces in \cite{saw08,{st07}}
and also generalizes the function
space $F_{p(\cdot),q(\cdot)}^{s(\cdot)}(\rn)$ in \cite{dhr09}.

Here, we should point out that, different from the classical case with
exponents being constants, the definition of
$B_{p(\cdot),q(\cdot)}^{s(\cdot)}(\rn)$
is more complicated than that of $F_{p(\cdot),q(\cdot)}^{s(\cdot)}(\rn)$.
The main reason is that the mixed Lebesgue-sequence space
$$\ell^{q(\cdot)}(L^{p(\cdot)}(\rn))$$ (see \cite[Definition 3.1]{ah10})
involved in the definition of $B_{p(\cdot),q(\cdot)}^{s(\cdot)}(\rn)$,
 as was pointed out by Almeida and H\"ast\"o
in \cite[Remark 4.2]{ah10}, does not enjoy one key feature of iterated
function spaces, namely, inheritance of properties from constituent spaces.
To limit the length of this article,
we leave the study of Besov-type spaces with all variable exponents
in a furthercoming article.

The purpose of this article is to introduce and study a more generalized scale of
function spaces, based on the Triebel-Lizorkin-type space
$F_{p,q}^{s,\tau}(\rn)$, with variable exponent of smoothness,
$s(\cdot)$, variable
exponents of integrability, $p(\cdot)$ and $q(\cdot)$, and a set function $\phi$,
denoted by $\btlve$. These spaces generalize classical Triebel-Lizorkin-type
spaces and Triebel-Lizorkin spaces with variable smoothness and integrability.
Molecular and atomic characterizations, Peetre maximal function characterizations
of these spaces are also established in this article. As applications,
we show a trace theorem of Triebel-Lizorkin-type spaces with variable exponents
and give out some equivalent quasi-norms under some restrictions of the
set function $\phi$.

We begin with some basic notation.
In what follows, for a measurable function
$p(\cdot):\ \rn\to(0,\fz)$ and a measurable set $E$ of $\rn$, let
$$p_-(E):=\mathop\mathrm{ess\,inf}_{x\in E}p(x)\quad \mathrm{and}\quad
p_+(E):=\mathop\mathrm{ess\,sup}_{x\in E}p(x).$$
For notational simplicity, we let $p_-:=p_-(\rn)$ and $p_+:=p_+(\rn)$.
Denote by $\cp(\rn)$ the \emph{collection of all measurable functions
$p(\cdot):\ \rn\to(0,\fz)$ satisfying $0<p_-\le p_+<\fz$}.

For $p(\cdot)\in\cp(\rn)$ and a measurable set $E\subset\rn$,
the \emph{space} $L^{p(\cdot)}(E)$
is defined to be the set of all measurable functions $f$ such that
$$\|f\|_{L^{p(\cdot)}(E)}:=\inf\lf\{\lz\in(0,\fz):\ \int_E\lf[\frac{|f(x)|}
{\lz}\r]^{p(x)}\,dx\le1\r\}<\fz.$$
For $r\in(0,\fz)$, denote by $L_{\rm loc}^r(\rn)$ the \emph{set} of all
$r$-locally integrable functions on $\rn$.
Denote by $L^\fz(\rn)$ the \emph{set} of all measurable functions $f$ such that
$$\|f\|_{L^\fz(\rn)}:=\mathop{\rm ess\,sup}\limits_{y\in \rn}|f(y)|<\fz.$$
\begin{remark}\label{r-vlpp}
Let $p(\cdot)\in\cp(\rn)$.

(i) It was presented in \cite[p.\,3671]{ns12}
(see also \cite[Theorem 2.17]{cfbook}) that,
for all $\lz\in\cc$,
$$\|\lz f\|_{\vlp}=|\lz|\|f\|_{\vlp}$$
and, if $r\in(0,\min\{p_-,1\}]$, then, for all $f,\ g\in\vlp$,
$$\|f+g\|_{\vlp}^r\le \|f\|_{\vlp}^r+\|g\|_{\vlp}^r.$$

(ii) If $$\int_\rn\lf[\frac{|f(x)|}{\delta}\r]^{p(x)}\,dx\le \wz C$$
for some $\delta\in(0,\fz)$ and some positive constant $\wz C$
independent of $\delta$,
 then it is easy to see that $\|f\|_{\vlp}\le C\delta$, where $C$ is
 a positive constant
independent of $\delta$, but depending on $p_-$ (or $p_+$) and $\wz C$.

(iii) Let $p(\cdot)\in\cp(\rn)$ satisfy $1<p_-\le p_+<\fz$.
Define the conjugate
exponent $\wz{P}(\cdot)$ of $p(\cdot)$ by setting,
for all $x\in\rn$, $\wz{P}(x):=\frac{p(x)}{p(x)-1}$.
 It was proved, in \cite[Theorem 2.6]{cfbook}, that,
if $f\in \vlp$ and $g\in L^{\wz{P}(\cdot)}(\rn)$, then $fg\in L^1(\rn)$
and
$$\int_\rn|f(x)g(x)|\,dx\le C\|f\|_{\vlp}\|g\|_{L^{\wz{P}(\cdot)}(\rn)},$$
where $C$ is a positive constant depending on $p_-$ or $p_+$,
 but independent of $f$
and $g$.

{\rm (iv)} Obviously, the space $\vlp$ has the \emph{lattice} property, namely,
if $|f|\le|g|$, then
$$\|f\|_{\vlp}\le\|g\|_{\vlp}.$$
\end{remark}
Recall that a measurable function $g\in\cp(\rn)$ is said to satisfy the
\emph{locally {\rm log}-H\"older continuous condition},
denoted by $g\in C_{\rm loc}^{\log}(\rn)$,
if there exists a positive constant $C_{\log}(g)$ such that,
for all $x,\ y\in\rn$,
\begin{equation}\label{ve1}
|g(x)-g(y)|\le \frac{C_{\log}(g)}{\log(e+1/|x-y|)},
\end{equation}
and $g$ is said to satisfy the
\emph{globally {\rm log}-H\"older continuous condition},
denoted by $g\in C^{\log}(\rn)$,
if $g\in C_{\rm loc}^{\log}(\rn)$ and there exist positive constants
$C_\fz$ and $g_\fz$
such that, for all $x\in\rn$,
\begin{equation*}
|g(x)-g_\fz|\le \frac{C_\fz}{\log(e+|x|)}.
\end{equation*}
\begin{remark}\label{r-log}
(i) Let $g\in C^{\log}(\rn)$. Then $g_\fz=\lim_{|x|\to\fz}p(x)$.

(ii) Let $g\in\cp(\rn)$. Then $g\in C^{\log}(\rn)$ if
and only if $1/g\in C^{\log}(\rn)$.
\end{remark}

For all $x\in\rn$ and $r\in(0,\fz)$, denoted by $Q(x,r)$ the cube centered at $x$
with side-length $r$, whose sides parallel axes of coordinates.
Let $\phi:\ \rn\times[0,\fz)\to(0,\fz)$ be a measurable function.
In this article, we always suppose that $\phi$ satisfies the
following two conditions:

\begin{enumerate}
\item[(\textbf{S1})]  there exist positive constants $c_1$ and $\wz c_1$ such that,
for all $x\in\rn$
and $r\in(0,\fz)$,
$$(\wz c_1)^{-1}\le\frac{\phi(x,r)}{\phi(x,2r)}\le c_1;$$

\item[(\textbf{S2})]  there exists a positive constant $c_2$ such that,
for all $x,\,y\in\rn$
and $r\in(0,\fz)$ with $|x-y|\le r$,
$$(c_2)^{-1}\le\frac{\phi(x,r)}{\phi(y,r)}\le c_2.$$
\end{enumerate}

In what follows, for all cubes $Q:=Q(x,r)$ with $x\in\rn$ and $r\in(0,\fz)$,
let $$\phi(Q):=\phi(Q(x,r)):=\phi(x,r).$$
\begin{remark}\label{r-phi}
(i) We point out that the conditions (\textbf{S1}) and (\textbf{S2}) of $\phi$
are, respectively, called \emph{doubling condition} and \emph{compatibility condition}, which have been used by Nakai \cite{nakai93,nakai06}
and Nakai and Sawano \cite{ns12} when
they studied generalized Campanato spaces.

(ii) Let $\phi(Q):=|Q|^\tau$ with
$\tau\in[0,\fz)$ for all cubes $Q$. Then, obviously, $\phi$ satisfies
the conditions (\textbf{S1}) and (\textbf{S2}).

(iii) Let $p(\cdot)\in C^{\log}(\rn)$.
The set function $\phi$, defined by setting, for all cubes $Q$,
$$\phi(Q):=\frac{\|\chi_Q\|_{\vlp}}{|Q|},$$
which is just \cite[Example 6.4]{ns12}, satisfies the conditions
(\textbf{S1}) and (\textbf{S2}).

(iv) Let $\phi$ be a \emph{nondecreasing} set function, namely,
there exists a positive constant $C$ such that, for all cubes $Q_1\subset Q_2$,
$\phi(Q_1)\le C\phi(Q_2)$. If $\phi$ satisfies the
condition (\textbf{S1}), then $\phi$ also satisfies the
condition (\textbf{S2}). Indeed, for all $x,\ y\in\rn$ and $r\in(0,\fz)$
with $|x-y|\le r$, it is easy to see that $Q(x,r)\subset Q(y,2r)$
and $Q(y,r)\subset Q(x,2r)$ and, by the condition
(\textbf{S1}), we see that
$$1\ls\frac{\phi(x,r)}{\phi(x,2r)}\ls
\frac{\phi(x,r)}{\phi(y,r)}\ls \frac{\phi(y,2r)}{\phi(y,r)}\ls1$$
with the implicit positive constants independent of $x,\, y$ and $r$.
Thus, $\phi$ satisfies the condition (\textbf{S2}).

(v) Let $\phi(Q):=\int_Qw(x)\,dx$ for all cubes $Q$, where $w$ is a classical
Muckenhoupt $A_p(\rn)$-weight with $p\in[1,\fz]$. It is well known that
each Muckenhoupt $A_p(\rn)$-weight is doubling, thus, by (iv), we conclude that
 $\phi$ satisfies the conditions (\textbf{S1}) and (\textbf{S2}).
For the definition and properties of Muckenhoupt $A_p(\rn)$-weights,
we refer the reader to
\cite{stein93}.
\end{remark}
Let $\cs(\rn)$ be the \emph{space of all Schwartz functions} on
$\rn$ and $\cs'(\rn)$
its \emph{topological dual space}.
We say a pair $(\vz, \Phi)$ of functions to be
\emph{admissible} if $\vz,\ \Phi\in\cs(\rn)$
satisfy
\begin{equation}\label{x.1}
 {\rm supp}\, \wh \vz\subset\lf\{\xi\in\rn:\ \frac12\le|\xi|\le2\r\}\
 {\rm and}\ |\wh \vz(\xi)|\ge
c>0\ {\rm when}\ \frac 35\le|\xi|\le\frac53
\end{equation}
and
\begin{equation}\label{x.2}
{\rm supp}\, \wh \Phi\subset\{\xi\in\rn:\ |\xi|\le2\}\ {\rm and}\
|\wh\Phi(\xi)|\ge c>0\ {\rm when}\
|\xi|\le\frac53,
\end{equation}
where $\wh f(\xi):=\int_{\rn}f(x)e^{-ix\cdot\xi}\,dx$ for all $\xi\in\rn$
and $f\in L^1(\rn)$, and $c$ is a positive constant independent of $\xi\in\rn$.
Throughout the article, for all $\vz\in\cs(\rn)$,
$j\in\nn:=\{1,2,\dots\}$ and $x\in\rn$, we put
$\vz_j(x):=2^{jn}\vz(2^jx)$ and $\wz \vz(x):=\overline{\vz(-x)}$.

For $j\in\zz$ and $k\in\zz^n$,
denote by $Q_{jk}$ the \emph{dyadic cube} $2^{-j}([0,1)^n+k)$,
$x_{Q_{jk}}:=2^{-j}k$
its \emph{lower left corner} and $\ell(Q_{jk})$ its \emph{side length}.
Let
$$\cq:=\{Q_{jk}:\ j\in\zz,\ k\in\zz^n\},\ \cq^\ast:=\{Q\in\cq:\ \ell(Q)\le1\}$$
 and $j_Q:=-\log_2\ell(Q)$ for all $Q\in\cq$.

Now we introduce Triebel-Lizorkin-type spaces with variable exponents.
\begin{definition}\label{d-tl}
Let $(\vz,\Phi)$ be a pair of admissible functions on $\rn$.
Let $p,\, q\in \cp(\rn)$ satisfy
$$0<p_-\le p_+<\fz,\ 0<q_-\le q_+<\fz$$
 and
$\frac1p$, $\frac1q\in C^{\log}(\rn)$,
$s\in C_{\loc}^{\log}(\rn)\cap L^\fz(\rn)$ and
$\phi$ be a set function satisfying the conditions
(\textbf{S1}) and (\textbf{S2}). Then the \emph{Triebel-Lizorkin-type
space with variable exponents},
$F_{p(\cdot),q(\cdot)}^{s(\cdot),\phi}(\rn)$, is defined to be the set of all
$f\in\cs'(\rn)$ such that
$$\|f\|_{F_{p(\cdot),q(\cdot)}^{s(\cdot),\phi}(\rn)}
:=\sup_{P\in\cq}\frac1{\phi(P)}
\lf\|\lf\{\sum_{j=\max\{j_P,0\}}^\fz\lf[2^{js(\cdot)}|\vz_j\ast f(\cdot)|\r]
^{q(\cdot)}\r\}^{\frac1{q(\cdot)}}\r\|_{L^{p(\cdot)}(P)}<\fz,$$
where, when $j=0$,
$\vz_0$ is replaced by $\Phi$, and the supremum is taken over all dyadic
cubes $P$ in $\rn$.
\end{definition}
\begin{remark}\label{r-defi}
Let $p(\cdot),\ q(\cdot),\ s(\cdot)$ be as in Definition \ref{d-tl}.

(i) When $\phi(Q):=1$ for all cubes $Q$, then
$$\btlve=F_{p(\cdot),q(\cdot)}^{s(\cdot)}(\rn),$$
where $F_{p(\cdot),q(\cdot)}^{s(\cdot)}(\rn)$ denotes the
\emph{Triebel-Lizorkin space
with variable smoothness and integrability}
which is introduced and investigated in \cite{dhr09}.
We point out that Diening et al. \cite{dhr09} studied the space
$F_{p(\cdot),q(\cdot)}^{s(\cdot)}(\rn)$
 under an additional assumption that $s$ is nonnegative,
 which is generalized to the case that $s:\ \rn\to\rr$ and
 $s\in C_{\loc}^{\log}(\rn)\cap L^\fz(\rn)$ by Kempka in \cite{kempka10}.

(ii) When $p,\, q,\, s$ are constant exponents and $\phi(Q):=|Q|^{\tau}$ with
$\tau\in[0,\fz)$ for all cubes $Q$, then
$$\btlve=F_{p,q}^{s,\tau}(\rn),$$
 where
$F_{p,q}^{s,\tau}(\rn)$ denotes the \emph{Triebel-Lizorkin-type space}
which was introduced and studied in \cite{ysiy}.

(iii) When $q,\, s$ are constant exponents and $\phi(Q):=|Q|^{\tau}$ with
$\tau\in[0,\fz)$ for all cubes $Q$, then
$$\btlve=F_{p(\cdot),q}^{s,\tau}(\rn),$$
which was investigated in \cite{lsuyy1}.

(iv) The condition, $0<p_-\le p_+<\fz$, is quite natural, since
there also exists the restriction $p<\fz$ in the case of constant exponents.
The assumption, $0<q_-\le q_+<\fz$, is different from
the case of constant exponents where $q=\fz$ is included.
This restriction comes from
the application of the convolution inequality in \cite[Theorem 3.2]{dhr09}
(see also Lemma \ref{l-estimate2} below), when proving that the space
$\btlve$ is independent of the choice of admissible function pairs $(\vz,\Phi)$.
Observe that, even when $\phi(Q)=1$ for all cubes $Q$,
this restriction is necessary;
see \cite[p.\,857]{kv12}.
\end{remark}

This article is organized as follows.

Section \ref{s2} is devoted to showing that the space $\btlve$
is independent of the choice of admissible function pairs $(\vz,\Phi)$, which
is a consequence of the $\vz$-transform characterization
of $\btlve$ in the sense of Frazier and Jawerth
(see Corollary \ref{c-inde} below).
Different from the method used in the case of constant exponents,
in the proof of the boundedness of the $\vz$-transform
$S_\vz$ from $\btlve$ to $\tlve$ (the sequence space corresponding to
the function space $\btlve$), we make full use of the
so-called \emph{$r$-trick lemma}
(namely, \cite[Lemma A.6]{dhr09}) and the vector-valued convolution
inequality (namely, \cite[Theorem 3.2]{dhr09}; see also
Lemma \ref{l-estimate2} below).
We point out that the vector-valued convolution inequality
also plays an essential
role throughout the remainder of this article.

In Section \ref{s3}, we establish equivalent characterizations of
$\btlve$ in term of
molecules, atoms (see Theorem \ref{t-md} below)
or Peetre maximal functions
(see Theorem \ref{t-pmfc} below). To prove Theorem \ref{t-md}, we
borrow some ideas from the proof of \cite[Theorem 3.12]{d13}
which gives the atomic
characterization of the Besov-type and the Triebel-Lizorkin-type spaces, and
the proof of \cite[Theorem 3.13]{kempka10} which gives
the molecular characterization
of 2-microlocal Besov and Triebel-Lizorkin spaces with variable integrability.
The Sobolev embedding (see Proposition \ref{p-se1} below)
plays a key role in the proof of Theorem \ref{t-md},
which may be of independent interest.
The proof of Theorem \ref{t-pmfc} is similar to that of \cite[Theorem 3.2]{lsuyy}
(see also \cite[Theorem 2.6]{ut12}) and strongly depends on the vector-valued
convolution inequality on $L^{p(\cdot)}(\ell^{q(\cdot)}(\rn))$;
see Lemma \ref{l-estimate2} below. As applications of Theorem \ref{t-pmfc},
some equivalent norms of $\btlve$ are obtained (see Theorem \ref{t-sum} below),
which are further used to show that the spaces $\btlve$ include
the Morrey space with variable exponents
$\cm_\phi^{p(\cdot)}(\rn)$ as a special case; see Proposition \ref{p-lpc} below.

At the end of Section \ref{s3}, via some examples, we show that, in general,
the scales of Triebel-Lizorkin-type spaces with variable exponents
 and variable Triebel-Lizorkin-Morrey
spaces (see \cite{ho12}) do not cover each other
(see Remark \ref{r-compare} below).
Here we point out that Triebel-Lizorkin-type spaces with variable exponents
in this article cover the Triebel-Lizorkin-type spaces
$F_{p,q}^{s,\tau}(\rn)$ for all $\tau\in[0,\fz)$, but the variable
Triebel-Lizorkin-Morrey space in \cite{ho12} does only cover
the Triebel-Lizorkin-type space $F_{p,q}^{s,\tau}(\rn)$ for $\tau\in[0,1/p)$.

In Section \ref{s4}, as an application of the atomic characterization of
$\btlve$, we mainly establish a trace theorem of Triebel-Lizorkin-type
spaces with variable exponents (see Theorem \ref{t-trace1} below).
In the case that $\phi$ is as in Remark \ref{r-defi}(i),
the corresponding result of Theorem \ref{t-trace1}
was obtained in \cite[Theorem 3.13]{dhr09} with
a certain weaker condition (see Remark \ref{r-trace-x}(ii) below),
however, the convergence of the trace of
$f\in F_{p(\cdot),q(\cdot)}^{s(\cdot)}(\rn)$ was not
given out exactly in \cite[p.\,1760]{dhr09}.
In Section \ref{s4}, we first show that the trace operator
is well defined on the space $\btlve$ (see Lemma \ref{l-welld} below),
with a certain restriction on $p$ and $s$, by an argument similar to that used in the
proof of Theorem \ref{t-md}(i). Indeed, in Lemma \ref{l-welld} below, we
prove that the trace of $f\in\btlve$ converges in $\cs'(\rr^{n-1})$.
Then, similar to the proof of \cite[Theorem 3.13]{dhr09},
we show that the trace space of $\btlve$ is independent
of the $n$-th coordinate of
variable exponents $p(\cdot)$ and $s(\cdot)$, and complete the proof
by an argument similar to that used in the proof of \cite[Theorem 6.8]{ysiy}.

Finally, we make some conventions on notation.
Throughout this article, we denote by
$C$ a \emph{positive constant} which is independent of the main
parameters, but may vary from line to line. The \emph{symbols}
$A\ls B$ means $A\le CB$. If $A\ls B$ and $B\ls A$, then we write $A\sim B$.
For all $a,\,b\in\rr$, let
$$a\vee b:=\max\{a,\,b\}.$$
If $E$ is a subset of $\rn$, we denote by $\chi_E$ its
 \emph{characteristic function}.
For all cubes $Q$, we use $c_Q$ to denote its the center.
For all $k:=(k_1,\dots,k_n)\in\zz^n$, let
$$|k|:=|k_1|+\cdots+|k_n|.$$
Let $\nn:=\{1,2,\dots\}$ and $\zz_+:=\{0\}\cup\nn$.

\section{The $\varphi$-transform characterization\label{s2}}

In this section, we first introduce the sequence space $\tlve$
corresponding to the space $\btlve$ and then establish their
$\vz$-transform characterization in the sense of Frazier and Jawerth \cite{fj90}.
As a consequence of the $\vz$-transform characterization,
we conclude that the space $\btlve$ is independent of the choice
of admissible function pairs $(\vz,\Phi)$.

\begin{definition}\label{d-tls}
Let $p(\cdot),\ s(\cdot)$ and $\phi$ be as in Definition
\ref{d-tl} and $q(\cdot)$
 as either in Definition \ref{d-tl} or $q(\cdot)\equiv\fz$.
Then the \emph{sequence space}
$f_{p(\cdot),q(\cdot)}^{s(\cdot),\phi}(\rn)$ is defined
to be the set of all sequences $t:=\{t_Q\}_{Q\in\cq^\ast}\subset\cc$ such that
$$\|t\|_{f_{p(\cdot),q(\cdot)}^{s(\cdot),\phi}(\rn)}
:=\sup_{P\in\cq}\frac1{\phi(P)}\lf\|\lf\{\sum_{{Q\subset P,\,Q\in\cq^\ast}}
\lf[|Q|^{-[\frac{s(\cdot)}n+\frac 12]}|t_Q|\chi_Q\r]^{q(\cdot)}
\r\}^{\frac1{q(\cdot)}}\r\|_{L^{p(\cdot)}(P)}<\fz$$
with the usual modification made when $q(\cdot)\equiv\fz$,
where the supremum is taken over all dyadic cubes $P$ in $\rn$.
\end{definition}

\begin{remark}\label{r-lattice}
(i) It is easy to see that $\tlve$ is a quasi-Banach lattice, namely, for
all $t^{(1)}:=\{t_Q^{(1)}\}_{Q\in\cq^\ast}\subset\cc$ and
$t^{(2)}:=\{t^{(2)}_Q\}_{Q\in\cq^\ast}\subset\cc$, if
$|t_Q^{(1)}|\le|t_Q^{(2)}|$ for all $Q\in\cq^\ast$, then
$$\|t^{(1)}\|_{\tlve}\le\|t^{(2)}\|_{\tlve}.$$

(ii) Let
$\mathcal D_0(\rn):=\{Q\subset\rn:\ Q\ {\rm is\ a\ cube\ and}\
\ell(Q)=2^{-j_0}\ {\rm for\ some}\
j_0\in\zz\}.$
Then it is easy to prove that the supremum in Definitions
\ref{d-tl} and \ref{d-tls}
can be equivalently taken over all cubes in $\mathcal D_0(\rn)$,
 the details being omitted.
\end{remark}

Let $(\vz,\Phi)$ be a pair of admissible functions. Then $(\wz \vz,\wz \Phi)$
is also a pair of admissible functions.
Thus, by \cite[pp.\,130-131]{fj90} or \cite[Lemma (6.9)]{fjw91},
we know that there exist
Schwartz functions $\psi$ and $\Psi$ satisfying \eqref{x.1} and \eqref{x.2},
respectively, such that, for all $\xi\in\rn$,
\begin{equation}\label{cz1}
\widehat{\Phi}(\xi)\widehat{\Psi}(\xi)+
\sum_{j=1}^\fz\widehat{\vz}(2^{-j}\xi)\widehat{\psi}(2^{-j}\xi)=1.
\end{equation}
Recall that the \emph{$\vz$-transform}
$S_\vz$ is defined to be the mapping taking each $f\in\cs'(\rn)$ to the
sequence $S_\vz(f):=\{(S_\vz f)_Q\}_{Q\in\cq^\ast}$,
where $(S_\vz f)_Q:=|Q|^{1/2}\Phi\ast f(x_Q)$ if $\ell(Q)=1$ and
$(S_\vz f)_Q:=|Q|^{1/2}\vz_{j_Q}\ast f(x_Q)$ if $\ell(Q)<1$;
the \emph{inverse $\vz$-transform} $T_\psi$ is defined to
be the mapping taking a sequence $t:=\{t_Q\}_{Q\in\cq^\ast}\subset\cc$ to
$$T_\psi t:=\sum_{{Q\in\cq^\ast,\,\ell(Q)=1}}t_Q\Psi_Q
+\sum_{{Q\in\cq^\ast,\,\ell(Q)<1}}t_Q\psi_Q;$$
see, for example, \cite[p.\,31]{ysiy}.

Now we state the following $\vz$-transform characterization for $\btlve$, which
is the main result of this section. For the corresponding result of
Triebel-Lizorkin-type spaces, see \cite[Theorem 2.1]{ysiy}.
\begin{theorem}\label{t-transform}
Let $p,\ q$, $s$ and $\phi$ be as in Definition \ref{d-tl}
and $\vz,\ \psi,\ \Phi$ and $\Psi$ be as in \eqref{cz1}.
Then the operators $S_\vz:\ F_{p(\cdot),q(\cdot)}^{s(\cdot),\phi}(\rn)
\to f_{p(\cdot),q(\cdot)}^{s(\cdot),\phi}(\rn)$ and
$T_\psi:\ f_{p(\cdot),q(\cdot)}^{s(\cdot),\phi}(\rn)
\to F_{p(\cdot),q(\cdot)}^{s(\cdot),\phi}(\rn)$ are bounded. Furthermore,
$T_\psi\circ S_\vz$ is the identity on
$F_{p(\cdot),q(\cdot)}^{s(\cdot),\phi}(\rn)$.
\end{theorem}

We remark that $T_\psi$ is well defined for all
$t\in f_{p(\cdot),q(\cdot)}^{s(\cdot),\phi}(\rn)$;
see Lemma \ref{l-estimate3} below.
The proof of Theorem \ref{t-transform} is given later.
From Theorem \ref{t-transform} and an argument similar to that
used in the proof of \cite[Remark 2.6]{fj90}, we immediately deduce the following conclusion,
the details being omitted.

\begin{corollary}\label{c-inde}
With all the notation as in Definition \ref{d-tl}, the space
$F_{p(\cdot),q(\cdot)}^{s(\cdot),\phi}(\rn)$ is independent of the choice
of the admissible function pairs $(\vz,\Phi)$.
\end{corollary}

Now we start to show Theorem \ref{t-transform}.
First, we need the following property.
\begin{lemma}\label{l-estimate3}
Let $p,\ q$, $s$ and $\phi$ be as in Definition \ref{d-tl}.
Then, for all $t\in f_{p(\cdot),q(\cdot)}^{s(\cdot),\phi}(\rn)$,
$$T_\psi t:=\sum_{{Q\in\cq^\ast,\,\ell(Q)=1}}t_Q\Psi_Q
+\sum_{{Q\in\cq^\ast,\,\ell(Q)<1}}t_Q\psi_Q$$
converges in $\cs'(\rn)$; moreover,
$T_\psi:\ f_{p(\cdot),q(\cdot)}^{s(\cdot),\phi}(\rn)\to \cs'(\rn)$
is continuous.
\end{lemma}
To prove Lemma \ref{l-estimate3}, we need the following technical lemmas.

\begin{lemma}\label{l-esti-cube}
Let $\phi$ be a set function satisfying the conditions (\textbf{S1}) and
(\textbf{S2}). Then
\begin{enumerate}
\item[{\rm(i)}] there exists a positive constant $C$ such that,
for any $j\in\zz_+$ and $k\in\zz^n$,
$$\phi(Q_{jk})\le C 2^{j\log_2c_1}(|k|+1)^{\log_2(c_1\wz c_1)};$$

\item[{\rm(ii)}] there exists a positive constant $C$ such that,
for all $Q\in\cq$ and $l\in\zz^n$,
$$\frac{\phi(Q+l\ell(Q))}{\phi(Q)}\le C(1+|l|)^{\log_2(c_1\wz c_1)},$$
where $c_1$ and $\wz c_1$ are as in the condition (\textbf{S1}).
\end{enumerate}
\end{lemma}
\begin{proof}
We first prove (i).
For any $j\in\zz_+$ and $k\in\zz^n$, let $\delta_k\in\zz_+$ be such that
$2^{\delta_k}\le n(|k|+1)< 2^{\delta_k+1}$ and $c_{Q_{jk}}$
the center of the cube
$Q_{jk}$.
Then, by the conditions (\textbf{S1}) and (\textbf{S2}) of $\phi$ and the fact that
$|c_{Q_{jk}}|\le n2^{-j}(|k|+1)$,
we see that
\begin{eqnarray*}
\phi(Q_{jk})
&=&\phi(Q(c_{Q{jk}},2^{-j}))
\le c_1^{\delta_k+1}\phi(Q(c_{Q_{jk}},2^{-j+\delta_k+1}))\\
&\le& c_2c_1^{\delta_k+1}\phi(Q(0,2^{-j+\delta_k+1}))
\ls c_1^{j+\delta_k}(\wz c_1)^{\delta_k}\phi(Q(0,1))\\
&\ls& 2^{j\log_2c_1}(|k|+1)^{\log_2(c_1\wz c_1)},
\end{eqnarray*}
which completes the proof of (i).

Next we show (ii). If $|l|\le1$, namely, $\ell(Q)|l|\le \ell(Q)$,
then, by the condition (\textbf{S2}) of $\phi$,
we find that
\begin{equation}\label{cube1}
c_2^{-1}\le\frac{\phi(Q+l\ell(Q))}{\phi(Q)}\le c_2.
\end{equation}
If $|l|>1$, namely, $\ell(Q)|l|> \ell(Q)$,
then there exists a $\gamma_l\in\nn$ such that
$2^{\gamma_l}\le|l|<2^{\gamma_l+1}$.
Thus, $\ell(Q)2^{\gamma_l+1}>\ell(Q)|l|$.
From the condition (\textbf{S1}) of $\phi$, we deduce that
\begin{eqnarray*}
\phi(Q+l\ell(Q))
&=&\phi(Q(c_Q+l\ell(Q),\ell(Q)))\\
&\le& c_1^{\gamma_l+1}\phi(Q(c_Q+l\ell(Q),\ell(Q)2^{\gamma_l+1}))
\end{eqnarray*}
and
$$\phi(Q)=\phi(Q(c_Q,\ell(Q)))
\ge \lf(\frac1{\wz c_1}\r)^{\gamma_l+1}\phi(Q(c_Q,\ell(Q)2^{\gamma_l+1})),$$
which, combined with the condition (\textbf{S2}) of $\phi$, implies that
\begin{eqnarray*}
\frac{\phi(Q+l\ell(Q))}{\phi(Q)}
&\le&\frac{c_1^{\gamma_l+1}\phi(Q(c_Q+l\ell(Q),\ell(Q)2^{\gamma_l+1}))}
{(\wz c_1)^{-\gamma_l-1}\phi(Q(c_Q,\ell(Q)2^{\gamma_l+1}))}\\
&\le& c_2c_1^{\gamma_l+1}(\wz c_1)^{\gamma_l+1}
\sim|l|^{\log_2(c_1\wz c_1)}.
\end{eqnarray*}
This, together with \eqref{cube1}, then finishes the proof of (ii) and hence
 Lemma \ref{l-esti-cube}.
\end{proof}

\begin{lemma}\label{l-esti-cube-1}
Let $p(\cdot)\in C^{\log}(\rn)$.
 Then there exists a positive
constant $C$ such that, for all dyadic cubes $Q_{jk}$
with $j\in\zz_+$ and $k\in\zz^n$,
\begin{eqnarray}\label{esti-cube-x}
\frac1C2^{-\frac n{p_-}j}(1+|k|)^{n(\frac1{p_+}-\frac 1{p_-})}
&\le& \|\chi_{Q_{jk}}\|_{L^{p(\cdot)}(\rn)}\noz\\
&\le& C2^{-\frac n{p_+}j}(1+|k|)^{n(\frac1{p_-}-\frac 1{p_+})}.
\end{eqnarray}
\end{lemma}
\begin{proof}
Let $Q_{00}$ be the dyadic cube $Q_{jk}$ with $j=0$
and $k=(0,\dots,0)\in \zz^n$.
For any $j\in\zz_+$ and $k\in\zz^n$, it is easy to see
that $Q_{jk}\subset 2(1+|k|)Q_{00}$.
Then, from \cite[Lemma 2.6]{zyl14}, we deduce that
\begin{eqnarray*}
\|\chi_{Q_{jk}}\|_{\vlp}
&\ls&\lf[\frac{|Q_{jk}|}{|2(1+|k|)Q_{00}|}\r]^{\frac1{p_+}}
\|\chi_{2(1+|k|)Q_{00}}\|_{\vlp}\\
&\ls&\lf[\frac{|Q_{jk}|}{|2(1+|k|)Q_{00}|}\r]^{\frac1{p_+}}
\lf[\frac{|2(1+|k|)Q_{00}|}{|Q_{00}|}\r]^{\frac1{p_-}}\|\chi_{Q_{00}}\|_{\vlp}\\
&\sim&2^{-\frac n{p_+}j}(1+|k|)^{n(\frac1{p_-}-\frac1{p_+})}
\end{eqnarray*}
and, similarly,
\begin{eqnarray*}
\|\chi_{Q_{jk}}\|_{\vlp}
\gs2^{-\frac n{p_-}j}(1+|k|)^{n(\frac1{p_+}-\frac 1{p_-})},
\end{eqnarray*}
which completes the proof of \eqref{esti-cube-x} and
 hence Lemma \ref{l-esti-cube-1}.
\end{proof}

In what follows, for $h\in\cs(\rn)$ and $M\in\zz_+$, let
$$\|h\|_{\cs_M(\rn)}
:=\sup_{|\gamma|\le M}\sup_{x\in\rn}|\partial^\gamma h(x)|(1+|x|)^{n+M+\gamma}.$$

\begin{proof}[Proof of Lemma \ref{l-estimate3}]
To prove this lemma, it suffices to show that there exists
an $M\in\nn$ such that, for all
$t\in f_{p(\cdot),q(\cdot)}^{s(\cdot),\phi}(\rn)$ and $h\in\cs(\rn)$,
$$|\la T_\psi t,h\ra|\ls \|t\|_{f_{p(\cdot),q(\cdot)}^{s(\cdot),\phi}(\rn)}
\|h\|_{\cs_M(\rn)}.$$
Indeed, by Remark \ref{r-vlpp}(iv), we see that, for any $Q\in\cq^\ast$,
\begin{eqnarray*}
|t_Q|
&=&\|t_Q\chi_Q\|_{L^{p(\cdot)}(Q)}\|\chi_Q\|_{L^{p(\cdot)}(Q)}^{-1}\\
&\le&\lf\|\lf\{\sum_{\gfz{\wz Q\subset Q}{\wz Q\in\cq^\ast}}
\lf[|\wz Q|^{-\frac{s(\cdot)}{n}-\frac12}|t_{\wz Q}|\chi_{\wz Q}\r]^{q(\cdot)}
\r\}^{\frac1{q(\cdot)}}\r\|_{L^{p(\cdot)}(Q)}
\|\chi_Q\|_{L^{p(\cdot)}(Q)}^{-1}
|Q|^{\frac{s_-}n+\frac12}\\
&\le&\|t\|_{f_{p(\cdot),q(\cdot)}^{s(\cdot),\phi}(\rn)}
\|\chi_Q\|_{L^{p(\cdot)}(Q)}^{-1}\phi(Q)|Q|^{\frac{s_-}n+\frac12},
\end{eqnarray*}
which implies that
\begin{eqnarray*}
|\la T_\psi t,h\ra|
&\le& \sum_{\ell(Q)=1}|t_Q||\la \Psi_Q,h\ra|+\sum_{\ell(Q)<1}|t_Q|
|\la\psi_Q,h\ra|\\
&\le& \|t\|_{f_{p(\cdot),q(\cdot)}^{s(\cdot),\phi}(\rn)}
\lf\{\sum_{\ell(Q)=1}\|\chi_Q\|_{L^{p(\cdot)}(Q)}^{-1}\phi(Q)
|\la \Psi_Q,h\ra|\r.\\
&&\hs\hs+\lf.\sum_{\ell(Q)<1}|Q|^{\frac{s_-}n+\frac12}
\|\chi_Q\|_{L^{p(\cdot)}(Q)}^{-1}\phi(Q)
|\la \psi_Q,h\ra|\r\}=:{\rm I}_1+{\rm I}_2.
\end{eqnarray*}

Let $M\in\nn$ be such that
$$M>2\max\lf\{\log_2(c_1\wz c_1)+n\lf(\frac 1{p_-}-\frac1{p_+}\r)+\frac n2,\,
\log_2 c_1+\frac n{p_-}-s_--2n\r\}.$$
Then, for I$_1$, by Lemmas \ref{l-esti-cube}
 and \ref{l-esti-cube-1}, we find that
\begin{eqnarray*}
{\rm I}_1
&\le& \|f\|_{\tlve}\sum_{k\in\zz^n}\|\chi_{Q_{0k}}\|_{L^{p(\cdot)}
(Q_{{0k}})}^{-1}\phi(Q_{0k})
|\la \Psi_{Q_{0k}},h\ra|\\
&\ls& \|f\|_{\tlve}\|h\|_{\cs_M(\rn)}
\sum_{k\in\zz^n} (1+|k|)^{n(\frac1{p_-}-\frac1{p_+})+\log_2(c_1\wz c_1)-n-M}\\
&\ls& \|f\|_{\tlve}\|h\|_{\cs_M(\rn)}.
\end{eqnarray*}
On the other hand, for I$_2$, by \cite[Lemma 2.4]{ysiy} and
Lemmas \ref{l-esti-cube} and \ref{l-esti-cube-1}, we conclude that
\begin{eqnarray*}
{\rm I}_2
&\le& \|f\|_{\tlve}\|h\|_{\cs_M(\rn)}\sum_{j=1}^\fz\sum_{k\in\zz^n}
2^{-j(s_-+\frac n2-\log_2c_1-\frac n{p_-})}2^{-jM}\\
&&\hs\hs\times(1+|k|)^{n(\frac1{p_-}-\frac1{p_+})+\log_2(c_1\wz c_1)}
\frac1{(1+|2^{-j}k|)^{n+M}}\\
&\ls& \|f\|_{\tlve}\|h\|_{\cs_M(\rn)},
\end{eqnarray*}
which, together with the estimate of ${\rm I}_1$, implies that
$$|\la T_\psi t,h\ra|\ls \|t\|_{f_{p(\cdot),q(\cdot)}^{s(\cdot),\phi}(\rn)}
\|h\|_{\cs_M(\rn)}.$$
Thus,
$$T_\psi t=\sum_{{Q\in\cq^\ast,\,\ell(Q)=1}}t_Q\Psi_Q
+\sum_{{Q\in\cq^\ast,\,\ell(Q)<1}}t_Q\psi_Q$$
converges in $\cs'(\rn)$ and
$T_\psi:\ \tlve\to\cs'(\rn)$ is continuous, which completes the
 proof of Lemma \ref{l-estimate3}.
\end{proof}

For a sequence $t=\{t_Q\}_{Q\in\cq^\ast}\subset \cc$, $r\in(0,\fz)$
and $\lz\in(0,\fz)$,
let
$$(t_{r,\lz}^\ast)_Q:=\lf\{\sum_{{R\in \cq^\ast,\,\ell(R)=\ell(Q)}}
\frac{|t_R|^r}{[1+\ell(R)^{-1}|x_R-x_Q|]^\lz}\r\}^{\frac1r},\quad
Q\in\cq^\ast,$$
and $t_{r,\lz}^\ast:=\{(t_{r,\lz}^\ast)_Q\}_{Q\in\cq^\ast}$.

We have the following estimates.

\begin{lemma}\label{l-estimate1}
Let $p,\ q$, $s$ and $\phi$ be as in Definition \ref{d-tl},
$r\in(0,\min\{p_-,q_-\})$
and $$\lz\in(n+C_{\log}(s)+r\log_2(c_1\wz c_1),\fz),$$
where $c_1$ and $\wz c_1$ are as in the condition {\textbf{(S1)}}
and $C_{\log}(s)$ is as in \eqref{ve1} with $g$ replaced by $s$.
Then there exists a constant $C\in[1,\fz)$ such that,
for all $t\in f_{p(\cdot),q(\cdot)}^{s(\cdot),\phi}(\rn)$,
$$\|t\|_{f_{p(\cdot),q(\cdot)}^{s(\cdot),\phi}(\rn)}
\le \|t_{r,\lz}^\ast\|_{f_{p(\cdot),q(\cdot)}^{s(\cdot),\phi}(\rn)}
\le C\|t\|_{f_{p(\cdot),q(\cdot)}^{s(\cdot),\phi}(\rn)}.$$
\end{lemma}

To prove Lemma \ref{l-estimate1}, we need Lemma \ref{l-estimate2} below,
which is just \cite[Theorem 3.2]{dhr09} (the vector-valued convolution inequality)
and plays a key role throughout this article.
In what follows, for any $m\in(0,\fz)$ and $j\in\zz$, let
$$\eta_{j,m}(x):=2^{jn}(1+2^j|x|)^{-m},\quad x\in\rn.$$

\begin{lemma}\label{l-estimate2}
Let $p,\ q\in C^{\log}(\rn)$ satisfy $1<p_-\le p_+<\fz$ and $1<q_-\le q_+<\fz$.
Let $m\in(n,\fz)$. Then there exists a positive constant $C$ such that,
for all sequences $\{f_j\}_{j\in\nn}\subset{ L}_{\loc}^1(\rn)$,
$$\lf\|\lf\{\sum_{j\in\nn}|\eta_{j,m}\ast f_j|^{q(\cdot)}\r\}
^{\frac1{q(\cdot)}}\r\|_{L^{p(\cdot)}(\rn)}
\le C\lf\|\lf\{\sum_{j\in\nn}|f_j|^{q(\cdot)}\r\}^{\frac1{q(\cdot)}}
\r\|_{L^{p(\cdot)}(\rn)}.$$
\end{lemma}

The following Lemma \ref{l-eta} is just \cite[Lemma 19]{kv12}
(see also \cite[Lemma 6.1]{dhr09}).
\begin{lemma}\label{l-eta}
Let $s\in C_{\rm loc}^{\log}(\rn)$ and $L\in [C_{\log}(s),\fz)$,
where $C_{\log}(s)$ is as in \eqref{ve1} with $g$ replaced by $s$. Then there
exists a positive constant $C$ such that, for all $x,\ y\in\rn$,
$m\in(0,\fz)$ and $v\in\zz_+$,
$$2^{vs(x)}\eta_{v,m+L}(x-y)\le C2^{vs(y)}\eta_{v,m}(x-y).$$
\end{lemma}

\begin{proof}[Proof of Lemma \ref{l-estimate1}]
Notice that, for all $Q\in\cq^\ast$, $|t_Q|\le(t_{r,\lz}^\ast)_Q$.
This immediately implies that
$\|t\|_{f_{p(\cdot),q(\cdot)}^{s(\cdot),\phi}(\rn)}
\le \|t_{r,\lz}^\ast\|_{f_{p(\cdot),q(\cdot)}^{s(\cdot),\phi}(\rn)},$
since $\tlve$ is a quasi-Banach lattice (see Remark \ref{r-lattice}(i)).

Conversely, let $P$ be a given dyadic cube. For any $Q\in\cq^\ast$, let
$v_Q:=t_Q$ if $Q\subset 3P$ and $v_Q:=0$ otherwise, and let $u_Q:=t_Q-v_Q$.
Set $v:=\{v_Q\}_{Q\in\cq^\ast}$ and $u:=\{u_Q\}_{Q\in\cq^\ast}$.
Then, for all $Q\in\cq^\ast$, we have
\begin{equation}\label{estimate1-x}
(t_{r,\lz}^\ast)_Q
\ls(v_{r,\lz}^\ast)_Q+(u_{r,\lz}^\ast)_Q.
\end{equation}
From the proof of \cite[Theorem 3.11]{dhr09}, we deduce that,
for all $t\in f_{p(\cdot),q(\cdot)}^{s(\cdot),1}(\rn)$,
\begin{equation*}
\|t_{r,\lz}^\ast\|_{f_{p(\cdot),q(\cdot)}^{s(\cdot),1}(\rn)}
\ls\|t\|_{f_{p(\cdot),q(\cdot)}^{s(\cdot),1}(\rn)},
\end{equation*}
which implies that
\begin{eqnarray*}
{\rm I}_P
&:=&\frac1{\phi(P)}\lf\|\lf[\sum_{{Q\subset P,\,Q\in\cq^\ast}}
\lf\{|Q|^{-[\frac{s(\cdot)}n+\frac 12]}(v_{r,\lz}^\ast)_Q\chi_Q\r\}
^{q(\cdot)}\r]^{\frac1{q(\cdot)}}\r\|_{L^{p(\cdot)}(P)}\\
&\ls&\frac1{\phi(P)}\lf\|v_{r,\lz}^\ast
\r\|_{f_{p(\cdot),q(\cdot)}^{s(\cdot),1}(\rn)}
\ls\frac1{\phi(P)}\|v\|_{f_{p(\cdot),q(\cdot)}^{s(\cdot),1}(\rn)}\\
&\sim&\frac1{\phi(P)}\lf\|\lf[\sum_{{Q\subset 3P,\,Q\in\cq^\ast}}
\lf\{|Q|^{-[\frac{s(\cdot)}n+\frac12]}|t_Q|\chi_Q\r\}^{q(\cdot)}
\r]^{\frac1{q(\cdot)}}
\r\|_{L^{p(\cdot)}(3P)}.
\end{eqnarray*}
By this and the condition (\textbf{S1}) of $\phi$, we conclude that
\begin{equation}\label{estimate1-y}
\|v_{r,\lz}^\ast\|_{\tlve}\le\sup_{P\in\cq}{\rm I}_P\ls\|t\|_{f_{p(\cdot),
q(\cdot)}^{s(\cdot),\phi}(\rn)}.
\end{equation}

Next, we deal with $u$.
To this end, let, for $i\in\zz_+$ and $k\in\zz^n$,
$$A(i,k,P):=\lf\{\wz Q\in\cq^\ast:\ \ell(\wz Q)=2^{-i}\ell(P),
\ \wz Q\subset P+k\ell(P)\r\}.$$
Then, we have
\begin{eqnarray*}
{\rm J}_P
&:=&\frac1{\phi(P)}\lf\|\lf[\sum_{{Q\subset P,\,Q\in\cq^\ast}}
\lf\{|Q|^{-[\frac{s(\cdot)}{n}+\frac12]}
(u_{r,\lz}^\ast)_Q\chi_Q\r\}^{q(\cdot)}\r]^{\frac1{q(\cdot)}}
\r\|_{L^{p(\cdot)}(P)}\\
&=&\frac1{\phi(P)}\lf\|\lf\{\sum_{i=0}^\fz\sum_{\gfz{Q\subset P,Q\in\cq^\ast}
{\ell(Q)=2^{-i}\ell(P)}}
\Bigg(|Q|^{-[\frac{s(\cdot)}{n}+\frac12]}\r.\r.\\
&&\quad\times\lf.\lf.\lf.\lf\{\sum_{\gfz{k\in\zz^n}{|k|\ge2}}
\sum_{R\in A(i,k,P)}
\frac{|u_R|^{r}}
{[1+\ell(Q)^{-1}|x_R-x_Q|]^\lz}\r\}^{\frac1{r}}\chi_Q\r)
^{q(\cdot)}\r\}^{\frac1{q(\cdot)}}
\r\|_{L^{p(\cdot)}(P)}.
\end{eqnarray*}
Notice that, when $x\in Q$, $y\in R$ and $\ell(R)=\ell(Q)$,
$$1+[\ell(Q)]^{-1}|x-y|\sim 1+[\ell(Q)]^{-1}|x_Q-x_R|.$$
From this, we deduce that, for all $x\in Q$, $\mu\in(0,\fz)$,
$j\in\zz_+$ and $k\in\zz^n$,
\begin{eqnarray*}
\eta_{j_Q,\mu}
\ast\lf(\lf[\sum_{R\in A(i,k,P)}|u_R|\chi_R\r]^r\r)(x)
&=&\sum_{R\in A(i,k,P)}\int_{R}\frac{2^{nj_Q}|u_R|^r}
{(1+2^{j_Q}|x-y|)^{\mu}}\,dy\\
&\sim&\sum_{R\in A(i,k,P)}\frac{|u_R|^r}
{[1+\ell(Q)^{-1}|x_Q-x_R|]^{\mu}}.
\end{eqnarray*}
Since $$\lz>n+C_{\log}(s)+r\log_2(c_1\wz c_1),$$ it follows that
there exist $m\in(n,\fz)$ and $L\in(C_{\log}(s),\fz)$ such that
$$\lz>m+L+r\log_2(c_1\wz c_1).$$
Observe that, when $|k|\ge2$, $i\in\zz_+$, $Q\subset P$ with
$\ell(Q)=2^{-i}\ell(P)$
and $R\in A(i,k,P)$, $$1+[\ell(Q)]^{-1}|x_Q-x_R|\sim 2^i|k|.$$
Then, by Lemma \ref{l-eta}, we conclude that
\begin{eqnarray*}
{\rm J}_P
&\ls&\frac1{\phi(P)}\lf\|\lf\{\sum_{i=0}^\fz
\lf(\sum_{\gfz{\ell(Q)=2^{-i}\ell(P)}{Q\in\cq^\ast,\,Q\subset P}}
\lf\{\sum_{{k\in\zz^n,\,|k|\ge2}}(2^i|k|)^{m+L-\lz}
|Q|^{-[\frac{s(\cdot)}{n}+\frac12]r}\r.\r.\r.\r.\\
&&\quad\quad\times\lf.\lf.\lf.\lf.
\eta_{j_Q,m+L}\ast
\lf(\lf[\sum_{R\in A(i,k,P)}|u_R|\chi_R\r]^r\r)
\r\}^{\frac1{r}}\chi_Q
\r)^{q(\cdot)}\r\}^{\frac1{q(\cdot)}}\r\|_{L^{p(\cdot)}(P)}\\
&\ls&\frac1{\phi(P)}\lf\|\lf\{\sum_{i=0}^\fz
\lf[\sum_{{k\in\zz^n,\,|k|\ge2}}(2^i|k|)^{m+L-\lz}\r.\r.\r.\\
&&\quad\quad\times\lf.\lf.\lf.
\eta_{i+j_P,m}\ast
\lf(\lf\{\sum_{R\in A(i,k,P)}|R|^{-[\frac{s(\cdot)}{n}
+\frac12]}|u_R|\chi_R\r\}^r\r)
\r]^{\frac{q(\cdot)}{r}}
\r\}^{\frac1{q(\cdot)}}\r\|_{L^{p(\cdot)}(P)},
\end{eqnarray*}
which, combined with Remark \ref{r-vlpp}(i) and the fact that, for
all $d\in[0,1]$ and $\{\theta_j\}_j\subset\cc$,
\begin{equation}\label{simple-ineq}
\lf(\sum_{j}|\theta_j|\r)^d\le \sum_j|\theta_j|^d,
\end{equation}
implies that
\begin{eqnarray*}
{\rm J}_P&\ls&\frac1{\phi(P)}\lf\|\sum_{i=0}^\fz\sum_{{k\in\zz^n,\,
|k|\ge2}}(2^i|k|)^{m+L-\lz}\r.\\
&&\hs\times\lf.\eta_{i+j_P,m}\ast
\lf(\lf\{\sum_{R\in A(i,k,P)}|R|^{-[\frac{s(\cdot)}{n}
+\frac12]}|u_R|\chi_R\r\}^r\r)
\r\|_{L^{\frac{p(\cdot)}{r}}(P)}^{\frac1{r}}\\
&\ls&\frac1{\phi(P)}\lf\{\sum_{i=0}^\fz\sum_{{k\in\zz^n,\,
|k|\ge2}}(2^i|k|)^{m+L-\lz}\r.\\
&&\hs\times\lf.\lf\|\eta_{i+j_P,m}\ast
\lf(\lf\{\sum_{R\in A(i,k,P)}|R|^{-[\frac{s(\cdot)}{n}
+\frac12]}|u_R|\chi_R\r\}^r\r)
\r\|_{L^{\frac{p(\cdot)}{r}}(P)}\r\}^{\frac1{r}}.
\end{eqnarray*}
From this, $m\in(n,\fz)$, Lemma \ref{l-estimate2},
Lemma \ref{l-esti-cube}(ii) and $$\lz>m+L+r\log_2(c_1\wz c_1),$$ we deduce that
\begin{eqnarray*}
{\rm J}_P
&\ls&\frac1{\phi(P)}\lf\{\sum_{i=0}^\fz\sum_{\gfz{k\in\zz^n}{|k|\ge2}}
(2^i|k|)^{m+L-\lz}\lf\|\sum_{R\in A(i,k,P)}|R|^{-[\frac{s(\cdot)}{n}+\frac12]}
|u_R|\chi_R\r\|_{L^{p(\cdot)}(\rn)}^{r}\r\}^\frac1{r}\\
&\ls&\lf\{\sum_{i=0}^\fz\sum_{\gfz{k\in\zz^n}{|k|\ge2}}
(2^i|k|)^{m+L-\lz}\lf[\frac{\phi(P+k\ell(P))}{\phi(P)}\r]^{r}
\r\}^{\frac1{r}}
\|t\|_{f_{p(\cdot),q(\cdot)}^{s(\cdot),\phi}(\rn)}\\
&\ls&\lf\{\sum_{i=0}^\fz\sum_{\gfz{k\in\zz^n}{|k|\ge2}}
2^{i(m+L-\lz)}|k|^{m+L-\lz}|k|^{r\log (c_1\wz c_1)}\r\}^\frac1{r}
\|t\|_{f_{p(\cdot),q(\cdot)}^{s(\cdot),\phi}(\rn)}
\sim\|t\|_{f_{p(\cdot),q(\cdot)}^{s(\cdot),\phi}(\rn)}.
\end{eqnarray*}
This, together with the arbitrary of $P\in\cq$, further implies that
\begin{equation}\label{estimate1-z}
\|u_{r,\lz}^\ast\|_{\tlve}\ls\|t\|_{\tlve}.
\end{equation}

Combining \eqref{estimate1-x}, \eqref{estimate1-y} and \eqref{estimate1-z},
we conclude that
$$\|t_{r,\lz}^\ast\|_{f_{p(\cdot),q(\cdot)}^{s(\cdot),\phi}(\rn)}
\ls\|t\|_{f_{p(\cdot),q(\cdot)}^{s(\cdot),\phi}(\rn)},$$
which completes the proof of Lemma \ref{l-estimate1}.
\end{proof}

Now we give the proof of Theorem \ref{t-transform}.

\begin{proof}[Proof of Theorem \ref{t-transform}]
We first show that $S_\vz$ is bounded from $\btlve$ to $\tlve$.
Let $f\in\btlve$. From \cite[Lemma A.6]{dhr09} (the $r$-trick lemma)
and its proof, we deduce that,
for all $L\in[C_{\log}(s),\fz)$, $m\in(n+\log_2c_1,\fz)$,
$$r\in(0,\min\{1,p_-,q_-\})$$
 and
 $x\in Q:=Q_{jk}\in \cq^\ast$,
\begin{equation*}
\sup_{z\in Q}|\vz_j\ast f(z)|^r\ls 2^{jn}\sum_{l\in\zz^n}
\int_{Q_{j(k+l)}}(1+2^j|x-y|)^{-(2m+L)}|\vz_j\ast f(y)|^r\,dy,
\end{equation*}
which implies that
\begin{eqnarray}\label{vz1}
&&\|S_\vz f\|_{\tlve}\noz\\
&&\hs\hs\ls\sup_{P\in\cq}\frac1{\phi(P)}\lf\|\lf[\sum_{j=(j_P\vee 0)}^\fz
\Bigg\{\sum_{k\in\zz^n}2^{js(\cdot)}\r.\r.\noz\\
&&\hs\hs\hs\times\lf.\lf.\lf.\lf[2^{jn}\sum_{l\in\zz^n}\int_{Q_{j(k+l)}}
\frac{|\vz_j\ast f(y)|^r}{(1+2^j|\cdot-y|)^{2m+L}}\,dy\r]^{\frac1r}
\chi_{Q_{jk}}\r\}^{q(\cdot)}
\r]^{\frac1{q(\cdot)}}\r\|_{L^{p(\cdot)}(P)}\noz\\
&&\hs\hs\ls\sup_{P\in\cq}\frac1{\phi(P)}\lf\|\lf[\sum_{j=(j_P\vee 0)}^\fz
\Bigg\{\sum_{k\in\zz^n}2^{js(\cdot)}\sum_{l\in\zz^n}(1+|l|)^{-m}\r.\r.\noz\\
&&\hs\hs\hs\times\lf.\lf.\lf[\eta_{j,m+L}\ast
(|\vz_j\ast f\chi_{Q_{j(k+l)}}|^r)\r]\chi_{Q_{jk}}\Bigg\}^{\frac{q(\cdot)}r}\r]
^{\frac1{q(\cdot)}}\r\|_{L^{p(\cdot)}(P)},
\end{eqnarray}
where the last inequality comes from the fact that, when
$x\in Q_{jk}$ and $y\in Q_{j(k+l)}$, $$1+2^j|x-y|\sim 1+|l|.$$
Observe that, for any given $P\in\cq$, if $Q_{jk}\subset P$, then
$Q_{j(k+l)}\subset 3n|l|P$ for all $l\in\zz^n$. By this and \eqref{vz1},
we see that
\begin{eqnarray*}
\|S_\vz f\|_{\tlve}
\ls&&\sup_{P\in\cq}\frac1{\phi(P)}
\lf\|\lf[\sum_{j=(j_P\vee0)}^\fz\Bigg\{2^{jrs(\cdot)}\r.\r.\\
&&\!\times\!\lf.\lf.\lf.\lf[\sum_{l\in\zz^n}(1+|l|)^{-m}\eta_{j,m+L}\ast
(|\vz_j\ast f\chi_{3n|l|P}|^r)\r]\r\}^{\frac{q(\cdot)}r}\r]^{\frac1{q(\cdot)}}
\r\|_{L^{p(\cdot)}(P)},
\end{eqnarray*}
which, combined with the Minkowski inequality, Lemma \ref{l-eta}
and Remark \ref{r-vlpp}(i), further implies that
\begin{eqnarray*}
&&\|S_\vz f\|_{\tlve}\\
&&\hs\ls\sup_{P\in\cq}\frac1{\phi(P)}
\Bigg\{\sum_{l\in\zz^n}(1+|l|)^{-m}\\
&&\hs\hs\times\lf.\lf\|\lf(\sum_{j=(j_P\vee0)}^\fz
\lf[\eta_{j,m+L}\ast(|2^{js(\cdot)}\vz_j\ast f\chi_{3n|l|P}|^r)\r]
^{\frac{q(\cdot)}{r}}\r)^{\frac r{q(\cdot)}}
\r\|_{L^{\frac{p(\cdot)}{r}}(P)}\r\}^{\frac1r}.
\end{eqnarray*}
From this, Lemma \ref{l-estimate2}, $m\in(n+\log_2c_1,\fz)$ and
the condition (\textbf{S1}) of $\phi$,
it follows that
\begin{eqnarray*}
&&\|S_\vz f\|_{\tlve}\\
&&\hs\ls\sup_{P\in\cq}\frac1{\phi(P)}
\lf\{\sum_{l\in\zz^n}(1+|l|)^{-m}\lf\|
\lf[\sum_{j=(j_P\vee 0)}^\fz\lf(2^{js(\cdot)}|\vz_j\ast f|\r)^{q(\cdot)}
\r]^\frac{1}{q(\cdot)}\r\|_{L^{p(\cdot)}(3n|l|P)}^r\r\}^\frac1r\\
&&\hs\ls\|f\|_{\btlve}\lf[\sum_{l\in\zz^n}(1+|l|)^{-m}(1+|l|)^{\log_2c_1}
\r]^\frac1r\ls\|f\|_{\btlve},
\end{eqnarray*}
which implies that $S_\vz$ is bounded from $\btlve$ to $\tlve$.

By repeating the argument used in the proof of \cite[Theorem 2.1]{ysiy},
with \cite[Lemmas 2.7 and 2.8]{ysiy} therein replaced by Lemmas \ref{l-estimate3}
and \ref{l-estimate1} here, we conclude
that $T_\psi$ is bounded from $\tlve$ to $\btlve$, the details being omitted.
Finally, by the Calder\'on reproducing formula
(see, for example, \cite[Lemma 2.3]{ysiy}),
we know that $T_\psi\circ S_\vz$ is the
identity on $\btlve$, which completes the proof of Theorem \ref{t-transform}.
\end{proof}

\section{Several equivalent characterizations of $\btlve$\label{s3}}

In this section, we first establish molecular and atomic
characterizations for $\btlve$
via Sobolev embeddings.
Secondly, we characterize $\btlve$ in terms of the Peetre maximal function,
which is further applied to show that
$$\cs(\rn)\hookrightarrow \btlve\hookrightarrow\cs'(\rn)$$
and give out two equivalent quasi-norms of $\btlve$, which may be useful in applications.

For notational simplicity, in what follows, for all $Q\in\cq^\ast$,
let $\wz \chi_Q:=|Q|^{-1/2}\chi_Q$.

\begin{proposition}\label{p-se1}
Let $\phi$ be a set function as in Definition \ref{d-tl},
$s_0,\ s_1,\ p_0,\ p_1$ be measurable functions satisfying that,
for all $x\in\rn$,
$-\fz<s_1(x)\le s_0(x)<\fz$,
$0<p_0(x)\le p_1(x)<\fz$
and $$s_0(x)-\frac n{p_0(x)}=s_1(x)-\frac n{p_1(x)}.$$
Assume that $0<(p_0)_-\le (p_1)_-\le (p_1)_+<\fz$ and
$s_0,\ \frac 1{p_0}\in C_{\loc}^{\log}(\rn)$. If $q(x)=\fz$ for all $x\in\rn$
or $0<q_-\le q(x)<\fz$ for all $x\in\rn$, then
$$f_{p_0(\cdot),q(\cdot)}^{s_0(\cdot),\phi}(\rn)\hookrightarrow
f_{p_1(\cdot),q(\cdot)}^{s_1(\cdot),\phi}(\rn).$$
\end{proposition}
\begin{proof}
Let $t:=\{t_{Q}\}_{Q\in\cq^\ast}\in
f_{p_0(\cdot),q(\cdot)}^{s_0(\cdot),\phi}(\rn)$.
We need to prove
\begin{equation*}
\|t\|_{f_{p_1(\cdot),q(\cdot)}^{s_1(\cdot),\phi}(\rn)}
\ls\|t\|_{f_{p_0(\cdot),q(\cdot)}^{s_0(\cdot),\phi}(\rn)}.
\end{equation*}
To this end, let $P\in \cq$ be any given dyadic cube.
For all $Q\in\cq^\ast$, let $u_{Q}:=t_Q$ when $Q\subset P$
and $u_Q:=0$ otherwise. Then, by the Sobolev embedding theorem
(\cite[Theorem 3.1]{vj09}), namely,
$
f_{p_0(\cdot),q(\cdot)}^{s_0(\cdot),1}(\rn)\hookrightarrow
f_{p_1(\cdot),q(\cdot)}^{s_1(\cdot),1}(\rn),
$
we conclude that
\begin{eqnarray*}
&&\lf\|\lf\{\sum_{j=(j_P\vee 0)}^\fz\sum_{k\in\zz^n}
\lf[2^{js_1(\cdot)}|t_{Q_{jk}}|\wz \chi_{Q_{jk}}\r]^{q(\cdot)}
\r\}^{\frac 1{q(\cdot)}}
\r\|_{L^{p_1(\cdot)}(P)}\\
&&\hs=\lf\|\lf\{\sum_{j=0}^\fz\sum_{k\in\zz^n}
\lf[2^{js_1(\cdot)}| u_{Q_{jk}}|\wz\chi_{Q_{jk}}\r]^{q(\cdot)}
\r\}^{\frac 1{q(\cdot)}}
\r\|_{L^{p_1(\cdot)}(\rn)}
=\|u\|_{f_{p_1(\cdot),q(\cdot)}^{s_1(\cdot),1}(\rn)}\\
&&\hs\ls\|u\|_{f_{p_0(\cdot),q(\cdot)}^{s_0(\cdot),1}(\rn)}
\sim\lf\|\lf\{\sum_{j=0}^\fz\sum_{k\in\zz^n}
\lf[2^{js_0(\cdot)}|u_{Q_{jk}}|\wz \chi_{Q_{jk}}\r]^{q(\cdot)}
\r\}^{\frac 1{q(\cdot)}}
\r\|_{L^{p_0(\cdot)}(\rn)}\\
&&\hs\sim\lf\|\lf\{\sum_{j=(j_P\vee 0)}^\fz\sum_{k\in\zz^n}
\lf[2^{js_0(\cdot)}|t_{Q_{jk}}|\wz \chi_{Q_{jk}}\r]^{q(\cdot)}
\r\}^{\frac 1{q(\cdot)}}
\r\|_{L^{p_0(\cdot)}(P)},
\end{eqnarray*}
which implies that
\begin{eqnarray*}
\|t\|_{f_{p_1(\cdot),q(\cdot)}^{s_1(\cdot),\phi}(\rn)}
&\ls&\sup_{P\in\cq}\frac1{\phi(P)}
\lf\|\lf\{\sum_{j=(j_P\vee 0)}^\fz\sum_{k\in\zz^n}
\lf[2^{js_0(\cdot)}|t_{Q_{jk}}|\wz \chi_{Q_{jk}}\r]^{q(\cdot)}
\r\}^{\frac 1{q(\cdot)}}
\r\|_{L^{p_0(\cdot)}(P)}\\
&\sim& \|t\|_{f_{p_0(\cdot),q(\cdot)}^{s_0(\cdot),\phi}(\rn)}.
\end{eqnarray*}
This finishes the proof of Proposition \ref{p-se1}
\end{proof}

\begin{proposition}\label{p-se2}
Let $\phi$ be a set function as in Definition \ref{d-tl},
$s_0,\ s_1,\ p_0,\ p_1$ be measurable functions satisfying that,
for all $x\in\rn$,
$-\fz<s_1(x)\le s_0(x)<\fz$, $0<p_0(x)\le p_1(x)<\fz$
and $$s_0(x)-\frac n{p_0(x)}=s_1(x)-\frac n{p_1(x)}.$$
Assume that $0<(p_0)_-\le (p_1)_+<\fz$,
$s_0,\ \frac 1{p_0}\in C_{\loc}^{\log}(\rn)$ and
$\inf_{x\in \rn}[s_0(x)-s_1(x)]>0.$
Then, for
all $q\in(0,\fz]$,
$$f_{p_0(\cdot),\fz}^{s_0(\cdot),\phi}(\rn)\hookrightarrow
f_{p_1(\cdot),q}^{s_1(\cdot),\phi}(\rn).$$
\end{proposition}
The proof of Proposition \ref{p-se2} is similar to that of Proposition
\ref{p-se1}, with \cite[Theorem 3.1]{vj09} replaced by
\cite[Theorem 3.2]{vj09}, the details being omitted.
\begin{remark}
(i) When $\phi(Q)=1$ for all cube $Q$, the conclusions of
Propositions \ref{p-se1} and \ref{p-se2} are just
\cite[Theorem 3.1]{vj09} and \cite[Theorem 3.2]{vj09}, respectively.

(ii) When $p,\ q,\ s$ and $\phi$ are as in Remark \ref{r-defi}(ii),
Proposition \ref{p-se2}
goes back to \cite[Proposition 2.5]{ysiy}.
\end{remark}

Combining Theorem \ref{t-transform} and Proposition \ref{p-se1}, we immediately
obtain the following Corollary \ref{c-se1}, the details being omitted.

\begin{corollary}\label{c-se1}
Let $i\in\{0,1\}$, $p_i$, $q\in \cp(\rn)$ satisfy
$\frac1{p_i},\ \frac1{q}\in C^{\log}(\rn)$ and $s_i$ be measurable functions
satisfying $s_i\in C_{\loc}^{\log}(\rn)\cap L^\fz(\rn)$, and
$\phi$ a set function satisfying the conditions
(\textbf{S1}) and (\textbf{S2}).
Under the same assumptions as in Proposition \ref{p-se1},
the following conclusion
$$F_{p_0(\cdot),q(\cdot)}^{s_0(\cdot),\phi}(\rn)\hookrightarrow
F_{p_1(\cdot),q(\cdot)}^{s_1(\cdot),\phi}(\rn)$$
holds true.
\end{corollary}

\begin{corollary}\label{c-se2}
Let $i\in\{0,1\}$, $p_i$, $q_i\in \cp(\rn)$ satisfy
$\frac1{p_i},\ \frac1{q_i}\in C^{\log}(\rn)$ and $s_i$ be measurable functions
satisfying $s_i\in C_{\loc}^{\log}(\rn)\cap L^\fz(\rn)$, and
$\phi$ a set function satisfying the conditions
(\textbf{S1}) and (\textbf{S2}).
Assume that, for all $x\in\rn$,
$$s_0(x)-\frac n{p_0(x)}=s_1(x)-\frac n{p_1(x)}$$
and $\inf_{x\in\rn}[s_0(x)-s_1(x)]>0$.
Then $$F_{p_0(\cdot),q_0(\cdot)}^{s_0(\cdot),\phi}(\rn)\hookrightarrow
F_{p_1(\cdot),q_1(\cdot)}^{s_1(\cdot),\phi}(\rn).$$
\end{corollary}
\begin{proof}
By  Proposition \ref{p-se2} and \eqref{simple-ineq},
we see that
$$f_{p_0(\cdot),q_0(\cdot)}^{s_0(\cdot),\phi}(\rn)
\hookrightarrow f_{p_0(\cdot),\fz}^{s_0(\cdot),\phi}(\rn)
\hookrightarrow f_{p_1(\cdot),(q_1)_-}^{s_1(\cdot),\phi}(\rn)
\hookrightarrow f_{p_1(\cdot),q_1(\cdot)}^{s_1(\cdot),\phi}(\rn).$$
From this and Theorem \ref{t-transform}, we deduce that
$F_{p_0(\cdot),q_0(\cdot)}^{s_0(\cdot),\phi}(\rn)\hookrightarrow
F_{p_1(\cdot),q_1(\cdot)}^{s_1(\cdot),\phi}(\rn)$,
which completes the proof of Corollary \ref{c-se2}.
\end{proof}

Next we establish molecular and atomic characterizations of
Triebel-Lizorkin-type spaces with variable exponents.

\begin{definition}\label{d-atom}
Let $K\in\zz_+$, $L\in\zz$ and $R\in\nn$.
\begin{enumerate}
\item[(i)] A measurable function $m_Q$
on $\rn$ is called a $(K,L,R)$-\emph{smooth molecule} with $Q:=Q_{jk}\in\cq$,
where $j\in\zz$ and $k\in\zz^n$,
if it satisfies the following conditions:

\begin{enumerate}
\item[(M1)] (vanishing moment) when $j\in\nn$, $\int_\rn x^\gamma m_Q(x)\,dx=0$
for all $\gamma\in\zz_+^n$ and $|\gamma|< L$;

\item[(M2)] (smoothness condition) for all multi-indices $\alpha\in\zz_+^n$,
with $|\alpha|\le K$, and all $x\in\rn$
$$|D^\alpha m_Q(x)|\le 2^{(|\alpha|+n/2)j}(1+2^j|x-x_Q|)^{-R}.$$
\end{enumerate}

\item[(ii)] A measurable function $a_Q$
on $\rn$ is called a $(K,L)$-\emph{smooth atom} supported near
$Q:=Q_{jk}\in\cq$, where $j\in\zz$ and $k\in\zz^n$,
if it satisfies the following conditions:

\begin{enumerate}
\item[(A1)] supp\,$a_Q\subset 3Q$;

\item[(A2)] (vanishing moment) when $j\in\nn$, $\int_\rn x^\gamma a_Q(x)\,dx=0$
for all $\gamma\in\zz_+^n$ with $|\gamma|< L$;

\item[(A3)] (smoothness condition) for all multi-indices $\alpha\in\zz_+^n$
with $|\alpha|\le K$,
$$|D^\alpha a_Q(x)|\le 2^{(|\alpha|+n/2)j}\quad \textrm{for all}\quad x\in\rn.$$
\end{enumerate}
\end{enumerate}
\end{definition}

\begin{remark}\label{r-atom}
(i) If $L<0$, then the vanishing moment conditions (M1) and (A2) are avoid.

(ii) Let $a_Q$ be a $(K,L)$-smooth atom with $Q:=Q_{jk}\in\cq^\ast$
with $j\in\zz_+$ and $k\in\zz^n$.
Then, by combining the conditions (A1) and (A3),
 we conclude that, for all $R\in(0,\fz)$, $\alpha\in\zz_+^n$ with $|\alpha|\le K$
 and $x\in\rn$,
$$|D^\alpha a_Q(x)|\le C 2^{(|\alpha|+n/2)j}\frac1{(1+2^j|x-x_Q|)^R},$$
where $C$ is a positive constant independent of $x,\,\alpha,\,Q$ and $a_Q$,
but depending on $R$.
Thus, each $(K,L)$-smooth atom is a $(K,L,R)$-smooth
 molecule up to a harmless positive constant.
\end{remark}

\begin{theorem}\label{t-md}
Let $p$, $q$, $s$ and $\phi$ be as in Definition \ref{d-tl}.
\begin{enumerate}
\item[{\rm (i)}] Let
$K\in( s_++\max\{0,\log_2\wz c_1\},\fz)$ and
\begin{equation}\label{md7}
 L\in\lf(\frac n{\min\{1,p_-,q_-\}}-n-s_-,\fz\r).
\end{equation}
Suppose that $\{m_Q\}_{Q\in\cq^\ast}$ is a
family of $(K,L,R)$-smooth molecules with $R$ large enough
and that $t:=\{t_Q\}_{Q\in\cq^\ast}\in\tlve$. Then
$f:=\sum_{Q\in\cq^\ast}t_Qm_Q$
converges in $\cs'(\rn)$ and
$$\|f\|_{\btlve}\le C\|t\|_{\tlve}$$
with $C$ being a positive constant independent of $t$.

\item[{\rm (ii)}] Conversely, if $f\in\btlve$, then, for any given $K,\ L\in\zz_+$,
 there exist a sequence
$t:=\{t_Q\}_{Q\in\cq^\ast}\subset\cc$ and a sequence
$\{a_Q\}_{Q\in\cq^\ast}$ of $(K,L)$-smooth atoms such that
$f=\sum_{Q\in\cq^\ast}t_Q a_Q$ in $\cs'(\rn)$
and
$$\|t\|_{\tlve}\le C\|f\|_{\btlve}$$
with $C$ being a positive constant independent of $f$.
\end{enumerate}
\end{theorem}

\begin{remark}\label{r-mad}
(i) When $\phi(P):=1$ for all cubes $P\subset\rn$, conclusions of Theorem \ref{t-md}
coincide with those of \cite[Corollary 5.6]{kempka10}; when $p,\ q$, $s$ and
$\phi$ are as in Remark
\ref{r-defi}(ii),
Theorem \ref{t-md} goes back to \cite[Theorem 3.12]{d13} (see also
\cite[Theorem 3.3]{ysiy}).

(ii) In the case that $\phi(P):=1$ for all cubes $P\subset\rn$ and $s\ge0$,
the vanishing moment and the smoothness conditions
of Theorem \ref{t-md}
can be further refined. Indeed, it was proved in
\cite[Theorem 3.11]{dhr09} that the vanishing moment
and the smoothness conditions of atoms can be localized on dyadic cubes
associated with atoms in Theorem \ref{t-md}.
\end{remark}

\begin{proof}[Proof of Theorem \ref{t-md}]
The proof of (ii) is similar to that of \cite[Theorme 3.3]{ysiy} (see also
\cite[Theorem 4.1]{fj90}). Indeed, by repeating the argument
that used in the proof of
\cite[Theorem 3.3]{ysiy}, with \cite[Lemma 2.8]{ysiy} therein replaced by
Lemma \ref{l-estimate1}, we can prove (ii), the details being omitted.

Next we prove (i) by two steps.

\emph{Step 1)} We show that $f=\sum_{Q\in\cq^\ast}t_Qm_Q$ converges in $\cs'(\rn)$.
To this end, it suffices to show that
\begin{equation}\label{3.1x}
\lim_{{N\to\fz,\,\Lambda\to\fz}}\sum_{j=0}^{N}\sum_{k\in\zz^n,\,|k|\le \Lambda}
t_{Q_{jk}}m_{Q_{jk}}
\end{equation}
exists in $\cs'(\rn)$. For all $h\in\cs(\rn)$ and $j\in\zz_+$,
by the vanishing moment condition (M1),
we see that
\begin{eqnarray*}
&&\int_\rn\sum_{k\in\zz^n,\,|k|\le \Lambda}t_{Q_{jk}}m_{Q_{jk}}(y)h(y)\,dy\\
&&\hs=\int_\rn\sum_{{k\in\zz^n,\,|k|\le \Lambda}}t_{Q_{jk}}m_{Q_{jk}}(y)
\lf[h(y)-\sum_{{\gamma\in\zz^n,\,|\gamma|< L}}
(y-x_{Q_{jk}})^\gamma\frac{D^\gamma h(x_{Q_{jk}})}{\gamma!}\r]\,dy
\end{eqnarray*}
and, by Taylor's remainder theorem, we find that, for all $y\in\rn$,
\begin{eqnarray*}
&&\lf|h(y)-\sum_{{\gamma\in\zz^n,\,|\gamma|\le L}}
(y-x_{Q_{jk}})^\gamma\frac{D^\gamma h(x_{Q_{jk}})}{\gamma!}\r|\\
&&\hs\ls|y-x_{Q_{jk}}|^{L}\sum_{|\gamma|=L}
\frac{|D^\gamma h(\xi(y))|}{\gamma !}\\
&&\hs\ls(1+|y-x_{Q_{jk}}|)^{L}(1+|y|)^{-\delta}
\sup_{\xi\in\rn}\sum_{|\gamma|=L}(1+|\xi|)^\delta
\frac{|D^\gamma h(\xi)|}{\gamma!},
\end{eqnarray*}
where $\delta\in(0,\fz)$ and
$\xi(y):=y+\theta(x_{Q_{jk}}-y)$ with some $\theta\in(0,1)$ depending on
$y$ and $x_{Q_{jk}}$, which, together with
the fact that, for all $y\in\rn$,
\begin{equation*}
|m_{Q_{jk}}(y)|\le 2^{jn/2}(1+2^j|y-x_{Q_{jk}}|)^{-R},
\end{equation*}
further implies that
\begin{eqnarray}\label{mdc2}
&&\lf|\int_\rn\sum_{{k\in\zz^n,\,|k|\le \Lambda}}
t_{Q_{jk}}m_{Q_{jk}}(y)h(y)\,dy\r|\noz\\
&&\hs\ls\int_\rn\sum_{{k\in\zz^n,\,|k|\le \Lambda}}|t_{Q_{jk}}|
2^{-j(L-n/2)}\frac{(1+|y|)^{-\delta}}
{(1+2^j|y-x_{Q_{jk}}|)^{R-L}}\,dy\noz\\
&&\hs\ls\sum_{v=0}^\fz\int_{D_v}\sum_{k\in\zz^n}|t_{Q_{jk}}|
2^{-j(L-n/2)}\frac{(1+|y|)^{-\delta}}
{(1+2^j|y-x_{Q_{jk}}|)^{R-L}}\,dy,
\end{eqnarray}
where $D_0:=\{x\in\rn:\ |x|\le1\}$ and, for all $v\in\nn$,
 $$D_v:=\{x\in\rn:\ 2^{v-1}<|x|\le 2^{v}\}.$$
For all $v\in\zz_+$, $y\in D_v$ and $j\in\zz_+$, let
$W_0^{y,j}:=\{k\in\zz^n:\ 2^{j}|y-x_{Q_{jk}}|\le 1\}$
and, for $i\in\nn$,
$$W_i^{y,j}:=\{k\in\zz^n:\ 2^{i-1}<2^{j}|y-x_{Q_{jk}}|\le 2^{i}\}.$$
Then we have
\begin{eqnarray*}
H(v,j,y)
&:=&\sum_{k\in\zz^n}|t_{Q_{jk}}||Q_{jk}|^{-\frac12}
(1+2^j|y-2^{-j}k|)^{-(R-L)}\noz\\
&\sim&\sum_{i=0}^\fz\sum_{ k\in W_i^{y,j}}2^{-i(R-L)}
|t_{Q_{jk}}||Q_{jk}|^{-\frac12}\noz\\
&\sim&\sum_{i=0}^\fz2^{-i(R-L)}
\int_{\cup_{\wz k\in W_i^{y,j}}Q_{j\wz k}}2^{jn}
\lf[\sum_{k\in\zz^n}|t_{Q_{jk}}|\wz \chi_{Q_{jk}}(z)\r]\,dz.
\end{eqnarray*}
Observe that, if $z\in\cup_{\wz k\in W_i^{y,j}}Q_{j\wz k}$,
then $z\in Q_{j\wz k_0}$ for some
$\wz k_0\in W_i^{y,j}$ and, for $y\in D_v$,
$1+2^j|y-z|\sim 1+2^i$; moreover,
\begin{eqnarray*}
|z|\le|z-x_{Q{j\wz k_0}}|+|x_{Q{j\wz k_0}}-y|+|y|
\ls2^{-j}+2^{-j+i}+2^{v}\ls 2^{i+v},
\end{eqnarray*}
which implies that
$$\bigcup_{\wz k\in W_i^{y,j}}Q_{j\wz k}
\subset Q(0,2^{i+v+c_0})$$ with some positive constant $c_0\in\nn$.
From this, we deduce that, for all $a\in(n,\fz)$, $v,\,j\in\zz_+$ and $y\in\rn$,
\begin{eqnarray*}
H(v,j,y)
&\sim&\sum_{i=0}^\fz2^{-i(R-L-a)}
\int_{\cup_{\wz k\in W_i^{y,j}}Q_{j\wz k}}\frac{2^{jn}}{(1+2^{j}|y-z|)^a}\\
&&\hs\hs\times\lf[\sum_{k\in\zz^n}|t_{Q_{jk}}|\wz \chi_{Q_{jk}}(z)
\chi_{Q(0,2^{i+v+c_0})}(z)\r]\,dz\\
&\ls& \sum_{i=0}^\fz2^{-i(R-L-a)}
\eta_{j,a}\ast\lf(\sum_{k\in\zz^n}|t_{Q_{jk}}|\wz \chi_{Q_{jk}}
\chi_{Q(0,2^{i+v+c_0})}\r)(y),
\end{eqnarray*}
which, combined with \eqref{mdc2}, implies that
\begin{eqnarray}\label{mdc-x}
&&\lf|\int_\rn\sum_{{k\in\zz^n,\,|k|\le \Lambda}}t_{Q_{jk}}
m_{Q_{jk}}(y)h(y)\,dy\r|\noz\\
&&\hs\ls2^{-jL}\sum_{v=0}^\fz
\sum_{i=0}^\fz(1+2^{v})^{-\delta_0}2^{-i(R-L-a)}\noz\\
&&\hs\hs\hs\times\int_{D_v}\eta_{j,a}
\ast\lf(\sum_{k\in\zz^n}|\wz t_{Q_{jk}}|\chi_{Q_{jk}}
\chi_{Q(0,2^{i+v+c_0})}\r)(y)(1+|y|)^{-\delta+\delta_0}\,dy,
\end{eqnarray}
where $\delta_0\in(0,\fz)$ is determined later.

By \eqref{md7}, we find that there exists $r\in(0,\min\{1,p_-,q_-\})$
such that $$s_-+\frac n{p_-}(r-1)>-L.$$
Let, for all $x\in\rn$,
$\wz p(x):=p(x)/r$, $(\wz p(x))^\ast:=\frac{\wz p(x)}{\wz p(x)-1}$
and $\wz s$ be a measurable function on $\rn$
such that, for all $x\in\rn$,
$$s(x)-\frac n{p(x)}=\wz s(x)-\frac n{\wz p(x)}.$$
Then
\begin{eqnarray*}
\wz s_-&=&\inf_{x\in\rn}\lf\{s(x)+\frac{n(r-1)}{p(x)}\r\}
\ge\inf_{x\in\rn}[s(x)]+\inf_{x\in\rn}\lf[\frac{n(r-1)}{p(x)}\r]\\
&=& s_-+\frac n{p_-}(r-1)>-L.
\end{eqnarray*}
Choosing $\delta\in(0,\fz)$ and $\delta_0\in(\max\{0,\log_2\wz c_1\},\fz)$ such that
$\delta\in(n(1-r)/p_++\delta_0,\fz)$, by the H\"older inequality
in Remark \ref{r-vlpp}(iii), \eqref{mdc-x},
Lemma \ref{l-estimate2}, Remark \ref{r-lattice}(ii)
and Proposition \ref{p-se1}, we conclude that
\begin{eqnarray*}
&&\lf|\int_\rn\sum_{{k\in\zz^n,\,|k|\le \Lambda}}t_{Q_{jk}}
m_{Q_{jk}}(y)h(y)\,dy\r|\\
&&\hs\ls2^{-j(L+\wz s_-)}\sum_{v=0}^\fz2^{-v\delta_0}
\sum_{i=0}^\fz2^{-i(R-L-a)}
\lf\|(1+|\cdot|)^{-\delta+\delta_0}\r\|_{L^{(\wz p(\cdot))^\ast}(\rn)}\\
&&\hs\hs\hs\times\lf\|\eta_{j,a}
\ast\lf[\sum_{k\in\zz^n}2^{j\wz s(\cdot)}|t_{Q_{jk}}|\wz \chi_{Q_{jk}}
\chi_{Q(0,2^{i+v+c_0})}\r]\r\|_{L^{\wz p(\cdot)}(\rn)}\\
&&\hs\ls2^{-j(L+\wz s_-)}\sum_{v=0}^\fz2^{-v\delta_0}
\sum_{i=0}^\fz2^{-i(R-L-a)}\lf\|\sum_{k\in\zz^n}2^{j\wz s(\cdot)}
|t_{Q_{jk}}|\wz \chi_{Q_{jk}}
\r\|_{L^{\wz p(\cdot)}(Q(0,2^{i+v+c_0}))}\\
&&\hs\ls2^{-j(L+\wz s_-)}\sum_{v=0}^\fz2^{-v\delta_0}
\sum_{i=0}^\fz2^{-i(R-L-a)} \phi(Q(0,2^{i+v+c_0}))
\|t\|_{f_{\wz p(\cdot),q(\cdot)}^{\wz s(\cdot),\phi}(\rn)}\\
&&\hs\ls2^{-j(L+\wz s_-)}\sum_{v=0}^\fz2^{-v(\delta_0-\log_2\wz c_1)}
\sum_{i=0}^\fz2^{-i(R-L-a-\log_2\wz c_1)}
\|t\|_{\tlve}\\
&&\hs\ls2^{-j(L+\wz s_-)}\|t\|_{\tlve},
\end{eqnarray*}
where $R\in(0,\fz)$ is chosen large enough,
which, together with $L>-\wz s_-$, implies that \eqref{3.1x}
exists in $\cs'(\rn)$ and
$|\la f,h\ra|\ls\|t\|_{\tlve}$.

\emph{Step 2}) We prove that
$\|f\|_{\btlve}\ls\|t\|_{\tlve}.$
Let $P\in\cq$ be a given dyadic cube and
$r\in(0,\min\{1,p_-,q_-\})$ such that
$L>n/r-n-s_-$. Then, by Remark \ref{r-vlpp}(i), we find that
\begin{eqnarray*}
&&\frac1{\phi(P)}\lf\|\lf\{\sum_{j=(j_P\vee 0)}^\fz
\lf[2^{js(\cdot)}|\vz_j\ast f|\r]^{q(\cdot)}\r\}^{\frac 1{q(\cdot)}}
\r\|_{L^{p(\cdot)}(P)}\\
&&\hs\ls\frac1{\phi(P)}\lf\|\lf\{\sum_{j=(j_P\vee 0)}^\fz
\lf[2^{js(\cdot)r}\sum_{v=0}^{(j_P\vee 0)-1}\sum_{\ell(Q)=2^{-v}}
|t_Q|^r|\vz_j\ast m_Q|^r\r]^{\frac{q(\cdot)}{r}}
\r\}^{\frac r{q(\cdot)}}\r\|_{L^{\frac{p(\cdot)}r}(P)}^{\frac1r}\\
&&\hs\hs+\frac1{\phi(P)}\lf\|\lf\{\sum_{j=(j_P\vee 0)}^\fz
\lf[2^{js(\cdot)r}\sum_{v=(j_p\vee 0)}^{\fz}\sum_{\ell(Q)=2^{-v}}
|t_Q|^r|\vz_j\ast m_Q|^r\r]^{\frac{q(\cdot)}{r}}
\r\}^{\frac r{q(\cdot)}}\r\|_{L^{\frac{p(\cdot)}r}(P)}^{\frac1r}\\
&&\hs=:{\rm I}_1+{\rm I}_2,
\end{eqnarray*}
here $\sum_{v=0}^{(j_P\vee 0)-1}\cdots=0$ if $j_P\le0$.

Observe that I$_1=0$ if $j_P\le0$.
Thus, to estimate I$_1$, we only need to assume $j_P\in\nn$.
By \cite[Lemma 3.3]{fj85} (see also \cite[Lemma 3.5]{kempka10}), we find that,
for all $Q:=Q_{vk}\in\cq^\ast$ with $v\le j$ and $x\in\rn$,
$$|\vz_j\ast m_Q(x)|\ls 2^{vn/2}2^{(v-j)K}(1+2^v|x-x_Q|)^{-R},$$
which, combined with \eqref{simple-ineq},
implies that
\begin{eqnarray}\label{md6}
{\rm I}_1
&\ls&\frac1{\phi(P)}\lf\|\lf\{\sum_{j=j_P}^\fz\sum_{v=0}^{j_P-1}
2^{js(\cdot)r}\sum_{k\in\zz^n}|t_{Q_{vk}}|^r|Q_{vk}|^{-\frac r2}\r.\r.\noz\\
&&\hs\hs\times\lf.2^{(v-j)Kr}(1+2^v|\cdot-x_{Q_{vk}}|)^{-Rr}\Bigg\}
\r\|_{L^\frac{p(\cdot)}r(P)}^\frac1r.
\end{eqnarray}

We claim that, for all
$v,\ j\in\zz_+$ and $x\in P$,
\begin{eqnarray*}
J(v,j,x,P):=&&2^{js(x)r}\sum_{k\in\zz^n}|t_{Q_{vk}}|^r|Q_{vk}|^{-\frac r2}
2^{(v-j)Kr}(1+2^v|x-x_{Q_{vk}}|)^{-Rr}\\
\ls&&2^{(v-j)(K-s_+)r}\sum_{i=0}^\fz2^{-i(M-a-\vez/r)r}\noz\\
&&\times\eta_{v,ar}\ast\lf(\lf[\sum_{k\in\Omega_i^{x,v}}
|t_{Q_{vk}}|2^{vs(\cdot)}
\wz \chi_{Q_{vk}}\chi_{Q(c_P,2^{i-v+c_0})}\r]^r\r)(x),\noz
\end{eqnarray*}
where $a\in(n/r,\fz)$, $\vez\in[C_{\log}(s),\fz)$, $c_P$ is the center of $P$,
$c_0\in\nn$ is a positive constant independent of $x,\ P,\ i,\ v,\ k$,
$\Omega_0^{x,v}:=\lf\{k\in\zz^n:\ 2^v|x-x_{Q_{vk}}|\le 1\r\}$
and, for all $i\in\nn$,
$$\Omega_i^{x,v}:=\lf\{k\in\zz^n:\ 2^{i-1}<2^v|x-x_{Q_{vk}}|\le 2^i\r\}.$$
Indeed, it is easy to see that
\begin{eqnarray}\label{md2}
J(v,j,x,P)
&&\ls2^{js(x)r}2^{(v-j)Kr}\sum_{i=0}^\fz\sum_{k\in\Omega_i^{x,v}}2^{-Rri}
|t_{Q_{vk}}|^r|Q_{vk}|^{-\frac r2}\noz\\
&&\sim2^{[js(x)+(v-j)K]r}\noz\\
&&\hs\times\sum_{i=0}^\fz
\lf\{2^{-Rri+vn}\int_{\cup_{\wz k\in\Omega_i^{x,v}}Q_{v\wz k}}
\lf[\sum_{k\in\Omega_i^{x,v}}|t_{Q_{vk}}|
\wz \chi_{Q_{vk}}(y)\r]^r\,dy\r\}.\qquad
\end{eqnarray}
Observe that, if $y\in \cup_{\wz k\in\Omega_i^{x,v}}Q_{v\wz k}$, then there
exists a $\wz k_0\in\Omega_i^{x,v}$ such that $y\in Q_{v\wz k_0}$ and
$1+2^v|x-y|\sim1+2^i;$ moreover, since $v\le j_P$, it follows that
\begin{eqnarray*}
|y-c_P|
&\le& |y-x_{Q_{v\wz k_0}}|+|x-x_{Q_{v\wz k_0}}|+|x-c_P|\\
&\ls& 2^{-v}+2^{i-v}+2^{-j_P}\ls 2^{i-v},
\end{eqnarray*}
which implies that
$\cup_{\wz k\in\Omega_i^{x,v}}Q_{v\wz k}\subset Q(c_P,2^{i-v+c_0})$ for some
constant $c_0\in\nn$. From this, \eqref{md2} and Lemma \ref{l-eta},
 we deduce that, for all $a\in(n/r,\fz)$, $v,\,j\in\zz_+$ and $x\in P$,
\begin{eqnarray*}
J(v,j,x,P)
&&\ls 2^{js(x)r}2^{(v-j)Kr}\sum_{i=0}^\fz2^{(a+\vez/r-R)ri}
\int_{\cup_{\wz k\in\Omega_i^{x,v}}Q_{v\wz k}}
\frac{2^{vn}}{(1+2^v|x-y|)^{ar+\vez}}\\
&&\hs\times\lf[\sum_{k\in\Omega_i^{x,v}}|t_{Q_{vk}}|\wz \chi_{Q_{vk}}
\chi_{Q(c_P,2^{i-v+c_0})}(y)\r]^r\,dy\\
&&\ls2^{(v-j)(K-s_+)r}\sum_{i=0}^\fz2^{(a+\vez/r-R)ri}\\
&&\hs\times\eta_{v,ar}\ast\lf(\lf[\sum_{k\in\Omega_i^{x,v}}2^{vs(\cdot)}|t_{Q_{vk}}|\wz \chi_{Q_{vk}}
\chi_{Q(c_P,2^{i-v+c_0})}\r]^r\r)(x),
\end{eqnarray*}
which implies that the claim holds true.

By this claim, \eqref{md6} and
Remark \ref{r-vlpp}(i), we conclude that
\begin{eqnarray*}
{\rm I}_1
&\ls&\frac1{\phi(P)}
\lf\{\sum_{j=j_P}^\fz\sum_{v=0}^{j_P-1}2^{(v-j)(K-s_+)r}
\sum_{i=0}^\fz2^{(a+\vez/r-M)ri}\r.\\
&&\hs\hs\lf.\times\lf\|\eta_{v,ar}\ast
\lf(\lf[\sum_{k\in\Omega_i^{\cdot,v}}| t_{Q_{vk}}|2^{vs(\cdot)}\wz\chi_{Q_{vk}}
\chi_{B(c_P,2^{i-v+c_0})}\r]^r\r)\r\|_{L^{\frac{p(\cdot)}r}(P)}\r\}^\frac1r,
\end{eqnarray*}
which, together with Lemma \ref{l-estimate2}, $j_P\in\nn$ and
Remark \ref{r-lattice}(ii), further implies that
\begin{eqnarray*}
{\rm I}_1
&\ls&\frac1{\phi(P)}
\lf\{\sum_{j=j_P}^\fz\sum_{v=0}^{j_P}2^{(v-j)(K-s_+)r}
\sum_{i=0}^\fz2^{(a+\vez/r-R)ri}\r.\\
&&\hs\lf.\times\lf\|
\sum_{k\in\zz^n}|t_{Q_{vk}}|2^{vs(\cdot)}\wz \chi_{Q_{vk}}
\r\|_{L^{{p(\cdot)}}(Q(c_P,2^{i-v+c_0}))}^r\r\}^\frac1r\\
&\ls&\|t\|_{\tlve}
\lf\{\sum_{j=j_P}^\fz\sum_{v=0}^{j_P}2^{(v-j)(K-s_+)r}\r.\\
&&\hs\times\lf.\sum_{i=0}^\fz2^{(a+\vez/r-R)ri}
\frac{[\phi(Q(c_P,2^{i-v+c_0}))]^r}
{[\phi(P)]^r}\r\}^\frac1r.
\end{eqnarray*}
From this, $K\in(s_++\max\{0,\log_2\wz c_1\},\fz)$ and the fact that, when $v\le j_P$,
\begin{eqnarray*}
\phi(Q(c_P,2^{i-v+c_0}))
&\ls& (\wz c_1)^i\phi(Q(c_P,2^{-v}))
\ls (\wz c_1)^{i+j_P-v}\phi(P)\\
&\sim& 2^{(i+j_P-v)\log_2\wz c_1}\phi(P),
\end{eqnarray*}
 we deduce that
\begin{eqnarray}\label{mdi1}
{\rm I}_1
&\ls&\|t\|_{\tlve}
\lf\{2^{j_P\log_2\wz c_1}\sum_{j=j_P}^\fz2^{-j(K-s_+)r}
\sum_{v=0}^{j_P}2^{v(K-s_+-\log_2\wz c_1)r}\r.\noz\\
&&\hs\hs\times\lf.
\sum_{i=0}^\fz2^{-i(R-a-\vez/r-\log_2\wz c_1)r}\r\}^\frac1r
\ls\|t\|_{\tlve},
\end{eqnarray}
where $R\in(0,\fz)$ is chosen such that $R>a+\vez/r+\log_2\wz c_1$.

We now estimate I$_2$.
By applying \cite[Lemmas A.2 and A.5]{dhr09} and
an argument similar to that used in the proof of \cite[Lemma 6.3]{dhr09},
we see that, for all $j\in\zz_+$, $Q:=Q_{vk}\in\cq^\ast$ and $x\in\rn$,
$$
|\vz_j\ast m_Q(x)|\ls 2^{-\beta(j,v)}|Q|^{-1/2}
(\eta_{j,R}\ast\eta_{v,R}\ast\chi_{Q})(x),
$$
where
$$\beta(j,v):=K\max\{j-v,0\}+L\max\{v-j,0\}.$$
Let $$M\in (n/r+\log_2(c_1\wz c_1),\fz)$$ and
$\vez\in[C_{\log}(s),\fz)$
be such that $R=2M+\vez/r$.
Thus, we have
\begin{eqnarray}\label{md3}
{\rm I}_2
&\ls&\frac1{\phi(P)}\lf\|\lf\{\sum_{j=(j_P\vee 0)}^\fz
\lf[\sum_{v=(j_P\vee 0)}^\fz\sum_{\ell(Q)=2^{-v}}
2^{js(\cdot)r}2^{-\beta(j,v)r}\r.\r.\r.\noz\\
&&\hs\hs\times\lf.\lf.|t_Q|^r|Q|^{-\frac r2}
(\eta_{j,2M+\vez/r}\ast\eta_{v,2M+\vez/r}\ast\chi_Q)^r\Bigg]^{\frac{q(\cdot)}r}
\r\}^{\frac r{q(\cdot)}}\r\|_{L^{p(\cdot)}(P)}^\frac1r.
\end{eqnarray}
By \cite[Lemma A.4]{dhr09}, we find that, for all
 $v\in\zz_+$ and $\ell(Q)=2^{-v}$,
\begin{eqnarray*}
&&2^{js(\cdot)r-\beta(j,v)r}(\eta_{j,2M+\vez/r}
\ast\eta_{v,2M+\vez/r}\ast\chi_Q)^r\\
&&\hs\ls2^{js(\cdot)r-\beta(j,v)r}2^{n\max\{v-j,0\}(1-r)}
\eta_{j,2Mr+\vez}\ast\eta_{v,2Mr+\vez}\ast\chi_Q\\
&&\hs\sim2^{vs(\cdot)r}2^{-(K-s_+)r\max\{j-v,0\}}\\
&&\hs\hs\hs\hs\times2^{-(L-\frac nr+n+s_-)r\max\{v-j,0\}}
\eta_{j,2Mr+\vez}\ast\eta_{v,2Mr+\vez}\ast\chi_Q,
\end{eqnarray*}
which, combined with \eqref{md3} and Lemma \ref{l-eta}, implies that
\begin{eqnarray}\label{md5}
{\rm I}_2
&\ls&\frac1{\phi(P)}
\lf\|\lf\{\sum_{j=(j_P\vee 0)}^\fz\lf(\eta_{j,2Mr}\ast
\lf[\sum_{v=(j_P\vee 0)}^\fz
\sum_{\ell(Q)=2^{-v}}|t_Q|^r|Q|^{-\frac r2}\r.\r.\r.\r.\noz\\
&&\hs\hs\times\lf.\lf.\lf.2^{vs(\cdot)r-\vez(j,v)r}
\eta_{v,2Mr+\vez}\ast\chi_Q\Bigg]\r)^{\frac{q(\cdot)}r}\r\}
^{\frac r{q(\cdot)}}\r\|_{L^{\frac{p(\cdot)}r}(P)}^{\frac1r}\noz\\
&\ls&\frac1{\phi(P)}\lf\|\lf\{\sum_{j=(j_P\vee 0)}^\fz
\lf[\sum_{l\in\zz^n}\int_{P+l\ell(P)}
\frac{2^{jn}}{(1+2^j|\cdot-y|)^{2Mr}}
\lf(\sum_{v=(j_P\vee 0)}^\fz2^{vs(\cdot)-\vez(j,v)}
\r.\r.\r.\r.\noz\\
&&\hs\hs\times\lf.\lf.\lf.\lf.\sum_{\ell(Q)=2^{-v}}
\lf[\frac{|t_Q|}{|Q|^{\frac 12}}\r]^r
\eta_{v,2Mr+\vez}\ast\chi_Q\r)(y)\,dy\r]
^{\frac{q(\cdot)}r}\r\}
^{\frac r{q(\cdot)}}\r\|_{L^{\frac{p(\cdot)}r}(P)}^{\frac1r},
\end{eqnarray}
where
$$\vez(j,v):=(K-s_+)\max\{j-v,0\}+(L-n/r+n+s_-)\max\{v-j,0\}.$$
From this, the fact that, when $j\ge j_P$, $l\in\zz^n$, $x\in P$
 and $y\in P+l\ell(P)$,
$$1+2^j|x-y|\ge 1+2^{j_P}|x-y|\sim 1+|l|,$$
the Minkowski inequality, Lemma \ref{l-estimate2} and
Remark \ref{r-vlpp}(i), we further deduce that
\begin{eqnarray}\label{md4}
{\rm I}_2
&\ls&\frac1{\phi(P)}
\lf\|\lf\{\sum_{j=(j_P\vee 0)}^\fz\lf[
\sum_{l\in\zz^n}(1+|l|)^{-Mr}\eta_{j,Mr}\ast
\lf(\lf[\sum_{v=(j_P\vee 0)}^\fz\sum_{\ell(Q)=2^{-v}}
|t_Q|^r\r.\r.\r.\r.\r.\noz\\
&&\hs\hs\times\lf.\lf.\lf.\lf.|Q|^{-\frac r2}2^{vs(\cdot)r-\vez(j,v)r}
\eta_{v,2Mr+\vez}\ast\chi_Q\Bigg]\chi_{P+l\ell(P)}\r)\r]^{\frac{q(\cdot)}r}\r\}
^{\frac r{q(\cdot)}}\r\|_{L^{\frac{p(\cdot)}r}(P)}^{\frac1r}\noz\\
&\ls&\frac1{\phi(P)}
\lf\|\sum_{l\in\zz^n}(1+|l|)^{-Mr}
\lf\{\sum_{j=(j_P\vee 0)}^\fz\lf[\eta_{j,Mr}\ast
\lf(\lf[\sum_{v=(j_P\vee 0)}^\fz\sum_{\ell(Q)=2^{-v}}
|t_Q|^r\r.\r.\r.\r.\r.\noz\\
&&\hs\hs\hs\times\lf.\lf.\lf.\lf.|Q|^{-\frac r2}2^{vs(\cdot)r-\vez(j,v)r}
\eta_{v,2Mr+\vez}\ast\chi_Q\Bigg]
\chi_{P+l\ell(P)}\r)\r]^{\frac{q(\cdot)}r}\r\}
^{\frac r{q(\cdot)}}\r\|_{L^{\frac{p(\cdot)}r}(P)}^{\frac1r}\noz\\
&\ls&\frac1{\phi(P)}
\lf\{\sum_{l\in\zz^n}(1+|l|)^{-Mr}\lf\|\lf[\sum_{j=(j_P\vee 0)}^\fz
\lf(\sum_{v=(j_P\vee 0)}^\fz\sum_{\ell(Q)=2^{-v}}
|t_Q|^r|Q|^{-\frac r2}\r.\r.\r.\r.\noz\\
&&\hs\hs\hs\times\lf.\lf.\lf.2^{vs(\cdot)r-\vez(j,v)r}
\eta_{v,2Mr+\vez}\ast\chi_Q\Bigg)^{\frac{q(\cdot)}r}
\r]^{\frac r{q(\cdot)}}\r\|_{L^{\frac{p(\cdot)}r}(P+l\ell(P))}\r\}^{\frac1r}.
\end{eqnarray}
By the H\"older inequality,
$$K\in(s_++\max\{0,\log_2\wz c_1\},\fz),$$
$L\in(n/r-n-s_-,\fz)$
and the fact that $0<q_-\le q_+<\fz$, we see that
\begin{eqnarray*}
&&\lf[\sum_{j=(j_P\vee 0)}^\fz
\lf(\sum_{v=(j_P\vee 0)}^\fz\sum_{\ell(Q)=2^{-v}}
|t_Q|^r|Q|^{-\frac r2}2^{vs(\cdot)r-\vez(j,v)r}
\eta_{v,2Mr+\vez}\ast\chi_Q\r)^{\frac{q(\cdot)}r}
\r]^{\frac r{q(\cdot)}}\\
&&\hs\ls\lf\{\sum_{j=(j_P\vee 0)}^\fz\sum_{v=(j_P\vee 0)}^\fz
2^{-\vez(j,v)r}\lf[\sum_{\ell(Q)=2^{-v}}
|t_Q|^r|Q|^{-\frac r2}2^{vs(\cdot)r}
\eta_{v,2Mr+\vez}\ast\chi_Q\r]^{\frac{q(\cdot)}r}\r\}^{\frac r{q(\cdot)}}\\
&&\hs\ls\lf\{\sum_{v=(j_P\vee 0)}^\fz\lf[\sum_{\ell(Q)=2^{-v}}
|t_Q|^r|Q|^{-\frac r2}2^{vs(\cdot)r}
\eta_{v,2Mr+\vez}\ast\chi_Q\r]^{\frac{q(\cdot)}r}\r\}^{\frac r{q(\cdot)}},
\end{eqnarray*}
which, together with \eqref{md4} and some arguments similar to
those used in the proofs of \eqref{md5} and \eqref{md4}, implies that
\begin{eqnarray*}
{\rm I}_2
&\ls&\frac1{\phi(P)}
\lf[\sum_{l\in\zz^n}(1+|l|)^{-Mr}\lf\|\lf\{\sum_{v=(j_P\vee 0)}^\fz
\lf[\sum_{\ell(Q)=2^{-v}}|t_Q|^r|Q|^{-\frac r2}\r.\r.\r.\r.\\
&&\hs\lf.\lf.\lf.\times2^{vs(\cdot)r}
\eta_{v,2Mr+\vez}\ast\chi_Q\Bigg]^{\frac{q(\cdot)}r}\r\}^{\frac r{q(\cdot)}}
\r\|_{L^{\frac{p(\cdot)}r}(P+l\ell(P))}\r]^{\frac1r}\\
&\ls&\frac1{\phi(P)}\lf\{\sum_{l\in\zz^n}(1+|l|)^{-Mr}
\lf\|\lf(\sum_{v=(j_P\vee 0)}^\fz\Bigg[\sum_{k\in\zz^n}(1+|k|)^{-Mr}\r.\r.\r.\\
&&\hs\times\eta_{v,Mr}\ast\lf(\lf\{\sum_{\ell(Q)=2^{-v}}
|t_Q|2^{vs(\cdot)}\wz\chi_Q\r\}^r\r.\\
&&\hs\lf.\lf.\lf.\lf.\times\chi_{P+(l+k)
\ell(P)}\Bigg)\r]^\frac{q(\cdot)}{r}
\r)^\frac r{q(\cdot)}\r\|_{L^{\frac{p(\cdot)}r}(P+l\ell(P))}\r\}^\frac1r.
\end{eqnarray*}
From this, the Minkowski inequality, Remark \ref{r-vlpp}(i),
Lemmas \ref{l-estimate2} and \ref{l-esti-cube}(ii),
we deduce that
\begin{eqnarray}\label{mdi2}
{\rm I}_2
&\ls&\frac1{\phi(P)}
\Bigg[\sum_{k,l\in\zz^n}(1+|l|)^{-Mr}(1+|k|)^{-Mr}\noz\\
&&\hs\hs\times\lf.\lf\|\lf\{\sum_{v=(j_P\vee 0)}^\fz
\lf[\sum_{\ell(Q)=2^{-v}}2^{vs(\cdot)}|t_Q|\wz \chi_Q
\r]^{q(\cdot)}\r\}^{\frac1{q(\cdot)}}\r\|_{L^{p(\cdot)}(P+(l+k)\ell(P))}^r
\r]^\frac1r\noz\\
&\ls&\|t\|_{\tlve}
\lf\{\sum_{k,l\in\zz^n}(1+|l|)^{-Mr}(1+|k|)^{-Mr}
\frac{[\phi(P+(l+k)\ell(P))]^r}{[\phi(P)]^r}\r\}^\frac1r\noz\\
&\ls&\|t\|_{\tlve}
\lf\{\sum_{k,l\in\zz^n}(1+|l|)^{-Mr+r\log_2(c_1\wz c_1)}
(1+|k|)^{-Mr+r\log_2(c_1\wz c_1)}\r\}^\frac12\noz\\
&\sim&\|t\|_{\tlve},
\end{eqnarray}
where $M$ is chosen large enough.

Finally, combining \eqref{mdi1} and \eqref{mdi2}, we conclude that
\begin{eqnarray*}
\|f\|_{\btlve}
\ls\sup_{P\in\cq}({\rm I}_1+{\rm I}_2)\ls\|t\|_{\tlve},
\end{eqnarray*}
which completes the proof of Theorem \ref{t-md}.
\end{proof}

Next we establish the Peetre maximal function characterization of $\btlve$.

Let $(\vz,\Phi)$ be a pair of admissible functions.
Recall that the \emph{Peetre maximal function} of $f\in\cs'(\rn)$
is defined by setting, for all $j\in\zz_+$, $a\in(0,\fz)$ and $x\in\rn$,
$$(\vz_j^\ast f)_a(x):=\sup_{y\in\rn}
\frac{|\vz_j\ast f(x+y)|}{(1+2^j|y|)^a},$$
where $\vz_0$ is replaced by $\Phi$.
The following Lemma \ref{l-pmf} comes from \cite[(2.48) and (2.66)]{ut12}.

\begin{lemma}\label{l-pmf}
Let $(\vz,\Phi)$ be a pair of admissible functions, $f\in\cs'(\rn)$ and $N\in\nn$.
Then, for all $t\in[1,2]$, $a\in(0,N]$, $\ell\in\zz_+$ and $x\in\rn$,
\begin{equation*}
\lf[(\vz_{2^{-\ell}t}^\ast f)_a(x)\r]^r
\le C\sum_{v=0}^\fz2^{-vNr}2^{(v+\ell)n}
\int_\rn\frac{|(\vz_{v+\ell})_t\ast f(y)|^r}
{(1+2^\ell|x-y|)^{ar}}\,dy,
\end{equation*}
where $r$ is an arbitrary fixed positive number,
$\vz_0$ is replaced by $\Phi$ and $C$ is a positive constant independent of
$\vz,\ \Phi$, $f$, $x,\ \ell$ and $t$.
\end{lemma}

\begin{theorem}\label{t-pmfc}
Let $p$, $q$, $s$ and $\phi$ be as in Definition \ref{d-tl}.
Let
\begin{equation}\label{peetre1}
a\in\lf(\frac n{\min\{p_-,q_-\}}+\log_2\wz c_1+C_{\log}(s),\fz\r).
\end{equation}
Then
$f\in\btlve$ if and only if $f\in\cs'(\rn)$ and $\|f\|_{\btlve}^\ast<\fz$,
where
$$\|f\|_{\btlve}^\ast
:=\sup_{P\in\cq}\frac{1}{\phi(P)}
\lf\|\lf\{\sum_{j=(j_P\vee 0)}^\fz\lf[2^{js(\cdot)}(\vz_j^\ast f)_a\r]^{q(\cdot)}
\r\}^{\frac 1{q(\cdot)}}\r\|_{L^{p(\cdot)}(P)}.$$
\end{theorem}
\begin{proof}
Observe that, by definitions, we have
$\|f\|_{\btlve}\le\|f\|_{\btlve}^\ast$. Next we show that
$\|f\|_{\btlve}^\ast\ls\|f\|_{\btlve}$ for all $f\in\btlve$.

By \eqref{peetre1}, we find that there exist
$r\in(0,\min\{p_-,q_-\})$ and $\vez\in(\log_2\wz c_1,\fz)$
such that $a>n/r+\vez+C_{\log}(s)$.
For any given dyadic cube $P\subset\rn$,
by Lemma \ref{l-pmf},
we see that
\begin{eqnarray}\label{qc1}
{\rm J}_P
:=&&\frac1{\phi(P)}\lf\|\lf\{\sum_{j=(j_P\vee 0)}^\fz
\lf[2^{js(\cdot)}(\vz_j^\ast f)_a\r]^{q(\cdot)}\r\}^{\frac1{q(\cdot)}}\r\|
_{L^{p(\cdot)}(P)}\noz\\
\ls&&\frac1{\phi(P)}
\lf\|\lf\{\sum_{j=(j_P\vee 0)}^\fz
\lf[2^{js(\cdot)}\lf(\sum_{v=0}^\fz2^{-vNr}2^{(v+j)n}\r.\r.\r.\r.\noz\\
&&\times\lf.\lf.\lf.\lf.\int_\rn\frac{|\vz_{v+j}
\ast f(y)|^r}{(1+2^j|\cdot-y|)^{ar}}\,dy\r)^\frac1r\r]
^{q(\cdot)}\r\}^\frac1{q(\cdot)}\r\|_{L^{p(\cdot)}(P)},
\end{eqnarray}
where $N\in\nn\cap[a,\fz)$ is determined later.
Notice that $j\ge (j_P\vee0)$ and, for all $x\in P$ and
$y\in (2^{k+1}\sqrt nP)\backslash(2^{k}\sqrt nP)=:D_{k,P}$
with $k\in\nn$,
$$1+2^j|x-y|\gs2^j2^{-j_P}2^k.$$
Then it follows that, for all $x\in P$,
\begin{eqnarray*}
&&\int_\rn\frac{|\vz_{v+j}\ast f(y)|^r}{(1+2^j|x-y|)^{ar}}\,dy\\
&&\hs=\lf\{\int_{2\sqrt nP}+\sum_{k=1}^\fz
\int_{D_{k,P}}\r\}\frac{|\vz_{v+j}
\ast f(y)|^r}{(1+2^j|x-y|)^{ar}}\,dy\\
&&\hs\ls 2^{-jn}\eta_{j,ar}\ast\lf(|\vz_{v+j}\ast f|^r\chi_{2\sqrt nP}\r)(x)\\
&&\hs\hs+2^{-j(\vez r+n)}2^{j_P\vez r}\sum_{k=1}^\fz2^{-k\vez r}
\eta_{j,(a-\vez)r}\ast\lf(|\vz_{v+j}\ast f|^r\chi_{D_{k,P}}\r)(x)\\
&&\hs=:{\rm I}_{P,1}+{\rm I}_{P,2},
\end{eqnarray*}
which implies that
\begin{eqnarray}\label{qc2}
{\rm J}_P&\ls&
\frac1{\phi(P)}\lf\|\lf\{\sum_{j=(j_P\vee 0)}^\fz
\lf[2^{js(\cdot)r}\sum_{v=0}2^{-vNr}2^{(v+j)n}{\rm I}_{P,1}\r]^\frac{q(\cdot)}r
\r\}^\frac1{q(\cdot)}\r\|_{L^{p(\cdot)}(P)}\noz\\
&&\hs\hs+\frac1{\phi(P)}\lf\|\lf\{\sum_{j=(j_P\vee 0)}^\fz
\lf[2^{js(\cdot)r}\sum_{v=0}2^{-vNr}2^{(v+j)n}{\rm I}_{P,2}\r]^\frac{q(\cdot)}r
\r\}^\frac1{q(\cdot)}\r\|_{L^{p(\cdot)}(P)}\noz\\
&=:&{\rm J}_{P,1}+{\rm J}_{P,2}.
\end{eqnarray}

For ${\rm J}_{P,1}$, by Lemmas \ref{l-eta} and \ref{l-estimate2},
the Minkowski inequality
and Remark \ref{r-vlpp}(i),
we find that
\begin{eqnarray}\label{qc3}
{\rm J}_{P,1}
&\ls&\frac1{\phi(P)}\lf\|\lf\{\sum_{j=(j_P\vee 0)}^\fz
\lf[\sum_{v=0}^\fz2^{(n-Nr)v}2^{js(\cdot)r}|\vz_{v+j}\ast f|^r\r]
^{\frac{q(\cdot)}{r}}\r\}^{\frac r{q(\cdot)}}
\r\|_{L^{\frac {p(\cdot)}r}(2\sqrt nP)}^\frac1r\noz\\
&\ls&\lf\{\sum_{v=0}^\fz2^{-v(Nr-n+s_-r)}\frac1{[\phi(P)]^r}\r.\noz\\
&&\hs\hs\times\lf.
\lf\|\lf\{\sum_{j=(j_p\vee0)}^\fz\lf[2^{(v+j)s(\cdot)}
|\vz_{v+j}\ast f|^r\r]^{q(\cdot)}\r\}^{\frac1{q(\cdot)}}
\r\|_{L^{p(\cdot)}(2\sqrt nP)}^r
\r\}^{\frac1r}\noz\\
&\ls&\lf\{\sum_{v=0}^\fz2^{-v(Nr-n+s_-r)}\r\}^\frac1r
\|f\|_{\btlve}\sim\|f\|_{\btlve},
\end{eqnarray}
where we used the condition (\textbf{S1}) of $\phi$
in the third inequality and $N\in\nn$ is chosen large enough
such that $N\in[a,\fz)\cap(\frac nr-s_-,\fz)$.

For ${\rm J}_{P,2}$, by an argument similar to the above, we find that
\begin{eqnarray}\label{3.17x}
{\rm J}_{P,2}&\ls&\Bigg\{\sum_{v=0}^\fz2^{-v(Nr-n+rs_-)}
\sum_{k=1}^\fz2^{-k\vez}\frac1{[\phi(P)]^r}\noz\\
&&\hs\hs\times\lf.\lf\|\lf\{\sum_{j=(j_p\vee0)}^\fz\lf[2^{(v+j)s(\cdot)}
|\vz_{v+j}\ast f|^r\r]^{q(\cdot)}\r\}^{\frac1{q(\cdot)}}
\r\|_{L^{p(\cdot)}(D_{k,P})}^r\r\}^\frac1r\noz\\
&\ls&\lf\{\sum_{v=0}^\fz2^{-v(Nr-n+rs_-)}
\sum_{k=1}^\fz2^{-k\vez r}
\frac{[\phi(2^{k+1+n}P)]^r}{[\phi(P)]^r}\r\}^\frac1r\|f\|_{\btlve}\noz\\
&\ls&\lf\{\sum_{k=1}^\fz2^{-k(\vez-\log_2\wz c_1)}\r\}^\frac1r\|f\|_{\btlve}
\sim \|f\|_{\btlve}.
\end{eqnarray}
Combining the estimates \eqref{qc1}, \eqref{qc2}, \eqref{qc3} and
\eqref{3.17x}, we conclude that
\begin{eqnarray*}
\|f\|_{\btlve}^\ast\le\sup_{P\in\cq}{\rm J}_P\ls
\sup_{P\in\cq}({\rm J}_{P,1}+{\rm J}_{P,2})\ls\|f\|_{\btlve},
\end{eqnarray*}
which completes the proof of Theorem \ref{t-pmfc}.
\end{proof}

As applications of Theorem \ref{t-pmfc}, we obtain two equivalent
quasi-norms of the space $\btlve$. To this end, for all $f\in\cs'(\rn)$,
let
$$\lf\|f\lf|\btlve\r.\r\|_1:=\sup_{P\in\cq}
\frac1{\phi(P)}\lf\|\lf\{\sum_{j=0}^\fz\lf[2^{js(\cdot)}
|\vz_j\ast f|\r]^{q(\cdot)}\r\}^{\frac1{q(\cdot)}}\r\|_{L^{p(\cdot)}(P)}$$
and
$$\lf\|f\lf|\btlve\r.\r\|_2:=\sup_{Q\in\cq}\sup_{x\in Q}|Q|^{-\frac{s(x)}n}
[\phi(Q)]^{-1}\|\chi_Q\|_{\vlp}|\vz_{j_Q}\ast f(x)|.$$

\begin{theorem}\label{t-sum}
Let $p,\ q,\ s$, $\phi$ be as in Definition \ref{d-tl}.
\begin{enumerate}
\item[{\rm (i)}] If $\wz c_1\in(0,2^{n/p_+})$, then $f\in\btlve$ if and only if
$f\in\cs'(\rn)$ and $\|f|\btlve\|_1<\fz$; moreover, there exists a positive constant $C$,
independent of $f$, such that
$$C^{-1}\|f\|_{\btlve}\le\lf\|f\lf|\btlve\r.\r\|_1\le C\|f\|_{\btlve}.$$

\item[{\rm (ii)}] If $c_1\in(0,2^{-n/p_-})$, then $f\in\btlve$ if and only if
$f\in\cs'(\rn)$ and $\|f|\btlve\|_2<\fz$; moreover, there exists a positive constant $C$,
independent of $f$, such that
$$C^{-1}\|f\|_{\btlve}\le\lf\|f\lf|\btlve\r.\r\|_2\le C\|f\|_{\btlve}.$$
\end{enumerate}
\end{theorem}
\begin{proof}
We first show (i). To this end, it suffices to show that
$$\|f|\btlve\|_1\ls\|f\|_{\btlve},$$ since the inverse inequality obviously
holds true by definitions.

Let $P\in\cq$ be a given dyadic cube. By Remark \ref{r-vlpp}(i), we see that
\begin{eqnarray*}
{\rm I}_P
:=&&\frac1{\phi(P)}\lf\|\lf\{\sum_{j=0}^\fz\lf[2^{js(\cdot)}
|\vz_j\ast f|\r]^{q(\cdot)}\r\}^{\frac1{q(\cdot)}}\r\|_{L^{p(\cdot)}(P)}\\
\ls&& \frac1{\phi(P)}\lf\|\lf\{\sum_{j=0}^{(j_P\vee0)-1}\lf[2^{js(\cdot)}
|\vz_j\ast f|\r]^{q(\cdot)}\r\}^{\frac1{q(\cdot)}}\r\|_{L^{p(\cdot)}(P)}\\
&&+\frac1{\phi(P)}\lf\|\lf\{\sum_{j=(j_P\vee0)}^{\fz}\lf[2^{js(\cdot)}
|\vz_j\ast f|\r]^{q(\cdot)}\r\}^{\frac1{q(\cdot)}}\r\|_{L^{p(\cdot)}(P)}
=:{\rm I}_{P,1}+{\rm I}_{P,2},
\end{eqnarray*}
where $\sum_{j=0}^{(j_P\vee0-1)}\cdots=0$ if $j_P\le0$.

Obviously, ${\rm I}_{P,2}\le \|f\|_{\btlve}$.

For ${\rm I}_{P,1}$,
we only need to estimate it in the case that $j_P>0$.
For any $j\in\nn$ with $j\le j_P-1$, there exists
a unique dyadic cube $P_j$ such that
$P\subset P_j$ and $\ell(P_j)=2^{-j}$.
Since $s\in C_{\rm loc}^{\log}(\rn)\cap L^{\fz}(\rn)$, it follows that,
for all $a\in(0,\fz)$, $x\in P$ and $y\in P_j$,
\begin{eqnarray*}
2^{js(x)}|\vz_j\ast f(x)|
&\ls& 2^{js(x)}(1+2^j|x-y|)^a(\vz_j^\ast f)_a(y)\\
&\ls&2^{j[s(x)-s(y)]}2^{js(y)}(\vz_j^\ast f)_a(y)\\
&\ls&2^{j\frac{C_{\log(s)}}{\log(e+1/|x-y|)}}2^{js(y)}(\vz_j^\ast f)_a(y)
\ls2^{js(y)}(\vz_j^\ast f)_a(y),
\end{eqnarray*}
which implies that, for all $x\in P$,
\begin{equation}\label{sum2}
2^{js(x)}|\vz_j\ast f(x)|\ls\inf_{y\in P_j}2^{js(y)}(\vz_j^\ast f)_a(y).
\end{equation}
Thus, choosing $r\in(0,\min\{1,p_-,q_-\})$ and $a$ as in Theorem \ref{t-pmfc},
by Theorem \ref{t-pmfc} and Remark \ref{r-vlpp}(i),
 we conclude that
\begin{eqnarray}\label{sum1}
{\rm I}_{P,1}
&\ls& \frac1{\phi(P)}\lf\|\lf\{\sum_{j=0}^{j_P-1}\lf[\inf_{y\in P_j}
2^{js(y)}(\vz_j^\ast f)_a(y)\r]^{q(\cdot)}\r\}
^{\frac1{q(\cdot)}}\r\|_{L^{p(\cdot)}(P)}\noz\\
&\ls&\frac1{\phi(P)}\lf\|\sum_{j=0}^{j_P-1}
\lf\|2^{js(\cdot)}(\vz_j^\ast f)_a\r\|_{L^{p(\cdot)}(P_j)}^{r}
\|\chi_{P_j}\|_{\vlp}^{-r}\r\|_{L^{\frac{p(\cdot)}r}(P)}^\frac1r\noz\\
&\ls&\|f\|_{\btlve}
\lf\{\sum_{j=0}^{j_P-1}\lf[\frac{\phi(P_j)}{\phi(P)}\r]^r
\lf[\frac{\|\chi_{P}\|_{\vlp}}{\|\chi_{P_j}\|_{\vlp}}\r]^r\r\}^\frac1r.
\end{eqnarray}
On the other hand, by \cite[Lemma 2.6]{zyl14}, we find that
$$\|\chi_{P_j}\|_{\vlp}\gs2^{-j\frac n{p_+}}2^{j_P\frac{n}{p_+}}
\|\chi_P\|_{\vlp}$$
and, by the condition (\textbf{S1}) of $\phi$, we see that
$\phi(P)\ge 2^{j\log_2\wz c_1}2^{-j_P\log_2\wz c_1}\phi(c_P,2^{-j}),$
which, together with \eqref{sum1} and the condition (\textbf{S2}) of $\phi$,
implies that
\begin{eqnarray*}
{\rm I}_{P,1}
&\ls&\|f\|_{\btlve}
\lf\{\sum_{j=0}^{j_P-1}2^{j(\frac n{p_+}-\log_2\wz c_1)r}
\lf[\frac{\phi(c_P,2^{-j})}{\phi(P_j)}\r]^r\r\}^{\frac1r}
2^{j_P(\log_2\wz c_1-\frac n{p_+})}\\
&\ls&\|f\|_{\btlve}
\lf\{\sum_{j=0}^{j_P-1}2^{j(\frac n{p_+}-\log_2\wz c_1)r}\r\}^{\frac1r}
2^{j_P(\log_2\wz c_1-\frac n{p_+})}\sim\|f\|_{\btlve},
\end{eqnarray*}
where we used the fact that $\wz c_1\in(0,2^{n/p_+})$ in the last inequality.
Therefore,
$$\lf\|f\lf|\btlve\r.\r\|_1=\sup_{P\in\cq}{\rm I}_P\ls\|f\|_{\btlve},$$
which completes the proof of (i).

Now we prove (ii). For all $Q\in\cq^\ast$, from \eqref{sum2} and
Theorem \ref{t-pmfc}, we deduce that, for all $x\in Q$,
\begin{eqnarray*}
&&[\phi(Q)]^{-1}\|\chi_Q\|_{\vlp}|Q|^{-\frac{s(x)}{n}}|\vz_{j_Q}\ast f(x)|\\
&&\hs\ls\frac{\|\chi_Q\|_{\vlp}}{\phi(Q)}\inf_{y\in Q}
|Q|^{-\frac{s(y)}{n}}(\vz_{j_Q}^\ast f)_a(y)\\
&&\hs\ls\frac{1}{\phi(Q)}
\lf\||Q|^{-\frac{s(\cdot)}{n}}(\vz_{j_Q}^\ast f)_a\r\|_{L^{p(\cdot)}(Q)}
\ls\|f\|_{\btlve},
\end{eqnarray*}
which implies that $\|f|\btlve\|_2\ls\|f\|_{\btlve}$.

Conversely, by choosing $r\in(0,\min\{1,p_-,q_-\})$ and
an argument similar to that used in the proof of (i),
we conclude that, for any $P\in \cq$,
\begin{eqnarray*}
&&\frac1{\phi(P)}\lf\|\lf\{\sum_{j=(j_P\vee 0)}^\fz
\lf[2^{js(\cdot)}|\vz_j\ast f|\r]^{q(\cdot)}\r\}^{\frac1{q(\cdot)}}
\r\|_{L^{p(\cdot)}(P)}\\
&&\hs\ls\lf\|f\lf|\btlve\r.\r\|_2\lf\|\lf\{\sum_{j=(j_P\vee 0)}^\fz
\lf[\sum_{\gfz{\ell(\wz Q)=2^{-j}}{\wz Q\in\cq,\wz Q\subset P}}
\frac{\phi(P)^{-1}\phi(\wz Q)\chi_{\wz Q}}{\|\chi_{\wz Q}\|_{\vlp}}
\r]^{q(\cdot)}\r\}^{\frac1{q(\cdot)}}
\r\|_{L^{p(\cdot)}(P)}\\
&&\hs\ls\lf\|f\lf|\btlve\r.\r\|_2\lf\{\sum_{j=(j_P\vee 0)}^\fz
2^{j(\frac n{p_-}+\log_2c_1)r}\r\}^\frac1r2^{-j_P(\frac n{p_-}+\log_2c_1)}\\
&&\hs\ls\lf\|f\lf|\btlve\r.\r\|_2,
\end{eqnarray*}
where we used the fact that $c_1\in(0,2^{-n/p_-})$ in the last inequality, which
implies that
$$\|f\|_{\btlve}\ls\|f|\btlve\|_2.$$
This finishes the proof of (ii)
and hence Theorem \ref{t-sum}.
\end{proof}

\begin{remark}
In the case that $p,\ q,\ s$ and $\phi$ are as in Remark \ref{r-defi}(ii),
Theorem \ref{t-sum}(i) coincides with \cite[Corollary 3.3(i)]{ysiy}
and Theorem \ref{t-sum}(ii) goes back to
\cite[Theorem 2.2(i)]{yyaa13}.
\end{remark}

We now compare the Triebel-Lizorkin-type space
with variable exponents in this article with the variable
Triebel-Lizorkin-Morrey space
$\mathcal{E}_{p(\cdot),q(\cdot),u}^{s(\cdot)}(\rn)$ introduced by Ho
\cite{ho12} and show that, in general, these two scales of
Triebel-Lizorkin spaces do not cover each other.

To recall the definition of the variable Triebel-Lizorkin-Morrey
space in \cite{ho12},
we need some notions.
A measurable function $u(x,r):\ \rn\times(0,\fz)\to(0,\fz)$ is said to
belong to $\mathcal{W}_q$ with $q\in(0,\fz)$ if there exist $C_1,
\ C_2\in(0,\fz)$ and $\lz\in[0,1/q)$ such that, for all $x\in\rn$,
$u(x,r)>1$ if $r\in[1,\fz)$,
$\frac{u(x,2r)}{u(x,r)}\le 4^{n\lz}$ if $r\in(0,\fz)$,
and
\begin{equation*}
C_2^{-1}\le\frac{u(x,t)}{u(x,r)}\le C_2\quad {\rm if}\quad 0<r\le t\le 2r.
\end{equation*}

\begin{definition}\label{d-tlm}
Let $p,\ q,\ s$, $\{\vz_j\}_{j\in\zz_+}$ be as in Definition \ref{d-tl} and
$u\in \mathcal{W}_{p_+}$.
Then the \emph{variable Triebel-Lizorkin-Morrey space}
$\mathcal{E}_{p(\cdot),q(\cdot),u}^{s(\cdot)}(\rn)$ is defined
to be the set of all
$f\in\cs'(\rn)$ such that
\begin{equation*}
\lf\|f\r\|_{\mathcal{E}_{p(\cdot),q(\cdot),u}^{s(\cdot)}(\rn)}
:=\sup_{\gfz{z\in\rn}{R\in(0,\fz)}}\frac1{u(z,R)}
\lf\|\lf\{\sum_{j=0}^\fz\lf[2^{js(\cdot)}|\vz_j\ast f|
\r]^{q(\cdot)}\r\}^\frac1{q(\cdot)}\r\|_{L^{p(\cdot)}(B(z,R))}<\fz.
\end{equation*}
\end{definition}

\begin{remark}\label{r-compare}
(i) We point out that the Triebel-Lizorkin-type space with variable
exponents in this article
can not be covered by the Triebel-Lizorkin-Morrey space in \cite{ho12} even when
$\wz c_1\in(0,2^{n/p_+})$.
To see this, it suffices to show that there exists a
set function $\phi$ satisfying (\textbf{S1}) and (\textbf{S2})
does not belong to $\mathcal{W}_{q}$ for any $q\in(0,\fz)$.

Indeed, for all $Q\subset\rn$, let
$\phi(Q):=\int_Q|x|^\alpha\,dx$, where $\alpha\in(-n,0)$.
Then, by \cite[p.\,196]{stein93}, we know that
$\phi$ is doubling, which, together with Remark \ref{r-phi}(iv),
further implies that
$\phi$ satisfies the conditions (\textbf{S1}) and (\textbf{S2}). However,
$\phi\notin \mathcal{W}_q$ for any $q\in(0,\fz)$. To see this,
let $x_0\in\rn$ and $r\in(1,2)$ satisfy $|x_0|\ge 2r$. Then
$$\phi(x_0,r):=\phi(Q(x_0,r))=\int_{Q(x_0,r)}|y|^\alpha\,dy
\sim|x_0|^\alpha $$
tends $0$ as $|x_0|\to\fz$ since $\alpha\in(-n,0)$, which implies that
$\phi\notin \mathcal{W}_q$ for any $q\in(0,\fz)$.

(ii) Also, the variable Triebel-Lizorkin-Morrey space investigated in \cite{ho12}
can not be covered by
the Triebel-Lizorkin-type space with variable exponents in this article.
To see this, it suffices to show that there exists a
function $u$ such that $u$ belongs to $\mathcal{W}_1$ but
does not satisfy the condition (\textbf{S2}).

Indeed, let, for all $x\in\rn$ and $r\in(0,\fz)$, $u(x,r):=r^{\lz(x)}$, where
$\lz(x):=n(1-\frac 1{1+|x|})$. Then,
as was pointed out in \cite[p.\,380]{ho12}, $u\in \mathcal{W}_1$. However, $u$
does not satisfy the condition (\textbf{S2}). To see this,
let $x,\ y\in\rn$ satisfy that $\vez<|x|<\frac{(1+\vez)r+\vez}{2+\vez}$ and
$|y|=\frac{1+|x|}{1+\vez}-1$,
where $\vez\in(0,\fz)$ and $r\in(\vez,\fz)$. Then
$$|x-y|\le|x|+|y|=|x|+\frac{1+|x|}{1+\vez}-1<r,$$
but
$$\frac{u(x,r)}{u(y,r)}=r^{n(\frac1{1+|y|}-\frac1{1+|x|})}
=r^{\frac{n\vez}{1+|x|}}
\to\fz,\quad{\rm as}\quad r\to\fz,$$
which implies that $u$
does not satisfy the condition (\textbf{S2}).
\end{remark}

As an application of Theorem \ref{t-sum}, we prove that the space
$F_{p(\cdot),2}^{0,\phi}(\rn)$ coincides with the \emph{Morrey space
with variable exponent}, $\cm_\phi^{p(\cdot)}(\rn)$, which
is defined to be the set of all measurable functions $f$ such that
$$\|f\|_{\cm_{\phi}^{p(\cdot)}(\rn)}
:=\sup_{P\in\cq}\frac1{\phi(P)}\|f\|_{L^{p(\cdot)}(P)}<\fz,$$
where the supremum is taken over all dyadic cubes of $\rn$.
\begin{remark}
(i) We point out that, in \cite{ho12}, Ho studied
the variable Morrey space $\cm_u^{p(\cdot)}(\rn)$, which is
defined in the same way as $\cm_\phi^{p(\cdot)}(\rn)$ above but with $\phi$ replaced
by $u$ as in Definition \ref{d-tlm} and the supremum is taken
over all balls of $\rn$. From Remark \ref{r-compare},
we deduce that the Morrey space with variable exponent $\cm_\phi^{p(\cdot)}(\rn)$
in this article and the variable Morrey space $\cm_u^{p(\cdot)}(\rn)$
 in \cite{ho12}
do not cover each other.

(ii) For $\vz:\ \rn\times(0,\fz)\to(0,\fz)$ and
a variable exponent $p:\ \rn\to[1,\fz)$,
Nakai \cite{nakai14} introduced the \emph{variable Morrey space} $L^{(p,\vz)}(\rn)$,
 which is defined
to be the set of all measurable functions $f$ such that
$$\|f\|_{L^{(p,\vz)}(\rn)}:=\sup_{{\rm balls}\ B\subset\rn}\|f\|_{p,\vz,B}<\fz,$$
where, for all balls $B:=B(x,r)\subset\rn$, $\vz(B):=\vz(x,r)$ and
$$\|f\|_{p,\vz,B}:=\inf\lf\{\lz\in(0,\fz):\ \frac1{\vz(B)|B|}\int_B
\lf[\frac{|f(y)|}{\lz}\r]^{p(y)}\,dy\le1\r\},$$
and the supremum is taken over all balls $B$ of $\rn$.

We claim that, if there exists a positive constant $C$ such that, for all
$x\in \rn$ and $0<r<s<\infty$,
\begin{equation}\label{3.18y}
C^{-1}\phi(x,r)\le \phi(x,s)\le C\phi(x,r)
\end{equation}
and,  for all balls $B\subset\rn$ and all $y\in B$,
\begin{equation}\label{3.18x}
\vz(B)|B|\sim [\phi(B)]^{p(y)},
\end{equation}
then $\cm_\phi^{p(\cdot)}(\rn)$ coincides with $L^{(p,\vz)}(\rn)$.

Indeed, by \eqref{3.18y} and the definition of $\|\cdot\|_{\cm_\phi^{p(\cdot)}(\rn)}$,
we conclude that
\begin{equation}\label{3.19x}
\|f\|_{\cm_\phi^{p(\cdot)}(\rn)}
\sim\sup_{{\rm balls}\ B\subset\rn}\inf\lf\{\lz\in(0,\fz):\ \int_B\lf[
\frac{|f(y)|}{\phi(B)\lz}\r]^{p(y)}\,dy\le1\r\}.
\end{equation}
On the other hand, by \eqref{3.18x}, we find that
\begin{eqnarray*}
&&\inf\lf\{\lz\in(0,\fz):\ \int_B\lf[
\frac{|f(y)|}{\phi(B)\lz}\r]^{p(y)}\,dy\le1\r\}\\
&&\hs\sim
\inf\lf\{\lz\in(0,\fz):\ \frac1{\vz(B)|B|}\int_B
\lf[\frac{|f(y)|}{\lz}\r]^{p(y)}\,dy\le1\r\}
\end{eqnarray*}
which, combined with \eqref{3.19x}, implies that $\cm_\phi^{p(\cdot)}(\rn)$
coincides with
$L^{(p,\vz)}(\rn)$. This proves the above claim.

Obviously, in general, these two scales of Morrey spaces with variable exponents,
$\cm_\phi^{p(\cdot)}(\rn)$ and $L^{(p,\vz)}(\rn)$,
may not cover each other.
\end{remark}

In what follows, for all $p\in\cp(\rn)$,
denote by $L^{p(\cdot)}(\ell^2(\rn))$
the \emph{set} of all sequences $\{g_j\}_{j\in\zz_+}$ of measurable
functions such that
$$\|\{g_j\}_{j\in\zz_+}\|_{L^{p(\cdot)}(\ell^2(\rn))}
:=\lf\|\lf\{\sum_{j\in\zz_+}|g_j|^2\r\}^{\frac12}\r\|
_{\vlp}<\fz.$$
Let $(\vz,\Phi)$ and $(\psi,\Psi)$ be two pairs of admissible functions
satisfying \eqref{cz1}. The operator $\mathcal{G}$ is defined by setting,
for all $f\in L^{p(\cdot)}(\rn)$,
$\mathcal{G}(f):=\{\vz_j\ast f\}_{j\in\zz_+}$, where, when $j=0$,
 $\vz_0$ is replaced by $\Phi$,
 and its conjugate operator $\mathcal{G}^\ast$ is defined by setting, for all
$\{g_j\}_{j\in\zz_+}\in L^{\wz{P}(\cdot)}(\ell^2(\rn))$,
$$\mathcal{G}^\ast(\{g_j\}_{j\in\zz_+})
:=\sum_{j\in\zz_+}\psi_j\ast g_j,$$
where, when $j=0$, $\psi_0$ is replaced by $\Psi$.

\begin{remark}\label{r-lpc}
Let $p(\cdot)\in C^{\log}(\rn)$ satisfy $1<p_-\le p_+<\fz$.
Then, from the fact that $L^{\wz{P}(\cdot)}(\rn)=F_{\wz{P}(\cdot),2}^{0}(\rn)$
(see \cite[Theorem 4.2]{dhr09}),
we deduce that the operator $\mathcal{G}$ is bounded from
$L^{\wz{P}(\cdot)}(\rn)$ to $L^{\wz{P}(\cdot)}(\ell^2(\rn))$.
Furthermore, by an argument similar to that used in the proof of
\cite[Corollary 4.4]{ho12}, we conclude
that the operator $\mathcal{G}^\ast$ is bounded from $L^{p(\cdot)}(\ell^2(\rn))$
to $\vlp$.
\end{remark}

\begin{proposition}\label{p-lpc}
Let $p$ and $\phi$ be as in Definition \ref{d-tl} and $\wz c_1\in(0,2^{n/p_+})$.
If $1<p_-\le p_+<\fz$, then
$$\cm_\phi^{p(\cdot)}(\rn)=F_{p(\cdot),2}^{0,\phi}(\rn)$$
with equivalent norms.
\end{proposition}
\begin{proof}
We first prove that $\cm_\phi^{p(\cdot)}(\rn)\hookrightarrow
 F_{p(\cdot),2}^{0,\phi}(\rn)$. By Theorem \ref{t-sum}(i),
it suffices to show that, for all $f\in \cm_\phi^{p(\cdot)}(\rn)$,
\begin{equation}\label{morrey3}
\sup_{Q\in\cq}\frac1{\phi(Q)}\lf\|\lf\{\sum_{j=0}^\fz
|\vz_j\ast f|^2\r\}^\frac12\r\|_{L^{p(\cdot)}(Q)}
\ls\|f\|_{\cm_{\phi}^{p(\cdot)}(\rn)},
\end{equation}
 where $\{\vz_j\}_{j=0}^\fz$ are as in Definition \ref{d-tl}.

For all $Q:=Q(x_0,r)\in \cq$, let $f_1:=f\chi_{Q(x_0,2r)}$ and $f_2:=f-f_1$. From
\cite[Theorem 4.2]{dhr09} and the condition (\textbf{S1}) of $\phi$, we deduce that
\begin{eqnarray}\label{morrey1}
{\rm I}_1&:=&
\frac1{\phi(Q)}\lf\|\lf\{\sum_{j=0}^\fz|\vz_j\ast f_1|^2
\r\}^\frac12\r\|_{L^{p(\cdot)}(Q)}\noz\\
&\ls&\frac1{\phi(Q)}\|f_1\|_{L^{p(\cdot)}(\rn)}
\sim\frac1{\phi(Q)}\|f\|_{L^{p(\cdot)}(Q(x_0,2r))}
\ls\|f\|_{\cm_\phi^{p(\cdot)}(\rn)}.
\end{eqnarray}
On the other hand, by the Minkowski inequality, we find that, for all $x\in\rn$,
\begin{eqnarray*}
\lf\{\sum_{j=0}^\fz|\vz_j\ast f_2(x)|^2\r\}^\frac12
&\ls&\int_{\rn\setminus Q(x_0,2r)}
\lf\{\sum_{j=0}^\fz|\vz_j(x-y)|^2\r\}^\frac12|f(y)|\,dy\\
&\ls&\int_{\rn\setminus Q(x_0,2r)}\frac{|f(y)|}{|x-y|^n}\,dy
\sim\sum_{k=1}^\fz\int_{S_k}\frac{|f(y)|}{|x-y|^n}\,dy,
\end{eqnarray*}
where, for $k\in\nn$, $S_k:=Q_{k+1}\setminus Q_k$ and $Q_k:=Q(x_0,2^kr)$.
Observe that, when $x\in Q(x_0,r)$ and $y\in S_k$,
$|x-y|\ge 2^kr$. Setting $(p(\cdot))^\ast:=\frac{p(\cdot)}{p(\cdot)-1}$,
by the H\"older inequality of variable Lebesgue spaces
(see Remark \ref{r-vlpp}(iii)),
 \cite[Lemma 2.6]{zyl14} and \cite[Proposition 2.4]{ho12},
 we see that
\begin{eqnarray*}
&&\lf\|\lf\{\sum_{j=0}^\fz|\vz_j\ast f_2|^2\r\}^\frac12\r\|_{L^{p(\cdot)}(Q)}\\
&&\hs\hs\ls \sum_{k=1}^\fz\frac1{2^{kn}r^n}
\|f\|_{L^{p(\cdot)}(S_k)}\|\chi_{S_k}\|_{L^{p(\cdot)^\ast}(\rn)}
\|\chi_{Q}\|_{\vlp}\\
&&\hs\hs\ls\sum_{k=1}^\fz\frac{2^{-kn/p_+}}{r^n2^{kn}}
\|f\|_{L^{p(\cdot)}(S_k)}\|\chi_{Q_{k+1}}\|_{L^{p(\cdot)^\ast}(\rn)}
\|\chi_{Q_{k+1}}\|_{\vlp}\\
&&\hs\hs\ls\sum_{k=1}^\fz2^{-kn/p_+}
\|f\|_{L^{p(\cdot)}(S_k)},
\end{eqnarray*}
which, together with the condition (\textbf{S1}) of $\phi$,
implies that
\begin{eqnarray}\label{morrey2}
{\rm I}_2
&:=&\frac1{\phi(Q)}
\lf\|\lf\{\sum_{j=0}^\fz|\vz_j\ast f_2|^2\r\}^\frac12\r\|_{L^{p(\cdot)}(Q)}
\ls\sum_{k=1}^\fz2^{-kn/p_+}\frac{\|f\|_{L^{p(\cdot)}(Q_{k+1})}}{\phi(Q_{k+1})}
\frac{\phi(Q_{k+1})}{\phi(Q)}\noz\\
&\ls&\|f\|_{\cm_{\phi}^{p(\cdot)}(\rn)}\sum_{k=1}^\fz
2^{-kn/p_+}2^{k\log_2\wz c_1}\sim\|f\|_{\cm_{\phi}^{p(\cdot)}(\rn)},
\end{eqnarray}
where we used the fact that $\wz c_1\in(0,2^{n/p_+})$ in the last inequality.

Combining \eqref{morrey1} and \eqref{morrey2}, we conclude that
\begin{equation*}
\sup_{Q\in\cq}\frac1{\phi(Q)}\lf\|\lf\{\sum_{j=0}^\fz
|\vz_j\ast f|^2\r\}^\frac12\r\|_{L^{p(\cdot)}(Q)}
\ls\sup_{Q\in\cq}({\rm I}_1+{\rm I}_2)\ls\|f\|_{\cm_{\phi}^{p(\cdot)}(\rn)}
\end{equation*}
and \eqref{morrey3} holds true.

Next, we prove that $F_{p(\cdot),2}^{0,\phi}(\rn)\hookrightarrow
\cm_{\phi}^{p(\cdot)}(\rn)$. Let $f\in F_{p(\cdot),2}^{0,\phi}(\rn)$.
Then, by the Calder\'on reproducing formula (see \cite[Lemma 2.3]{ysiy}),
we find that
\begin{equation}\label{morrey-w}
f=\Psi\ast \Phi\ast f+\sum_{j=1}^\fz\psi_j\ast \vz_j\ast f
=:\sum_{j=0}^\fz \psi_j\ast\vz_j\ast f
\end{equation}
in $\cs'(\rn)$, where $\Psi,\ \Phi$, $\vz$ and $\psi$ are as in
\eqref{cz1}. For all $j\in\zz_+$, we use $f_j$ to denote $\vz_j\ast f$.
For all $Q:=Q(x_0,r)\in\cq$ and $j\in\zz_+$,
let $f_j^1:=f_j\chi_{Q(x_0,2r)}$ and $f_j^2:=f_j-f_j^1$.
Then we know that
\begin{eqnarray*}
&&\frac1{\phi(Q)}
\lf\|\sum_{j=0}^\fz\psi_j\ast\vz_j\ast f\r\|_{L^{p(\cdot)}(Q)}\\
&&\hs\ls\frac1{\phi(Q)}\lf\|\sum_{j=0}^\fz\psi_j\ast f_j^1\r\|_{L^{p(\cdot)}(Q)}
+\frac1{\phi(Q)}\lf\|\sum_{j=0}^\fz\psi_j\ast f_j^2\r\|_{L^{p(\cdot)}(Q)}
=:{\rm J}_1+{\rm J}_2.
\end{eqnarray*}

By Remark \ref{r-lpc}, the condition (\textbf{S1}) of $\phi$ and
Theorem \ref{t-sum}(i), we see that
\begin{eqnarray}\label{morrey-y}
{\rm J}_1
&\sim&\frac1{\phi(Q)}\lf\|\mathcal{G}^\ast
\lf(\lf\{f_j^1\r\}_{j\in\zz_+}\r)\r\|_{\vlp}\noz\\
&\ls&\frac1{\phi(Q)}\lf\|\lf\{\sum_{j=0}^\fz|\vz_j\ast f|^2
\r\}^\frac12\r\|_{L^{p(\cdot)}(Q(x_0,2r))}
\ls\|f\|_{F_{p(\cdot),2}^{0,\phi}(\rn)}.
\end{eqnarray}
On the other hand, for all $x\in Q(x_0,2r)$,
by the H\"older inequality, we find that
\begin{equation*}
\lf|\sum_{j\in\zz_+}\psi_j\ast f_j^2(x)\r|
\ls\int_{\rn\backslash Q(x_0,2r)}
\lf[\sum_{j\in\zz_+}|f_j(y)|^2\r]^\frac12|x-y|^{-n}\,dy.
\end{equation*}
Thus, by an argument similar to that used in the proof of
\eqref{morrey2}, we conclude that
$${\rm J}_2\ls \|f\|_{F_{p(\cdot),2}^{0,\phi}(\rn)},$$
which, combined with \eqref{morrey-y}, implies that
$$\frac1{\phi(Q)}
\lf\|\sum_{j=0}^\fz\psi_j\ast\vz_j\ast f\r\|_{L^{p(\cdot)}(Q)}
\ls{\rm J}_1+{\rm J}_2\ls \|f\|_{F_{p(\cdot),2}^{0,\phi}(\rn)}.$$
Therefore, $\sum_{j=0}^\fz\psi_j\ast\vz_j\ast f\in\cm_\phi^{p(\cdot)}(\rn)$ and
$$\lf\|\sum_{j=0}^\fz\psi_j\ast\vz_j\ast f\r\|_{\cm_\phi^{p(\cdot)}(\rn)}
\ls \|f\|_{F_{p(\cdot),2}^{0,\phi}(\rn)},$$
which, together with \eqref{morrey-w}, implies that
$F_{p(\cdot),2}^{0,\phi}(\rn)\hookrightarrow
\cm_{\phi}^{p(\cdot)}(\rn)$. This finishes the proof of Proposition \ref{p-lpc}.
\end{proof}

\begin{remark}
In the case that $p$ and $\phi$ are as in Remark \ref{r-defi}(ii), the conclusion
of Proposition \ref{p-lpc} is already known; see, for example,
\cite[Theorem 3.9]{st07}.
\end{remark}

We end this section by giving another application of Theorem \ref{t-pmfc}.

\begin{proposition}\label{p-subset}
Let $p,\ q,\ s$ and $\phi$ be as in Definition \ref{d-tl}. Then
$$\cs(\rn)\hookrightarrow\btlve\hookrightarrow\cs'(\rn).$$
\end{proposition}
\begin{proof}
We first prove that $\cs(\rn)\hookrightarrow\btlve$. To prove this embedding,
we need to show that there exists an $M\in\nn$ such that, for all $f\in\cs(\rn)$,
$$\|f\|_{\btlve}\ls\|f\|_{\cs_M(\rn)}.$$

Let $f\in\cs(\rn)$ and $(\vz,\Phi)$ be a pair of admissible functions.
 Let $P:=Q_{j_Pk_P}$ be an arbitrary dyadic cube.
If $j_P>0$, choosing $r\in(0,\min\{1,p_-,q_-\})$,
by \cite[Lemma 2.4]{ysiy}, Remark \ref{r-vlpp}(i) and Lemmas \ref{l-esti-cube}
and \ref{l-esti-cube-1}, we obtain
\begin{eqnarray}\label{subset1}
&&\frac1{\phi(P)}\lf\|\lf\{\sum_{j=j_P}^\fz
\lf[2^{js(\cdot)}|\vz_j\ast f|\r]^{q(\cdot)}
\r\}^\frac1{q(\cdot)}\r\|_{L^{p(\cdot)}(P)}\noz\\
&&\hs\ls\|f\|_{\cs_{M+1}(\rn)}\frac1{\phi(P)}
\lf\|\lf\{\sum_{j=j_P}^\fz\lf[2^{js(\cdot)}
\frac{2^{-jM}}{(1+|\cdot|)^{n+M}}\r]^{q(\cdot)}
\r\}^\frac1{q(\cdot)}\r\|_{L^{p(\cdot)}(P)}\noz\\
&&\hs\ls\|f\|_{\cs_{M+1}(\rn)}\frac1{\phi(P)}
\lf\{\sum_{j=j_P}^\fz2^{-j(M-s_+)r}\lf\|\frac1{(1+|\cdot|)^{(M+n)r}}
\r\|_{L^{\frac{p(\cdot)}r}(P)}\r\}^\frac1r\noz\\
&&\hs\ls\|f\|_{\cs_{M+1}(\rn)}\lf\{\sum_{j=j_P}^\infty2^{-j(M-s_+)r}
\frac1{(1+2^{-j_P}|k_P|)^{(M+n)r}}\r\}^\frac1r
\frac{\|\chi_P\|_{L^{p(\cdot)}(P)}}{\phi(P)}\noz\\
&&\hs\ls\|f\|_{\cs_{M+1}(\rn)}2^{-j_P(\frac M2+\frac n{p_+}-s_+-\log_2c_1)}
(1+|k_P|)^{-\frac M2+n(\frac1{p_-}-\frac1p_+)+\log_2(c_1\wz c_1)}\noz\\
&&\hs\ls\|f\|_{\cs_{M+1}(\rn)},
\end{eqnarray}
where $M$ is chosen large enough.

If $j_P\le0$, then we see that
\begin{eqnarray*}
{\rm I}_P&:=&\frac1{\phi(P)}\lf\|\lf\{\sum_{j=0}^\fz
\lf[2^{js(\cdot)}|\vz_j\ast f|\r]^{q(\cdot)}
\r\}^\frac1{q(\cdot)}\r\|_{L^{p(\cdot)}(P)}\\
&\ls&\frac1{\phi(P)}\|\Phi\ast f\|_{L^{p(\cdot)}(P)}+
\frac1{\phi(P)}\lf\{\sum_{j=1}^\fz2^{js_+r}
\|\vz_j\ast f\|_{L^{p(\cdot)}(P)}^r\r\}^\frac1r.
\end{eqnarray*}
When $P$ is away from the origin, by
an argument similar to that used in the proof of \eqref{subset1},
 we conclude that
${\rm I}_P\ls\|f\|_{\cs_{M+1}(\rn)}$ with $M$ being sufficiently large.
When one of the corners of $P$ is the origin, then
$P\subset \cup_{i=0}^{-j_P+n}S_i$, where $S_0:=B(0,1)$
and $S_i:=2^iS_0\backslash(2^{i-1}S_0)$ for all $i\in\{1,\dots,-j_P+1\}$.
From this, Lemmas \ref{l-esti-cube} and \ref{l-esti-cube-1}
 and the fact that $|k_P|\le1$, we deduce that
\begin{eqnarray*}
\frac1{\phi(P)}\|\Phi\ast f\|_{L^{p(\cdot)}(P)}
&\ls&\frac1{\phi(P)}\lf\{\sum_{i=0}^{-j_P+n}
\lf\|\frac1{(1+|\cdot|)^{M}}\r\|_{L^{p(\cdot)}(S_i)}^r\r\}^{\frac1r}
\ls\|f\|_{\cs_{M+1}(\rn)}
\end{eqnarray*}
and, similarly,
\begin{eqnarray*}
\frac1{\phi(P)}\lf\{\sum_{j=1}^\fz2^{js_+r}
\|\vz_j\ast f\|_{L^{p(\cdot)}(P)}^r\r\}^\frac1r\ls\|f\|_{\cs_{M+1}(\rn)},
\end{eqnarray*}
where $M$ is chosen large enough, which implies that
 ${\rm I}_P\ls\|f\|_{\cs_{M+1}(\rn)}$.
Therefore, $\cs(\rn)\hookrightarrow\btlve$ and
 $\|f\|_{\btlve}\ls\|f\|_{\cs_{M+1}(\rn)}$.

Next we show that $\btlve\hookrightarrow\cs'(\rn)$.
To this end, we need to prove that
there exists an $M\in\nn$ such that, for all $f\in\btlve$ and $h\in\cs(\rn)$,
\begin{equation*}
|\la f,h\ra|\ls\|f\|_{\btlve}\|h\|_{\cs_{M+1}(\rn)}.
\end{equation*}

Let $\vz$, $\psi$, $\Phi$ and $\Psi$ be as in Theorem \ref{t-transform}. Then,
by the Calder\'on reproducing formula in \cite[Lemma 2.3]{ysiy}, together with
\cite[Lemma 2.4]{ysiy},
we obtain
\begin{eqnarray}\label{subset2}
|\la f,h\ra|
&\le&\int_\rn|\Phi\ast f(x)||\Psi\ast h(x)|\,dx+\sum_{j=1}^\fz
\int_\rn|\vz_j\ast f(x)||\psi_j\ast h(x)|\,dx\noz\\
&\ls&\|h\|_{\cs_{M+1}(\rn)}\sum_{j=0}^\fz2^{-jM}\int_\rn
|\vz_j\ast f(x)|(1+|x|)^{-(n+M)}\,dx\noz\\
&\sim&\|h\|_{\cs_{M+1}(\rn)}\sum_{j=0}^\fz2^{-jM}
\sum_{k\in\zz^n}\int_{Q_{0k}} |\vz_j\ast f(x)|(1+|x|)^{-(n+M)}\,dx,\quad\quad
\end{eqnarray}
where we used $\Phi$ to replace $\vz_0$.
Notice that, for any $j\in\zz_+$, $k\in\zz^n$, $a\in(0,\fz)$ and $y\in Q_{jk}$,
\begin{eqnarray*}
\int_{Q_{0k}}|\vz_j\ast f(x)|\,dx
&\ls& (\vz_j^\ast f)_a(y)\int_{Q_{0k}}(1+2^j|x|+2^j|y|)^a\,dx\\
&\ls&2^{ja}(\vz_j^\ast f)_a(y)(1+|k|)^a.
\end{eqnarray*}
Then, by the arbitrariness of $y\in Q_{jk}$, we see that
\begin{equation*}
\int_{Q_{0k}}|\vz_j\ast f(x)|\,dx\ls 2^{ja}(1+|k|)^a\inf_{y\in Q_{jk}}
(\vz_j^\ast f)_a(y),
\end{equation*}
which, combined with \eqref{subset2}, Theorem \ref{t-pmfc}
and Lemmas \ref{l-esti-cube} and \ref{l-esti-cube-1}, implies that
\begin{eqnarray*}
|\la f,h\ra|
&\ls&\|h\|_{\cs_{M+1}(\rn)}\sum_{j=0}^\fz2^{-jM}
\sum_{k\in\zz^n}\int_{Q_{0k}}\frac{|\vz_j\ast f(x)|}{(1+|k|)^{n+M}}\,dx\\
&\ls&\|h\|_{\cs_{M+1}(\rn)}\sum_{j=0}^\fz2^{-jM+ja}
\sum_{k\in\zz^n}\frac{\inf_{y\in Q_{jk}}
(\vz_j^\ast f)_a(y)}{(1+|k|)^{n+M-a}}\\
&\ls&\|h\|_{\cs_{M+1}(\rn)}\sum_{j=0}^\fz2^{-jM+ja}
\sum_{k\in\zz^n}(1+|k|)^{-(n+M-a)}
\frac{\|(\vz_j^\ast f)_a\|_{L^{p(\cdot)}(Q_{jk})}}
{\|\chi_{Q_{jk}}\|_{L^{p(\cdot)}(Q_{jk})}}\\
&\ls&\|f\|_{\btlve}\|h\|_{\cs_{M+1}(\rn)}\sum_{j=0}^\fz2^{j(a-M-s_-)}\\
&&\hs\times\sum_{k\in\zz^n}(1+|k|)^{(a-n-M)}
\frac{\phi(Q_{jk})}{\|\chi_{Q_{jk}}\|_{L^{p(\cdot)}(Q_{jk})}}
\ls\|f\|_{\btlve}\|h\|_{\cs_{M+1}(\rn)},
\end{eqnarray*}
where $a$ is chosen as in \eqref{peetre1}.
This finishes the proof of Proposition \ref{p-subset}.
\end{proof}

\section{A trace theorem}\label{s4}
In this section, we mainly establish a trace theorem for Triebel-Lizorkin-type
spaces with variable exponents
by applying the atomic characterization of these spaces obtained in Theorem \ref{t-md}.

To state our main result of this section, we first give some notation.
For measurable functions $p$, $q$, $s$ and a set function $\phi$ being
 as in Definition
\ref{d-tl}, let
$F_{p(\wz{\cdot},0),q(\wz{\cdot},0)}^{s(\wz{\cdot},0),\wz \phi}(\rr^{n-1})$
denote the Triebel-Lizorkin-type spaces with variable exponents
$p(\wz{\cdot},0)$, $q(\wz{\cdot},0)$ and $s(\wz{\cdot},0)$ on $\rr^{n-1}\times\{0\}$,
where $\wz \phi$ is defined by setting,
for all cubes $\wz{Q}$ of $\rr^{n-1}$, $\wz \phi(\wz{Q}):=\phi(\wz{Q}\times[0,\ell(\wz{Q}))$.
In what follows, let $\rr_+^{n}:=\rr^{n-1}\times[0,\fz)$ and
$\rr_-^n:=\rr^{n-1}\times(-\fz,0]$.

Let $f\in \btlve$. Then, by Theorem \ref{t-md}, we have
$f=\sum_{Q\in\cq^\ast}t_Qa_Q$ in $\cs'(\rn)$ and
$$\lf\|\{t_Q\}_{Q\in\cq^\ast}\r\|_{\tlve}
\le C\|f\|_{\btlve},$$
where $C$ is a positive constant independent of $f$ and, for each $Q\in\cq^\ast$,
$a_Q$ is a smooth atom of
$\btlve$. Define the \emph{trace} of $f$ by setting, for all $\wz{x}\in\rr^{n-1}$,
\begin{equation}\label{d-trace}
\mathop\mathrm{Tr}(f)(\wz{x}):=\sum_{Q\in\cq^\ast}t_Qa_Q(\wz{x},0).
\end{equation}
This definition of $\mathop\mathrm{Tr}(f)$ is determined canonical for all
$f\in \btlve$, since the actual construction of $a_Q$ in the proof of
Theorem \ref{t-md} implies that $t_Qa_Q$ is obtained canonically.
Moreover, in Lemma \ref{l-welld} below, we show that the summation in
\eqref{d-trace} converges in $\cs'(\rr^{n-1})$.
Thus, the trace operator is well defined.

The main result of this section is the following trace theorem.

\begin{theorem}\label{t-trace1}
Let $n\ge2$, $p$, $q\in \cp(\rn)$ satisfy
$$0<p_-\le p_+<\fz, \quad 0<q_-\le q_+<\fz$$
and $\frac1p$, $\frac1q\in C^{\log}(\rn)$,
$s\in C_{\loc}^{\log}(\rn)\cap L^\fz(\rn)$ and
$\phi$ be a set function satisfying the conditions
(\textbf{S1}) and (\textbf{S2}).
If
\begin{equation}\label{trace1-x}
s_--\frac1{p_-}-(n-1)\lf[\frac1{\min\{1,p_-\}}-1\r]>0,
\end{equation}
then
\begin{equation*}
\mathop\mathrm{Tr}\btlve=
F_{p(\wz{\cdot},0),p(\wz{\cdot},0)}^{s(\wz{\cdot},0)
-\frac1{p(\wz{\cdot},0)},\wz \phi}(\rr^{n-1}).
\end{equation*}
\end{theorem}

\begin{remark}\label{r-trace-x}
(i) When $p,\ q,\ s$ and $\phi$ are as in Remark \ref{r-defi}(ii),
Theorem \ref{t-trace1} goes back to \cite[Theorem 6.8]{ysiy}.
Moreover, Theorem \ref{t-trace1} coincides with
the trace theorem for the classical Triebel-Lizorkin space $F_{p,q}^s(\rn)$
with constant
exponents (see \cite[Theorem 11.1 and p.\,134]{fj90}) and,
in this case, the condition \eqref{trace1-x} is optimal.

(ii) In the case that $\phi$ is as in Remark \ref{r-defi}(i),
it was proved in \cite[Theorem 3.13]{dhr09}
(see also \cite[Theorem 5.1(1)]{noi14}) that the conclusion of
Theorem \ref{t-trace1} is true if $s$ and $p$ satisfy that, for all $x\in\rn$,
\begin{equation}\label{condition}
s(x)-\frac1{p(x)}-(n-1)\lf[\frac1 {\min\{1,p(x)\}}-1\r]>\delta
\end{equation}
for some $\delta\in(0,\fz)$, which is a little weaker than \eqref{trace1-x}.
The reason that
the assumption \eqref{trace1-x}, in this case, is a little stronger than that
in \cite[Theorem 3.13]{dhr09} (see also \eqref{condition})
comes from an application
of Theorem \ref{t-md}, which, in this case, can be further refined; see Remark
\ref{r-mad}(ii).
\end{remark}

To prove Theorem \ref{t-trace1}, we first need to show that \eqref{d-trace}
converges in $\cs'(\rr^{n-1})$.

\begin{lemma}\label{l-welld}
Let $p,\ q,\ s$ and $\phi$ be as in Theorem \ref{t-trace1} satisfying
\eqref{trace1-x}. Then, for all
$f\in F_{p(\cdot),q(\cdot)}^{s(\cdot),\phi}(\rn)$,
$\mathop\mathrm{Tr}(f)\in \cs'(\rr^{n-1})$.
\end{lemma}

\begin{proof}
Let $f\in F_{p(\cdot),q(\cdot)}^{s(\cdot),\phi}(\rn)$.
Then, by Theorem \ref{t-md}, we can write
$$f=\sum_{Q\in\cq^\ast}t_Qa_Q$$ in $\cs'(\rn)$ and
$$\lf\|\{t_Q\}_{Q\in\cq^\ast}\r\|_{f_{p(\cdot),q(\cdot)}^{s(\cdot),\phi}(\rn)}
\ls\|f\|_{F_{p(\cdot),q(\cdot)}^{s(\cdot),\phi}(\rn)},$$
where, for each $Q\in\cq^\ast$, $a_Q$ is a ($K,\,L$)-smooth atom
supported near $Q$ of $F_{p(\cdot),q(\cdot)}^{s(\cdot),\phi}(\rn)$
with $K\in (s_++\log_2\wz c_1,\fz)$ and $L$ is as in \eqref{md7}.
Let
$$A:=\lf\{Q\in\cq^\ast:\
3\overline{Q}\cap\{(\wz{x},x_n)\in\rr^{n-1}\times \rr:\ x_n=0\}\neq\emptyset\r\},$$
where $\ov{Q}$ denotes the closure of $Q$ in $\rn$.
Since ${\rm supp}\,a_Q\subset 3Q$
for all $Q\in\cq^\ast$,
it follows that $a_Q(\wz{\cdot},0)=0$ if $Q\notin A$.
Observe that, if $Q_{jk}\in A$ with $j\in\zz_+$ and $k:=(k_1,\dots,k_n)\in\zz^n$,
then $|k_n|\le2$.
Therefore,
$$\sum_{j=0}^\fz\sum_{k\in\zz^n}t_{Q_{jk}}a_{Q_{jk}}(\wz{\cdot},0)
=\sum_{j=0}^\fz\sum_{{k\in\zz^n,\,|k_n|\le2}}t_{Q_{jk}}a_{Q_{jk}}(\wz{\cdot},0).$$
Thus, to complete the proof of Lemma \ref{l-welld},
it suffices to show that
\begin{equation}\label{welld-z}
\lim_{{N\to\fz,\,\Lambda\to\fz}}\sum_{j=0}^N
\sum_{\gfz{k\in\zz^n,\,|k_n|\le2}{|k|\le\Lambda}}t_{Q_{jk}}a_{Q_{jk}}(\wz{\cdot},0)
\end{equation}
exists in $\cs'(\rr^{n-1})$.
By \eqref{trace1-x}, we see that
\begin{eqnarray*}
&&s_--\frac n{p_-}+\frac{\min\{1,p_-\}}{p_-}(n-1)\\
&&\hs=s_--\frac 1{p_-}-\frac{n-1}{p_-}\lf(1-\min\{1,p_-\}\r)\\
&&\hs\ge s_--\frac 1{p_-}-\frac{n-1}{\min\{1,p_-\}}\lf(1-\min\{1,p_-\}\r)\\
&&\hs=s_--\frac 1{p_-}-(n-1)\lf[\frac 1{\min\{1,p_-\}}-1\r]>0,
\end{eqnarray*}
which implies that there exists $r\in(0,\min\{1,p_-\})$ such that
$$s_--\frac n{p_-}+\frac r{p_-}(n-1)>0.$$
Let $\wz p(\cdot)$ and $\wz s(\cdot)$ be as in the proof of Theorem \ref{t-md}.
Then $\wz s_--\frac r{p_-}>0$.
For all $j\in\zz_+$ and $k\in\zz^n$, let
$$\Delta_{k,j}:=\nrr\times[k_n2^{-j},(k_n+1)2^{-j}).$$
Then, by the smoothness condition (A3),
we know that, for all $h\in\cs(\nrr)$ and $j\in\zz_+$,
\begin{eqnarray*}
{\rm I}&:=&\lf|\lf\la\sum_{\gfz{k\in\zz^n,\,|k_n|\le2}{|k|\le\Lambda}}
t_{Q_{jk}}a_{Q_{jk}}(\wz{\cdot},0),h(\wz{\cdot})\r\ra\r|\\
&\ls&2^{jn/2}\sum_{\gfz{k\in\zz^n,\,|k_n|\le2}{|k|\le\Lambda}}
|t_{Q_{jk}}|\int_{\nrr}\frac{(1+|\wz{y}|)^{-\delta}}
{(1+2^j|(\wz{y},0)-x_{Q_{jk}}|)^R}\,d\wz{y}\\
&\sim&2^{jn/2+j}\sum_{\gfz{k\in\zz^n,\,|k_n|\le2}{|k|\le\Lambda}}
|t_{Q_{jk}}|\int_{\Delta_{k,j}}\frac{(1+|\wz{y}|)^{-\delta}}
{(1+2^j|(\wz{y},0)-x_{Q_{jk}}|)^R}\,d\wz{y}dy_n\\
&\ls&2^{jn/2+j}\sum_{\gfz{k\in\zz^n,\,|k_n|\le2}{|k|\le\Lambda}}
|t_{Q_{jk}}|\int_{\rn}\frac{(1+|y|)^{-\delta}\chi_{\Delta_{k,j}}(y)}
{(1+2^j|y-x_{Q_{jk}}|)^R}\,dy,
\end{eqnarray*}
where $R\in(0,\fz)$ is chosen large enough,
and $\delta\in(0,\fz)$ will be determined later.
By an argument similar to that used in the proof of Theorem \ref{t-md},
we find that
\begin{eqnarray}\label{welld-y}
{\rm I}
\ls2^{-j(\wz s_--1)}\|t\|_{f_{p(\cdot),q(\cdot)}^{s(\cdot),\phi}(\rn)}
\lf\|\frac{\chi_{\Delta_{j}}}{(1+|\cdot|)^{\delta-\delta_0}}
\r\|_{L^{(\wz p(\cdot))^\ast}(\rn)},
\end{eqnarray}
where $\Delta_j:=\nrr\times[-2^{-j+1},2^{-j+1}]$.
On the other hand, by choosing $\delta$ large enough, we see that
\begin{eqnarray*}
&&\int_{\rn}\lf[\frac{\chi_{\Delta_j}(y)/(1+|y|)^{-\delta+\delta_0}}
{2^{-j/(\wz p(\cdot)^\ast)_+}}\r]^{(\wz p(y))^\ast}\,dy\\
&&\hs\ls 2^j\int_{\nrr\times[-2^{-j+1},2^{-j+1}]}
\lf[\frac1{(1+|\wz{y}|)^{\delta-\delta_0}}\r]^{(\wz p(\cdot)^\ast)_-}\,d\wz{y}dy_n\\
&&\hs\ls\int_{\nrr}
\lf[\frac1{(1+|\wz{y}|)^{\delta-\delta_0}}\r]^{(\wz p(\cdot)^\ast)_-}\,d\wz{y}\ls1,
\end{eqnarray*}
which, together with Remark \ref{r-vlpp}(ii), implies that
\begin{eqnarray*}
\lf\|\frac{\chi_{\Delta_{j}}}{(1+|\cdot|)^{\delta-\delta_0}}
\r\|_{L^{(\wz p(\cdot))^\ast}(\rn)}
\ls2^{-j/(\wz p(\cdot)^\ast)_+}\sim 2^{-j(1-\frac r{p_-})}.
\end{eqnarray*}
From this and \eqref{welld-y}, we deduce that
${\rm I}\ls 2^{-j(\wz s_--\frac r{p_-})}
\|t\|_{f_{p(\cdot),q(\cdot)}^{s(\cdot),\phi}(\rn)}$,
which, combined with the fact that $\wz s_--\frac r{p_-}>0$,
implies that \eqref{welld-z}
converges in $\cs'(\rr^{n-1})$.
This finishes the proof of Lemma \ref{l-welld}.
\end{proof}

Next we prove Theorem \ref{t-trace1} by beginning with
several technical lemmas.

\begin{lemma}\label{l-emdedding}
Let $p_i$, $q_i$, $s_i$ and $\phi$ be as in Definition \ref{d-tl}
with $p,\ q,\ s$ replaced by $p_i,\ q_i,\ s_i$, respectively,
where $i\in\{1,\,2\}$.
If $s_1\le s_2$, $p_1\le p_2$, $(p_1)_\fz=(p_2)_\fz$ and
$q_1\ge q_2$, then
$$F_{p_2(\cdot), q_2(\cdot)}^{s_2(\cdot),\phi}(\rn)
\hookrightarrow F_{p_1(\cdot), q_1(\cdot)}^{s_1(\cdot),\phi}(\rn).$$
\end{lemma}
\begin{proof}
By \cite[Proposition 6.5]{dhr09}, we find that
$$L^{p_2(\cdot)}(\rn)\hookrightarrow L^{p_1(\cdot)}(\rn).$$
From this, $s_1\le s_2$ and \eqref{simple-ineq}, we can easily deduce the
desired conclusion,
the details being
omitted. This finishes the proof of Lemma \ref{l-emdedding}.
\end{proof}

\begin{lemma}\label{l-trac1}
Let $p_i$, $q_i$, $s_i$ and $\phi$ be as in Theorem \ref{t-trace1}
with $p,\ q,\ s$ replaced by $p_i,\ q_i,\ s_i$, where $i\in\{1,\,2\}$.
Assume that $s_1=s_2$ and $p_1=p_2$ on $\rr_+^n$ or $\rr_-^n$, and
that $s_1\le s_2$ and $p_1\le p_2$. If
\begin{equation}\label{trac1-x}
(s_2)_--\frac1{(p_2)_-}-(n-1)\lf[\frac{1}{\min\{1,(p_2)_-\}-1}\r]>0,
\end{equation}
then
$$\mathop\mathrm{Tr} F_{p_1(\cdot),q_1(\cdot)}^{s_1(\cdot),\phi}(\rn)
= \mathop\mathrm{Tr} F_{p_2(\cdot),q_2(\cdot)}^{s_2(\cdot),\phi}(\rn);$$
moreover, if $q(\cdot)$ is as in Theorem \ref{t-trace1}, then
$$\mathop\mathrm{Tr} F_{p_1(\cdot),q_1(\cdot)}^{s_1(\cdot),\phi}(\rn)
= \mathop\mathrm{Tr} F_{p_1(\cdot),q(\cdot)}^{s_1(\cdot),\phi}(\rn).$$
\end{lemma}

To prove Lemma \ref{l-trac1}, we need the following conclusion.

\begin{proposition}\label{p-ec1}
Let $p$, $q$, $s$, $\phi$ be as in Definition \ref{d-tl} and $\delta\in(0,1)$.
 Suppose
that, for each $Q\in\cq^\ast$,
$E_Q\subset 3Q$ is a measurable set with $|E_Q|\ge \delta |Q|$.
Then, for all $t:=\{t_Q\}_{Q\in\cq^\ast}\subset\cc$, $t\in \tlve$ if and
only if $\|t\|_{\wz\tlve}<\fz$,
where
\begin{eqnarray*}
&&\|t\|_{\wz\tlve}\\
&&\hs:=\sup_{P\in\cq}\frac1{\phi(P)}
\lf\|\lf\{\sum_{j=(j_P\vee 0)}^\fz\sum_{\ell(Q)=2^{-j}}\lf[2^{js(\cdot)}|t_Q||Q|
^{-\frac12}\chi_{E_Q}\r]^{q(\cdot)}\r\}^{\frac1{q(\cdot)}}\r\|_{L^{p(\cdot)}(P)}.
\end{eqnarray*}
\end{proposition}
\begin{proof}
We first suppose that $\|t\|_{\wz\tlve}<\fz$ and show that $t\in\tlve$.
Notice that, for all $m\in(n,\fz)$, $Q\in\cq^\ast$ and $x\in Q$,
$$\chi_Q(x)\ls\eta_{j_Q,m+C_{\log}(s)}\ast \chi_{E_Q}(x).$$ From this, Lemmas
\ref{l-estimate2} and \ref{l-eta}, we deduce that
\begin{eqnarray*}
&&\|t\|_{\tlve}\\
&&\hs=\sup_{P\in\cq}\frac1{\phi(P)}
\lf\|\lf\{\sum_{j=(j_P\vee 0)}^\fz\sum_{\ell(Q)=2^{-j}}\lf[2^{jrs(\cdot)}
|t_Q|^r|Q|
^{-\frac r2}\chi_{Q}\r]^{\frac{q(\cdot)}r}
\r\}^{\frac1{q(\cdot)}}\r\|_{L^{p(\cdot)}(P)}\\
&&\hs\ls\sup_{P\in\cq}\frac1{\phi(P)}
\lf\|\lf\{\sum_{j=(j_P\vee 0)}^\fz\lf[\eta_{j,m}\ast\lf(
2^{jrs(\cdot)}\sum_{\gfz{Q\subset P}{\ell(Q)=2^{-j}}}|t_Q|^r
\frac{\chi_{E_{Q}}}{|Q|^{\frac r2}}\r)\r]^\frac{q(\cdot)}{r}
\r\}^{\frac1{q(\cdot)}}
\r\|_{L^{p(\cdot)}(P)}\\
&&\hs\ls\sup_{P\in\cq}\frac1{\phi(P)}
\lf\|\lf\{\sum_{j=(j_P\vee 0)}^\fz\lf[2^{js(\cdot)}
\sum_{\gfz{Q\subset P}{\ell(Q)=2^{-j}}}|t_Q||Q|
^{-\frac12}\chi_{E_{Q}}\r]^{q(\cdot)}\r\}^{\frac1{q(\cdot)}}
\r\|_{L^{p(\cdot)}(P)}\\
&&\hs\sim\|t\|_{\wz\tlve},
\end{eqnarray*}
which implies that $t\in\tlve$.

Conversely, by an argument similar to the above
and the fact that, for all $m\in(n,\fz)$, $Q\in\cq^\ast$ and $x\in E_Q$,
$\chi_{E_Q}(x)\ls \eta_{j_Q,m+C_{\log(s)}}\ast \chi_{Q}(x)$, we conclude that,
for all $t\in\tlve$,
$\|t\|_{\wz \tlve}\ls \|t\|_{\tlve}$.
This finishes the proof of Proposition \ref{p-ec1}.
\end{proof}

\begin{proof}[Proof of Lemma \ref{l-trac1}]
From Remark \ref{r-log}(i) and the condition that $p_1=p_2$ on
$\rr_+^n$ or $\rr_-^n$,
we deduce that $(p_1)_\fz=(p_2)_\fz$.
Let $r_0:=\min\{(q_2)_-,(q_1)_-\}$ and
$r_1:=\max\{(q_2)_+,(q_1)_+\}$.
Then, by Lemma \ref{l-emdedding} and \eqref{simple-ineq},
we see that
\begin{equation}\label{trac1-y}
F_{p_2(\cdot),r_0}^{s_2(\cdot),\phi}(\rn)
\hookrightarrow F_{p_1(\cdot),q_1(\cdot)}^{s_1(\cdot),\phi}(\rn)
\hookrightarrow F_{p_1(\cdot),r_1}^{s_1(\cdot),\phi}(\rn)
\end{equation}
and
\begin{equation}\label{trac1-z}
F_{p_2(\cdot),r_0}^{s_2(\cdot),\phi}(\rn)
\hookrightarrow F_{p_2(\cdot),q_2(\cdot)}^{s_2(\cdot),\phi}(\rn)
\hookrightarrow F_{p_1(\cdot),r_1}^{s_1(\cdot),\phi}(\rn).
\end{equation}
By Lemma \ref{l-welld} and an argument similar to that used in the proof
of \cite[Lemma 7.2]{dhr09}, with \cite[Theorem 3.8 and Lemma 7.1]{dhr09}
replaced by Theorem \ref{t-md} and  Proposition \ref{p-ec1},
we conclude that, for all $f\in F_{p_1(\cdot),r_1}^{s_1(\cdot),\phi}(\rn)$,
$\mathop\mathrm{Tr}(f)$ exists in $\cs'(\nrr)$ and
$
\mathop\mathrm{Tr} F_{p_1(\cdot),r_1}^{s_1(\cdot),\phi}(\rn)
\subset \mathop\mathrm{Tr} F_{p_2(\cdot),r_0}^{s_2(\cdot),\phi}(\rn).
$
From this, \eqref{trac1-y} and \eqref{trac1-z}, we deduce that
\begin{eqnarray*}
\mathop\mathrm{Tr} F_{p_1(\cdot),q_1(\cdot)}^{s_1(\cdot),\phi}(\rn)
&&\subset \mathop\mathrm{Tr} F_{p_1(\cdot),r_1}^{s_1(\cdot),\phi}(\rn)
\subset\mathop\mathrm{Tr} F_{p_2(\cdot),r_0}^{s_2(\cdot),\phi}(\rn)
\subset \mathop\mathrm{Tr} F_{p_2(\cdot),q_2(\cdot)}^{s_2(\cdot),\phi}(\rn)\\
&&\subset \mathop\mathrm{Tr} F_{p_1(\cdot),r_1}^{s_1(\cdot),\phi}(\rn)
\subset\mathop\mathrm{Tr} F_{p_2(\cdot),r_0}^{s_2(\cdot),\phi}(\rn)
\subset\mathop\mathrm{Tr} F_{p_1(\cdot),q_1(\cdot)}^{s_1(\cdot),\phi}(\rn),
\end{eqnarray*}
which completes the proof of Lemma \ref{l-trac1}.
\end{proof}

\begin{remark}\label{r-trace}
By the proof of Lemma \ref{l-trac1}, we see that
the condition \eqref{trac1-x} is only used to
ensure that $\mathop\mathrm{Tr}F_{p_2(\cdot),q_2(\cdot)}^{s_2(\cdot),\phi}(\rn)$
exists in $\cs'(\nrr)$. Thus, by an argument similar to that used in the proof of
Lemma \ref{l-trac1}, we have the following conclusion, the details being
omitted. Under the same assumption as in Lemma \ref{l-trac1},
if, for all $f\in F_{p_2(\cdot),q_2(\cdot)}^{s_2(\cdot),\phi}(\rn)$,
the trace of $f$ defined as in \eqref{d-trace} exists in $\cs'(\nrr)$,
then, for all $g\in F_{p_1(\cdot),q_1(\cdot)}^{s_1(\cdot),\phi}(\rn)$,
the trace of $g$ defined as in \eqref{d-trace} also exists in $\cs'(\nrr)$;
moreover,
$$\mathop\mathrm{Tr}F_{p_1(\cdot),q_1(\cdot)}^{s_1(\cdot),\phi}(\rn)
= \mathop\mathrm{Tr}F_{p_2(\cdot),q_2(\cdot)}^{s_2(\cdot),\phi}(\rn).$$
\end{remark}

\begin{lemma}\label{l-trac-ind}
Let $p_i$, $q$, $s_i$ be as in Theorem \ref{t-trace1} with $p$ and $s$ replaced
by $p_i$ and $s_i$, $i\in\{1,2\}$.
Assume that $s_1(x)=s_2(x)$ and $p_1(x)=p_2(x)$ for all $x\in\rr^{n-1}\times\{0\}$.
If \eqref{trace1-x} is satisfied with $(s,p)$ replaced, respectively,
by $(s_1,p_1)$ and $(s_2,p_2)$, then
$$\mathop\mathrm{Tr} F_{p_1(\cdot),q(\cdot)}^{s_1(\cdot),\phi}(\rn)
=\mathop\mathrm{Tr} F_{p_2(\cdot),p_2(\cdot)}^{s_2(\cdot),\phi}(\rn).$$
\end{lemma}
\begin{proof}
For all $x\in\rn$ and $i\in\{1,2\}$, let $\wz s_i(x):=s_i(x)$ if $x\in\rr_-^n$
and $\wz s_i(x):=\min\{s_1(x),s_2(x)\}$ otherwise, and, for all $x\in\rn$,
$\wz s(x):=\min\{s_1(x),s_2(x)\}$.
Similarly, for $i\in\{1,2\}$,
let $\wz p_i(x):=p_i(x)$ if $x\in\rr_-^n$
and $$\wz p_i(x):=\min\{p_1(x),p_2(x)\}$$ otherwise, and, for all $x\in\rn$,
$$\wz p(x):=\min\{p_1(x),p_2(x)\}.$$ Then, by applying Lemma
\ref{l-trac1} and Remark \ref{r-trace},
we conclude that
\begin{eqnarray*}
\mathop\mathrm{Tr}F_{p_1(\cdot),q(\cdot)}^{s_1(\cdot),\phi}(\rn)
&=&\mathop\mathrm{Tr}F_{p_1(\cdot),p_1(\cdot)}^{s_1(\cdot),\phi}(\rn)
=\mathop\mathrm{Tr}F_{\wz p_1(\cdot),\wz p_1(\cdot)}^{\wz s_1(\cdot),\phi}(\rn)\\
&=&\mathop\mathrm{Tr}F_{\wz p(\cdot),\wz p(\cdot)}^{\wz s(\cdot),\phi}(\rn)
=\mathop\mathrm{Tr}F_{\wz p_2(\cdot),\wz p_2(\cdot)}^{\wz s_2(\cdot),\phi}(\rn)
=\mathop\mathrm{Tr}F_{p_2(\cdot),p_2(\cdot)}^{s_2(\cdot),\phi}(\rn),
\end{eqnarray*}
which completes the proof of Lemma \ref{l-trac-ind}.
\end{proof}

In what follows, let
$\cq(\rn):=\cq$ and $\cq^\ast(\rn):=\cq^\ast$. Denote by $\cq(\rr^{n-1})$ the set
of all dyadic cubes of $\rr^{n-1}$ and $\cq^\ast(\rr^{n-1})$ the set
of all dyadic cubes $\wz{Q}$ of $\rr^{n-1}$ with $\ell(\wz{Q})\le1$.
\begin{proof}[Proof of Theorem \ref{t-trace1}]
By Lemma \ref{l-trac-ind}, we may assume that $q=p$ with
$p$ and $s$ independent of the $n$-th coordinate $x_n$ with $|x_n|\le2$.
Indeed, let, for all $(\wz{x},x_n)\in\nrr\times[-2,2]$, $\wz p_0(\wz{x},x_n):=p(\wz{x},0)$.
Then $\wz p_0\in C^{\log}(\nrr\times[-2,2])$. By \cite[Proposition 4.1.7]{dhr11},
we find that $\wz p_0$ has an extension $\wz p\in C^{\log}(\rn)$
with $\wz p_-=(\wz p_0)_-$ and $\wz p_\fz=(\wz p_0)_\fz$.
Define $\wz s$ by setting, for all $(\wz{x},x_n)\in \nrr\times\rr$,
$\wz s(\wz{x},x_n):=s(\wz{x},0)$. Then it is easy to see that
$\wz s\in C_{\loc}^{\log}(\rn)\cap L^\fz(\rn)$. Moreover,
$\wz p$ and $\wz s$ are independent of the $n$-th coordinate
$x_n$ with $|x_n|\le2$,
and satisfy
$$\wz s_--\frac 1{\wz p_-}-(n-1)\lf[\frac1{\min\{1,\wz p_-\}}-1\r]>0.$$
Then, by Lemma \ref{l-trac-ind}, we see that
$$\mathop\mathrm{Tr} \btlve
=\mathop\mathrm{Tr} F_{\wz p(\cdot),\wz p(\cdot)}^{\wz s(\cdot),\phi}(\rn).$$

For notational simplicity, let, for all $\wz{x}\in\rr^{n-1}$,
$$\beta(\wz{x},0):=s(\wz{x},0)-\frac1{p(\wz{x},0)},$$
$$F_{p(\wz{\cdot},0)}^{\beta(\wz{\cdot},0),\wz\phi}(\rr^{n-1})
:=F_{p(\wz{\cdot},0),p(\wz{\cdot},0)}^{\beta(\wz{\cdot},0),\wz\phi}(\rr^{n-1})$$
and
$$f_{p(\wz{\cdot},0)}^{\beta(\wz{\cdot},0),\wz\phi}(\rr^{n-1})
:=f_{p(\wz{\cdot},0),p(\wz{\cdot},0)}^{\beta(\wz{\cdot},0),\wz\phi}(\rr^{n-1}).$$
We finish the proof of Theorem \ref{t-trace1} by two steps.

\emph{Step 1)} We show that, for all
$f\in\btlve$, $\mathop\mathrm{Tr}(f)\in\cs'(\nrr)$ and
\begin{eqnarray}\label{proof1}
\|\mathop\mathrm{Tr}(f)\|_{F_{p(\wz{\cdot},0)}
^{\beta(\wz{\cdot},0),\wz \phi}(\rr^{n-1})}
\ls \|f\|_{F_{p(\cdot),p(\cdot)}^{s(\cdot),\phi}(\rn)}.
\end{eqnarray}
Without loss of generality, we may assume that
$\|f\|_{F_{p(\cdot),p(\cdot)}^{s(\cdot),\phi}(\rn)}=1$.
By Theorem \ref{t-md}, we see that
$$f=\sum_{Q\in\cq^\ast(\rn)}t_Qa_Q$$
in $\cs'(\rn)$,
where, for all $Q\in\cq^\ast$, $a_Q$ is a ($K,\,L$)-smooth atom supported near $Q$ of
$F_{p(\cdot),p(\cdot)}^{s(\cdot),\phi}(\rn)$ with
\begin{equation}\label{4.9x}
K\in(s_++\max\{0,\log_2\wz c_1\},\fz)\ {\rm and}\
L\in\lf(\frac n{\min\{1,p_-\}}-n-s_-,\fz\r)
\end{equation}
 and
$t:=\{t_Q\}_{Q\in\cq^\ast(\rn)}\in f_{p(\cdot),p(\cdot)}^{s(\cdot),\phi}(\rn)$,
 which
can be chosen such that
\begin{equation}\label{trace-y}
\|t\|_{f_{p(\cdot),p(\cdot)}^{s(\cdot),\phi}(\rn)}
\ls\|f\|_{F_{p(\cdot),p(\cdot)}^{s(\cdot),\phi}(\rn)}.
\end{equation}
Since ${\rm supp}\,a_Q\subset 3Q$, it follows that, if $i\notin\{0,1,2\}$,
then $$a_{\wz{Q}\times[(i-1)\ell(\wz{Q}),i\ell(\wz{Q}))}(\wz{\cdot},0)=0,$$
which implies that
$\mathop\mathrm{Tr}(f)$ can be rewritten as
\begin{equation*}
\sum_{i=0}^2\sum_{\wz{Q}\in\cq^\ast(\rr^{n-1})}
t_{{\wz{Q}\times[(i-1)\ell(\wz{Q}),i\ell(\wz{Q}))}}a_{\wz{Q}\times[(i-1)\ell(\wz{Q}),
i\ell(\wz{Q}))}(\wz{\cdot},0).
\end{equation*}
Therefore, to show \eqref{proof1}, by Theorem \ref{t-md} again,
it can be reduced to prove that each
$$b_{\wz{Q}}^{(i)}:=[\ell(\wz{Q})]^\frac12a_{\wz{Q}\times[(i-1)\ell(\wz{Q}),i\ell(\wz{Q}))}$$
is a ($\wz K,\,\wz L$)-smooth atom of supported near $\wz{Q}$ of $F_{p(\wz{\cdot},0)}
^{\beta(\wz{\cdot},0),\wz \phi}(\rr^{n-1})$ with
\begin{equation}\label{trace-x1}
\wz K\in\lf((s(\wz{\cdot},0))_++\max\{0,\log_2\wz c_1\},\fz\r),
\end{equation}
\begin{equation}\label{trace-xx}
\wz L\in \lf(\frac{n-1}{\min\{1,(p(\wz{\cdot},0))_-\}}-(n-1)-(s(\wz{\cdot},0))_-,\fz\r)
\end{equation}
and
\begin{equation}\label{trace6}
\lf\|\lf\{\lz_{\wz{Q}}^{(i)}\r\}
_{\wz{Q}\in\cq^\ast(\rr^{n-1})}\r\|_
{f_{p(\wz{\cdot},0)}^{\beta(\wz{\cdot},0),\wz \phi}(\rr^{n-1})}<\fz,
\end{equation}
where, for all $\wz{Q}\in\cq^\ast(\rr^{n-1})$,
$$\lz_{\wz{Q}}^{(i)}:=[\ell(\wz{Q})]^{-\frac12}t_{\wz{Q}\times[(i-1)\ell(\wz{Q}),i\ell(\wz{Q}))}.$$

By \eqref{trace1-x}, we see that
$$\frac{n-1}{\min\{1,(p(\wz{\cdot},0))_-\}}-(n-1)-(s(\wz{\cdot},0))_-<0$$
and then, by Remark \ref{r-atom}(i), we know that the vanishing moment for
$(\wz K,\,\wz L)$-smooth atoms of
$F_{p(\wz{\cdot},0)}^{\beta(\wz{\cdot},0),\wz \phi}(\rr^{n-1})$ is avoid.
Since ${\rm supp}\,a_Q\subset 3Q$, $\wz K\le K$ and, for all $\alpha\in\zz_+^n$,
with $|\alpha|\le K$, and all $x\in\rn$,
$|D^\alpha a_Q(x)|\le 2^{(|\alpha|+n/2)j}$,
it follows that, for $i\in\{0,1,2\}$, $\wz\alpha\in\zz_+^n$, with
$|\wz\alpha|\le \wz K$, and all $x\in\rn$,
$$|D^{\wz\alpha} b_Q^{(i)}(x)|
\le 2^{(|\wz \alpha|+n/2)j}$$
and $\supp b_Q^{(i)}\subset 3\wz{Q}$.
Thus, for $i\in\{0,1,2\}$, $b_{\wz{Q}}^{(i)}$ is a ($\wz K,\,\wz L$)-smooth
atom supported near $\wz{Q}$ of
$F_{p(\wz{\cdot},0)}^{\beta(\wz{\cdot},0),\wz \phi}(\rr^{n-1})$ with
($\wz K,\,\wz L$) as in \eqref{trace-x1} and \eqref{trace-xx}.

 Let $\lz^{(i)}:=\{\lz_{\wz{Q}}^{(i)}\}_{\wz{Q}\in\cq^\ast(\rr^{n-1})}$, where $i\in\{0,1,2\}$.
Next we show that, for any given dyadic cube $\wz{P}\subset\rr^{n-1}$,
\begin{eqnarray*}
\frac1{\wz\phi(\wz{P})}\lf\|\lf\{\sum_{j=(j_{\wz{P}}\vee0)}^\fz
\sum_{\gfz{\wz{Q}\in\cq^\ast(\rr^{n-1})}{\ell(\wz{Q})=2^{-j}}}
\lf[2^{j(\beta(\wz{\cdot},0))}|\lz_{\wz{Q}}^{(i)}||\wz{Q}|^{-\frac12}\chi_{\wz{Q}}
\r]^{p(\wz{\cdot},0)}\r\}^{\frac1{p(\wz{\cdot},0)}}\r\|_{L^{p(\wz{\cdot},0)}(\wz{P})}
\end{eqnarray*}
is finite. By $\|f\|_{F_{p(\cdot),p(\cdot)}^{s(\cdot),\phi}(\rn)}= 1$
and \eqref{trace-y},
we see that there exists a positive constant $C_0$ such that,
for all $P\in\cq(\rn)$,
\begin{eqnarray*}
\frac1{\phi(P)}\lf\|\lf\{\sum_{j=(j_P\vee0)}^\fz
\sum_{\gfz{Q\in\cq^\ast(\rn)}{\ell(Q)=2^{-j}}}
\lf[2^{js(\cdot)}|t_Q||Q|^{-\frac12}\chi_Q\r]^{p(\cdot)}\r\}^{\frac1{p(\cdot)}}
\r\|_{L^{p(\cdot)}(P)}\le C_0,
\end{eqnarray*}
which, together with Remark \ref{r-vlpp}(ii), implies that,
for all $P\in\cq(\rn)$,
\begin{eqnarray}\label{trace5}
\int_\rn\lf\{\sum_{j=(j_P\vee0)}^\fz
\sum_{\gfz{Q\in\cq^\ast(\rn)}{\ell(Q)=2^{-j}}}
\lf[2^{js(\cdot)}|t_Q||Q|^{-\frac12}\frac{\chi_P}{C_0\phi(P)}
\chi_Q\r]^{p(\cdot)}\r\}\,dx\le1.
\end{eqnarray}
On the other hand, for all dyadic cube $\wz{P}\in\cq(\rr^{n-1})$,
we have
\begin{eqnarray*}
{\rm I}(\wz{P}):=&&\int_{\rr^{n-1}}\sum_{j=(j_{\wz{P}}\vee 0)}^\fz
\sum_{\gfz{\wz{Q}\in\cq(\rr^{n-1})}{\ell(\wz{Q})=2^{-j}}}
\lf[2^{j\beta(\wz{x},0)}\frac{|\lz_{\wz{Q}}^{(i)}|}{|\wz{Q}|^{\frac12}}
\frac{\chi_{\wz{P}}(\wz{x})\chi_{\wz{Q}}(\wz{x})}{C_0\wz\phi(\wz{P})}
\r]^{p(\wz{x},0)}\,d\wz{x}\\
=&&\sum_{j=(j_{\wz{P}}\vee 0)}^\fz
\sum_{\gfz{\wz{Q}\in\cq(\rr^{n-1})}{\ell(\wz{Q})=2^{-j}}}
2^{-j}\int_{\wz{Q}}\lf[2^{js(\wz{x},0)}|\lz_{\wz{Q}}^{(i)}||\wz{Q}|^{-\frac12}\frac{\chi_{\wz{P}}(\wz{x})}
{C_0\wz \phi(\wz{P})}\r]^{p(\wz{x},0)}\,d\wz{x}\\
\sim&&\sum_{j=(j_{\wz{P}}\vee 0)}^\fz
\sum_{\gfz{\wz{Q}\in\cq(\rr^{n-1})}{\ell(\wz{Q})=2^{-j}}}
\int_{\wh{\wz{Q}}_i}
\lf[2^{js(\wz{x},0)}\frac{|\lz_{\wz{Q}}^{(i)}|}{|\wz{Q}|^{\frac12}}\frac{\chi_{\wz{P}}(\wz{x})}
{C_0\wz \phi(\wz{P})}
\r]^{p(\wz{x},0)}\,d\wz{x}dx_n\\
\ls&&\sum_{j=(j_{\wz{P}}\vee 0)}^\fz
\int_{2(\wz{P}\times[0,\ell(\wz{P})))}
\sum_{\gfz{\wz{Q}\in\cq(\rr^{n-1})}{\ell(\wz{Q})=2^{-j}}}
\lf[2^{js(\wz{x},0)}|\lz_{\wz{Q}}^{(i)}||\wz{Q}|^{-\frac12}\r.\\
&&\times\lf.\lf[C_0\wz \phi(\wz{P})\r]^{-1}
\chi_{\wh{\wz{Q}}_i}(\wz{x},x_n)\r]^{p(\wz{x},0)}\,d\wz{x}dx_n,
\end{eqnarray*}
where
$$\wh{\wz{Q}}_i:=\wz{Q}\times\lf[\frac{(2i-1)\ell(\wz{Q})}{2},i\ell(\wz{Q})\r),$$
which, combined with the fact that
$\{\wh{\wz{Q}}_i\}_{\wz{Q}\in\cq^\ast(\rr^{n-1})}$
are disjoint each other, \eqref{trace5} and the condition \textbf{(S1)}
of $\phi$, implies that, for all
$\wz{P}\in\cq(\rr^{n-1})$,
\begin{eqnarray*}
{\rm I}(\wz{P})
&\ls&\int_{2(\wz{P}\times[0,\ell(\wz{P}))}
\lf\{\sum_{j=(j_{\wz{P}}\vee 0)}^\fz
\lf[\sum_{\gfz{\wz{Q}\in\cq(\rr^{n-1})}{\ell(\wz{Q})=2^{-j}}}
2^{js(\wz{x},0)}|\lz_{\wz{Q}}^{(i)}||\wz{Q}|^{-\frac12}\r.\r.\\
&&\hs\hs\hs\times\lf.\frac{1}{C_0\wz \phi(\wz{P})}
\chi_{\wz{Q}\times[\frac{(2i-1)\ell(\wz{Q})}{2},i\ell(\wz{Q}))}
\Bigg]^{p(x)}\r\}^{\frac{p(\wz{x},0)}{p(x)}}\,d\wz{x}dx_n\\
&\ls&\int_{2(\wz{P}\times[0,\ell(\wz{P}))}
\lf\{\sum_{j=(j_{\wz{P}}\vee 0)}^\fz
\lf[\sum_{\gfz{\wz{Q}\in\cq(\rr^{n-1})}{\ell(\wz{Q})=2^{-j}}}
2^{js(x)}|t_{\wz{Q}\times[(i-1)\ell(\wz{Q}),i\ell(\wz{Q}))}|\r.\r.\\
&&\hs\hs\hs\times\lf.\frac{|\ell(\wz{Q})|^{-\frac n2}}
{C_0\phi(\wz{P}\times[0,\ell(\wz{P})))}
\chi_{\wz{Q}\times[\frac{(2i-1)\ell(\wz{Q})}{2},i\ell(\wz{Q}))}
\Bigg]^{p(x)}\r\}\,dx\ls1,
\end{eqnarray*}
where we used the fact that $p(\wz{x},0)=p(\wz{x},x_n)$ for all $(\wz{x},x_n)\in\rn$ with
$|x_n|\le2$ in the last inequality.
By this and Remark \ref{r-vlpp}(ii), we conclude that,
for all $\wz{P}\in\cq(\rr^{n-1})$,
\begin{eqnarray*}
&&\frac1{\wz \phi(\wz{P})}\lf\|\lf\{
\sum_{j=(j_{\wz{P}}\vee 0)}^\fz\lf[\sum_{\gfz{\wz{Q}
\in\cq^\ast(\rr^{n-1})}{\ell(\wz{Q})=2^{-j}}}
2^{j\beta(\wz{\cdot},0)}
|\lz_{\wz{Q}}^{(i)}||\wz{Q}|^{-\frac12}\chi_{\wz{Q}}\r]^{p(\wz{\cdot},0)}\r\}
^{\frac1{p(\wz{\cdot},0)}}\r\|_{L^{p(\wz{\cdot},0)}(\wz{P})}
\end{eqnarray*}
is finite, which implies that
$
\|\lz^{(i)}\|_{f_{p(\wz{\cdot},0)}^{\beta(\wz{\cdot},0),\wz \phi}(\rr^{n-1})}
\ls1,
$
namely, \eqref{trace6} holds true. Therefore,
$$\|\mathop\mathrm{Tr}(f)\|_{F_{p(\wz{\cdot},0)}^{\beta(\wz{\cdot},0),\wz \phi}(\rr^{n-1})}
\ls\sum_{i=0}^2\|\lz^{(i)}\|_{f_{p(\wz{\cdot},0)}^{\beta(\wz{\cdot},0),\wz \phi}(\rr^{n-1})}
\ls\|f\|_{F_{p(\cdot),p(\cdot)}^{s(\cdot),\phi}(\rn)}.$$

\emph{Step 2)} We prove that the operator Tr is surjective.
Let $f\in F_{p(\wz{\cdot},0)}^{\beta(\wz{\cdot},0),
\wz \phi}(\rr^{n-1})$. Then,
by Theorem \ref{t-md}, we find that there exist a sequence
$$\lz:=\{\lz_{\wz{Q}}\}_{\wz{Q}\in\cq^\ast(\rr^{n-1})}\subset\cc$$
and a sequence $\{a_{\wz{Q}}\}_{\wz{Q}\in\cq^\ast(\rr^{n-1})}$
of ($K,\,L$)-smooth atoms of $F_{p(\wz{\cdot},0)}^{\beta(\wz{\cdot},0),\wz \phi}(\rr^{n-1})$
with $K$ and $L$ satisfying \eqref{4.9x}
such that
$f=\sum_{\wz{Q}\in\cq^\ast(\rr^{n-1})}\lz_{\wz{Q}}a_{\wz{Q}}$
converges in $\cs'(\rr^{n-1})$ and
\begin{equation}\label{trace-z}
\|\lz\|_{f_{p(\wz{\cdot},0)}^{\beta(\wz{\cdot},0),\wz \phi}(\rr^{n-1})}
\ls\|f\|_{F_{p(\wz{\cdot},0)}^{\beta(\wz{\cdot},0),\wz \phi}(\rr^{n-1})};
\end{equation}
moreover, for all $\wz{P}\in\cq(\rr^{n-1})$,
\begin{eqnarray}\label{trace7}
\frac1{\phi(\wz{P})}
\int_{\wz{P}}\sum_{j=(j_{\wz{P}}\vee0)}^\fz\sum_{\gfz{\wz{Q}\in\cq^\ast(\rr^{n-1})}
{\gfz{\wz{Q}\subset \wz{P}}{\ell(\wz{Q})=2^{-j}}}}
\lf[\frac{2^{j\beta(\wz{x},0)}
|\lz_{\wz{Q}}|\chi_{\wz{Q}}(\wz{x})}
{|\wz{Q}|^{\frac12}\|\lz\|_{f_{p(\wz{\cdot},0)}^{\beta(\wz{\cdot},0)
,\wz \phi}(\rr^{n-1})}}\r]^{p(\wz{x},0)}\,d\wz{x}\ls1.
\end{eqnarray}
Let $\eta\in C_c^\fz(\rr)$ satisfy ${\rm supp}\,\eta\subset(-\frac12,\frac12)$
and
$\eta(0)=1$. For all $\wz{Q}\in\cq(\rr^{n-1})$ and $\xi\in\rr$, let
$\eta_{\wz{Q}}(\xi):=\eta(2^{-\log_2\ell(\wz{Q})}\xi)$,
$$g:=\sum_{\wz{Q}\in\cq^\ast(\rr^{n-1})}\lz_{\wz{Q}}a_{\wz{Q}}\otimes\eta_{\wz{Q}}
=:\sum_{Q\in\cq^\ast(\rn)}t_Qb_Q,$$
where, for all $Q\in\cq^\ast(\rn)$ and $x:=(\wz{x},x_n)\in\rn$,
$$b_Q(x):=[\ell(\wz{Q})]^{-\frac12}a_{\wz{Q}}\otimes\eta_{\wz{Q}}(x)
:=[\ell(\wz{Q})]^{-\frac12}a_{\wz{Q}}(\wz{x})\eta_{\wz{Q}}(x_n),$$
$t_Q:=[\ell(\wz{Q})]^{1/2}\lz_{\wz{Q}}$ if $Q=\wz{Q}\times [0,\ell(\wz{Q}))$ and $t_Q:=0$
otherwise.

Next we show that $g$ converges in $\cs'(\rn)$ and
$$\|g\|_{F_{p(\cdot),p(\cdot)}^{s(\cdot),\phi}(\rn)}
\ls\|f\|_{F_{p(\wz{\cdot},0)}^{\beta(\wz{\cdot},0),\wz \phi}(\rr^{n-1})}.$$
It is easy to show that each $b_Q$
is a ($K,\,L$)-smooth atom supported near $Q$ of
$F_{p(\cdot),p(\cdot)}^{s(\cdot),\phi}(\rn)$ with $K$ and $L$
as in \eqref{4.9x}.
By Proposition \ref{p-ec1} and the fact that
$\{\wz{Q}\times[\frac12\ell(\wz{Q}),\ell(\wz{Q}))\}_{\wz{Q}\in\cq^\ast(\rr^{n-1})}$ are
disjoint each other, we find that
\begin{eqnarray}\label{trace-x}
&&\|\{t_Q\}_{Q\in\cq^\ast(\rn)}\|
_{f_{p(\cdot),p(\cdot)}^{s(\cdot),\phi}(\rn)}\noz\\
&&\hs=\sup_{P\in\cq(\rn)}\frac1{\phi(P)}
\lf\|\lf\{\sum_{j=(j_P\vee0)}^\fz\lf[\sum_{\wz{Q}\in\Theta_j}2^{js(\cdot)}
|\lz_{\wz{Q}}|[\ell(\wz{Q})]^{1/2}\r.\r.\r.\noz\\
&&\hs\hs\hs\times\lf.\lf.|\wz{Q}\times[0,\ell(\wz{Q}))|^{-1/2}
\chi_{\wz{Q}\times[0,\ell(\wz{Q}))}\Bigg]^{p(\cdot)}
\r\}^{1/{p(\cdot)}}\r\|_{L^{p(\cdot)}(P)}\noz\\
&&\hs\sim\sup_{P\in\cq(\rn)}\frac1{\phi(P)}
\lf\|\sum_{j=(j_P\vee0)}^\fz\sum_{\wz{Q}\in \Theta_j}2^{js(\cdot)}
|\lz_{\wz{Q}}||\wz{Q}|^{-\frac12}\chi_{\wz{Q}\times[\frac12\ell(\wz{Q}),\ell(\wz{Q}))}
\r\|_{L^{p(\cdot)}(P)}\noz\\
&&\hs\sim\sup_{\wz{P}\in\cq(\rr^{n-1})}
\frac1{\wz\phi(\wz{P})}\lf\|\sum_{\gfz{\wz{Q}\in\cq^\ast(\rr^{n-1})}
{\wz{Q}\subset \wz{P}}}|\wz{Q}|^{-\frac{s(\cdot)}{n-1}}|\lz_{\wz{Q}}|\r.\noz\\
&&\hs\hs\hs\times|\wz{Q}|^{-1/2}
\chi_{\wz{Q}\times[\frac12\ell(\wz{Q}),\ell(\wz{Q}))}
\Bigg\|_{L^{p(\cdot)}(\wz{P}\times[0,\ell(\wz{P})))},
\end{eqnarray}
where $\Theta_j:=\{\wz{Q}\in\cq^\ast(\rr^{n-1}):\
\wz{Q}\times[0,\ell(\wz{Q}))\subset P,\ \ell(\wz{Q})=2^{-j}\}$.
On the other hand, let
$\Gamma:=\|\lz\|_{f_{p(\wz{\cdot},0)}^{\beta(\wz{\cdot},0),\wz \phi}(\rr^{n-1})}$.
Then, for all $\wz{P}\in\cq(\rr^{n-1})$, by \eqref{trace7}, we find that
\begin{eqnarray*}
&&\frac1{\wz\phi(\wz{P})}\int_{\wz{P}\times[0,\ell(\wz{P}))}\lf\{
\sum_{\gfz{\wz{Q}\in\cq^\ast(\rr^{n-1})}
{\wz{Q}\subset \wz{P}}}
|\lz_{\wz{Q}}||\wz{Q}|^{-\frac12}\Gamma^{-1}
|\wz{Q}|^{-\frac{s(x)}{n-1}}\chi_{\wz{Q}\times[\frac12\ell(\wz{Q}),\ell(\wz{Q}))}
\r\}^{p(x)}\,dx\\
&&\hs\sim\frac1{\wz\phi(\wz{P})}
\sum_{\gfz{\wz{Q}\in\cq^\ast(\rr^{n-1})}
{\wz{Q}\subset \wz{P}}}\int_{\wz{Q}\times[\frac12\ell(\wz{Q}),\ell(\wz{Q}))}
\lf[|\wz{Q}|^{-\frac{s(x)}{n-1}-\frac12}|\lz_{\wz{Q}}|\Gamma^{-1}\r]^{p(x)}\,dx\\
&&\hs\sim\frac1{\wz\phi(\wz{P})}
\sum_{j=(j_{\wz{P}}\vee0)}^\fz\sum_{\gfz{\wz{Q}\in\cq^\ast(\rr^{n-1})}
{\gfz{\wz{Q}\subset \wz{P}}{\ell(\wz{Q})=2^{-j}}}}
\int_{\wz{Q}}\lf[2^{j\beta(\wz{x},0)}
|\lz_{\wz{Q}}||\wz{Q}|^{-\frac12}\Gamma^{-1}\r]^{p(\wz{x},0)}\,d\wz{x}\ls1,
\end{eqnarray*}
which, together with Remark \ref{r-vlpp}(ii), implies that
\begin{eqnarray*}
&&\frac1{\wz\phi(\wz{P})}\lf\|\sum_{\gfz{\wz{Q}\in\cq^\ast(\rr^{n-1})}
{\wz{Q}\subset \wz{P}}}|\wz{Q}|^{-\frac{s(\cdot)}{n-1}}|\lz_{\wz{Q}}||\wz{Q}|^{-\frac12}
\chi_{\wz{Q}\times[\frac12\ell(\wz{Q}),\ell(\wz{Q}))}
\r\|_{L^{p(\cdot)}(\wz{P}\times[0,\ell(\wz{P})))}\\
&&\hs\ls\|\lz\|_{f_{p(\wz{\cdot},0)}^{\beta(\wz{\cdot},0),\wz \phi}(\rr^{n-1})}.
\end{eqnarray*}
Form this and \eqref{trace-x}, we further deduce that
$$\lf\|\lf\{t_Q\r\}_{Q\in\cq^\ast(\rn)}
\r\|_{f_{p(\cdot),p(\cdot)}^{s(\cdot),\phi}(\rn)}
\ls\|\lz\|_{f_{p(\wz{\cdot},0)}^{\beta(\wz{\cdot},0),\wz \phi}(\rr^{n-1})}.$$
Therefore, by Theorem \ref{t-md}(i) and \eqref{trace-z}, we conclude that
$g=\sum_{Q\in\cq^\ast(\rn)}t_Qb_Q$
converges in $\cs'(\rr^{n})$,
$g\in F_{p(\cdot),p(\cdot)}^{s(\cdot),\phi}(\rn)$ and
$$\|g\|_{F_{p(\cdot),p(\cdot)}^{s(\cdot),\phi}(\rn)}
\ls\|f\|_{F_{p(\wz{\cdot},0)}^{\beta(\wz{\cdot},0),\wz \phi}(\rr^{n-1})};$$
furthermore, $\mathop\mathrm{Tr}(g)=f$ in $\cs'(\nrr)$, which implies that
$$\mathop\mathrm{Tr}:\ F_{p(\cdot),p(\cdot)}^{s(\cdot),\phi}(\rn)
 \rightarrow F_{p(\wz{\cdot},0),p(\wz{\cdot},0)}^{s(\wz{\cdot},0),
\wz \phi}(\rr^{n-1})$$
is surjective and hence completes the proof of Theorem \ref{t-trace1}.
\end{proof}


{\bf Acknowledgement.} The authors would like to express their
deep thanks to referees for their careful reading and many useful
comments which improve the presentation of this article.
This project is supported by the National
Natural Science Foundation of China
(Grant Nos.~11171027, 11361020 \& 11471042),
the Specialized Research Fund for the Doctoral Program of Higher Education
of China (Grant No. 20120003110003) and the Fundamental Research
Funds for Central Universities of China (Grant Nos.~2012LYB26,
2012CXQT09, 2013YB60 and 2014KJJCA10).

\bibliographystyle{amsplain}

\end{document}